\documentclass[10pt,a4paper]{article}

\usepackage[english]{babel}
\usepackage[T1]{fontenc}
\usepackage{bbm}
\usepackage{adjustbox}
\usepackage{graphicx}
\usepackage{verbatim}
\usepackage{subcaption}
\usepackage{multirow}
\usepackage{amsmath}
\usepackage{amssymb}
\usepackage{booktabs}
\usepackage{siunitx}
\usepackage{amsthm}
  \theoremstyle{plain}
  \newtheorem{theorem}{Theorem}
  \newtheorem{lemma}{Lemma}
  \newtheorem{proposition}{Proposition}
  \newtheorem{corollary}{Corollary}
  \newtheorem{assumption}{Assumption}
  \newtheorem{definition}{Definition}
  
  \newtheorem*{remark}{Remark}
\usepackage{mathrsfs}
\usepackage{mathtools}
\usepackage[left=2.7cm,right=2.7cm,top=3.1cm,bottom=3.1cm]{geometry}
\usepackage[colorlinks=true, allcolors=blue, hypertexnames=false]{hyperref}

\usepackage[round]{natbib}
\usepackage{algorithm}
\usepackage{algpseudocode}

 \allowdisplaybreaks[4]

\begin{document}

\title{Stabilizing Rate of Stochastic Linear Systems with Multiplicative Noise}

\author{Hui Jia\footnote{College of Artificial Intelligence, Nankai University, Tianjin 300350, P.R. China, E-mail: 2120210368@mail.nankai.edu.cn} \and Yuan-Hua Ni\footnote{Corresponding author. College of Artificial Intelligence, Nankai University, Tianjin 300350, P.R. China, E-mail: yhni@nankai.edu.cn}\and Guangchen Wang\footnote{School of Control Science and Engineering, Shandong University, Jinan 250061, P.R. China, E-mail: wguangchen@sdu.edu.cn}}
\date{\today}

\maketitle

\begin{abstract}
	This paper  develops a quantitative framework for analyzing the mean-square exponential stabilization of stochastic linear systems with multiplicative noise, focusing specifically on the optimal stabilizing rate, which characterizes the fastest exponential stabilization achievable under admissible control policies.	
The framework consists of two complementary developments. First, we extend
the norm-based analysis from deterministic switched systems to the stochastic
setting and establish computable upper and lower bounds for the optimal
stabilizing rate. Second, by restricting attention to state-feedback policies, we introduce an optimal control formulation of the feedback-restricted optimal stabilizing rate problem and derive a Bellman-type equation. Since this Bellman-type equation is not directly tractable, we recast it as a nonlinear matrix eigenvalue problem whose valid solutions require strictly positive-definite matrices. 
To overcome the possible absence of such solutions, we introduce
a regularization scheme and develop a Regularized Normalized Value
Iteration (RNVI) algorithm, which in turn generates strictly positive-definite fixed points for a perturbed version of the original nonlinear matrix eigenvalue problem while producing feedback controllers. Evaluating these regularized solutions further yields
certified lower and upper bounds for the feedback-restricted optimal stabilizing rate, providing a constructive procedure for estimating the fastest achievable mean-square decay rate. We also provide a sufficient condition for the certified gap to close and a necessary structural condition satisfied by each regular
	nonvanishing-gap fixed-point sequence. Numerical experiments further
	demonstrate the effectiveness of the proposed framework.

\end{abstract}

\section{Introduction}

In the theoretical framework of stochastic control, stabilizability analysis remains a core topic that permeates both fundamental research and engineering applications; classic theories focus centrally on a pivotal question for uncertain, noisy dynamic systems: can controller design guarantee bounded and asymptotically convergent trajectories?
Classic results (e.g., \cite{kushner1967}) based on Lyapunov methods and Riccati theory provide affirmative answers under appropriate conditions.
However, as modern engineering systems grow increasingly complex, 
merely proving stabilizability is no longer sufficient; the rate of stabilization, meaning how fast the state returns to the desired equilibrium, has become equally critical. 
This quantitative aspect directly influences both safety and performance in practical scenarios.
For instance, in power and energy systems~(\cite{Zhang2019,Weitenberg18}), distributed frequency control has been shown to achieve exponential convergence and to examine how communication disruptions influence the resulting convergence rate, while transient frequency control enforces finite time return to a prescribed safe band with a guaranteed convergence rate, and thus the pace of recovery is treated as part of the design rather than a byproduct.
For robotics and other safety-critical infrastructures~(\cite{Ames19,Nguyen16}), control Lyapunov and control barrier function methods often operate in tandem to enforce both stability and safety, and in particular, exponential barrier formulations embed explicit requirements on how rapidly safety constraints must be restored, 
indicating that the rate of convergence is an inherent component of reliable operation.
Across these domains, stabilization feasibility alone is insufficient and one must also take into account how rapidly stability is achieved.

Over the past decades, the classic theory of stability in deterministic systems is well established and provides the foundation for much of modern control. Lyapunov's second method (\cite{Khalil02}) offers general certificates of asymptotic and exponential stability for nonlinear systems, while for linear time-invariant systems, algebraic Riccati equations and linear quadratic regulator (LQR) design yield constructive stabilizing state-feedback laws (\cite{Anderson71}). The development of linear matrix inequality (LMI) technique (\cite{Boyd94}) unifies and generalizes classic criteria, making stability analysis and controller synthesis numerically tractable. These methods constitute a mature toolbox: Lyapunov theory provides the conceptual basis, Riccati and LQR methods furnish constructive synthesis, and LMIs deliver computational scalability. Yet, all of these are primarily designed to decide whether stabilization is feasible and they rarely provide certified information about the \emph{rate} of stabilization.
When stochastic disturbances are present, the problem becomes significantly more challenging. In particular, systems with \emph{multiplicative noise}, where randomness enters directly through system coefficients, arise naturally in applications ranging from portfolio optimization (\cite{Dombrovskii03}) to systems biology (\cite{yao11}). For such systems, the standard notion is \emph{mean-square exponential
	stability} (ESMS), which requires the second moment of the state to decay
exponentially. Mean-square stability is particularly relevant in stochastic control because it is formulated in terms of the second moment, which measures the expected state energy, or the average squared magnitude of stochastic deviations around the desired equilibrium, and is naturally compatible with quadratic performance criteria and Lyapunov and Riccati analysis; see, e.g., \cite{kushner1967}. This problem has been studied for decades, with early contributions dating back to the late 1960s (\cite{Kleinman69,McLane69,Wonham67}) and the continued development through numerous works \cite{Zhang15,Feng18,Qi18,Yu2018,Li21}. Much of this literature establishes stability conditions via modified algebraic Riccati equations (\cite{Wonham67,Jacques76,Blankenship77}) or LMI formulations (\cite{Boyd94,Elia05}), thereby connecting mean-square stability analysis to convex optimization. In parallel, a frequency-domain viewpoint is initiated by \cite{Willems71}, who establish necessary and sufficient conditions for mean-square stability in single-input single-output systems; this perspective has since been extended to multi-input multi-output cases (\cite{Hinrichsen96,Lu02,Tian17,Bamieh20}). From this standpoint, mean-square stability can be interpreted as a robust stability problem with respect to stochastic uncertainties. These approaches are powerful and constructive, but they remain \emph{qualitative}: they tell us whether stabilization is possible, not how fast it can be achieved.

	The missing rate information is particularly relevant in the
	multiplicative-noise setting. For such systems, the decay of the second
	moment is affected jointly by the coefficient matrices in the deterministic
	part and those in the multiplicative-noise part. Thus, even after
	mean-square stabilizability is known, the fastest achievable
	second-moment decay rate remains a nontrivial quantitative question. This
	quantity is different from the optimal disturbance attenuation level in
	\(H_\infty\) control: the former measures the fastest mean-square decay of
	the state, whereas the latter measures the worst-case input-output
	amplification from external disturbances to controlled outputs.

This quantitative viewpoint is related to the stabilizing rate framework of \cite{Hu17},
where stabilization is quantified both by feasibility and by how fast it can be achieved.
For deterministic switched systems, this framework enables computable upper and lower bounds on the fastest achievable stabilizing rate 
through norm and seminorm based constructions~(\cite{Hu17,HU2024}). 
These results establish a foundation for quantitative stability analysis, 
but they rely on finite-dimensional properties, such as norm equivalence. 
These properties fail in stochastic settings, where the natural state space is infinite-dimensional and none of the finite-dimensional assumptions carry over.  
Consequently, despite the appeal of Hu's quantitative perspective, 
a systematic extension to stochastic settings remains open.

To sum up, several important limitations remain in the existing literature. First, studies on stochastic systems with multiplicative noise provide tractable feasibility conditions, but they seldom address how fast such systems can be stabilized in the mean-square sense.
Second, works on deterministic switched systems do quantify the fastest convergence rate, yet their methods rely critically on finite-dimensional properties that are absent in the stochastic setting. 
Third, a systematic and computable treatment of fastest stabilizing rate for stochastic systems with multiplicative disturbances is still lacking.
To bridge these gaps, this paper develops a systematic framework for quantifying the
\emph{optimal stabilizing rate}, defined as the fastest mean-square exponential
stabilization achievable for stochastic systems with multiplicative noise.
Our contributions are organized around two complementary developments for
	the general and feedback-restricted optimal stabilizing rates. First, we extend the norm-based
	analysis from deterministic switched systems to stochastic systems with
	multiplicative noise. This extension provides a quantitative
	characterization of the stabilizing rate and establishes computable upper
	and lower bounds for the optimal stabilizing rate.
	Second, by restricting attention to state-feedback policies, we formulate
	the feedback-restricted optimal stabilizing rate problem as an optimal control problem with a
	nonlinear cost functional. This formulation leads to a Bellman-type
	equation and then to a nonlinear matrix eigenvalue problem. Since valid
	solutions of this eigenvalue problem require strictly positive-definite
	matrices, we introduce a regularization scheme and develop the Regularized
	Normalized Value Iteration (RNVI) algorithm. The resulting fixed points
	generate feedback controllers and certified lower and upper bounds for the
	feedback-restricted optimal stabilizing rate. We also provide a sufficient condition under
	which these bounds become tight and identify a necessary structural
	condition for nonvanishing certified gaps.
The main contributions are summarized as follows.
\begin{enumerate}
	\item \textbf{Norm-based characterization of the optimal stabilizing rate}
	
	We extend the norm-based framework of~\cite{Hu17} to stochastic systems
	with multiplicative noise and establish a quantitative characterization of
	the optimal stabilizing rate through computable upper and lower bounds.
	
	\item \textbf{Bridging stabilization theory and optimal control}  \\
	\noindent Restricting to state-feedback policies,
	we introduce an \emph{optimal control formulation with a nonlinear cost
		functional} for the feedback-restricted optimal stabilizing rate problem and derive a
	\emph{Bellman-type equation}. Under the conditions of
		Theorem~\ref{thm:bellman-main}, this formulation characterizes the feedback-restricted optimal
		stabilizing rate.
	As this equation is not directly amenable to analysis, it is reduced to a nonlinear matrix eigenvalue problem that requires a strictly positive-definite solution, providing a tractable representation for determining the feedback-restricted optimal stabilizing rate.
	
	\item \textbf{Regularized approximation for nonlinear matrix eigenvalue problems}  
	
	\noindent The nonlinear matrix eigenvalue equation derived from the Bellman-type formulation may fail to admit strictly positive-definite solutions, leading to ill-posedness on the boundary of the semi-definite cone.  
	To address this puzzle, we introduce a regularization parameter, which guarantees the existence of a strictly positive-definite fixed point.  
	More generally, this procedure offers a regularized approximation approach for such nonlinear matrix eigenvalue problems that require positive-definite solutions.  
	The resulting formulation not only restores well-posedness but also offers a constructive pathway to approximate eigen-solutions in settings where direct solutions are otherwise unattainable.
	
		\item \textbf{Certified bounds and gap analysis}
		
		\noindent The regularized fixed points yield certified lower and upper bounds for
		the feedback-restricted optimal stabilizing rate. We further study the gap between these
		bounds. First, we give a sufficient condition under which the certified
		gap closes as the regularization parameter tends to zero. Second, we
		identify a necessary structural condition satisfied by each regular
		nonvanishing-gap fixed-point sequence.

	\item \textbf{Regularized Normalized Value Iteration (RNVI)}  
	
	\noindent To make this regularized approach constructive, we develop the RNVI algorithm. RNVI iteratively computes the positive-definite solutions and the corresponding feedback controllers for a sequence of decreasing regularization parameters.  
	The procedure starts from a relatively large value of the regularization parameter, for which the convergence is numerically reliable, and proceeds gradually through a continuation scheme that uses each computed solution as the initialization for the next stage.  
Along this
		continuation path, the certified lower and upper bounds developed above
		are evaluated, and the tightest bounds among all parameter values are reported as the final certified estimate.

\end{enumerate}

Together, these results provide both theoretical guarantees and computational tools 
for analyzing and quantifying the general and feedback-restricted optimal stabilizing rates.
The remainder of the paper is organized as follows. Section~\ref{sec:problem} introduces the stochastic system model and the general and feedback-restricted optimal stabilizing rates. Section~\ref{sec:norm} extends the norm-based framework of \cite{Hu17} to our stochastic setting.
Section~\ref{sec:ocp-bellman} develops an optimal control formulation within
the feedback policy class and derives a Bellman-type equation, whose strictly
positive definite quadratic solutions are characterized by a nonlinear matrix
eigenvalue equation. Section~\ref{sec:regularization} develops the RNVI algorithm to obtain bounds for the feedback-restricted optimal stabilizing rate. Section~\ref{sec:numerics} reports numerical experiments that illustrate the performance of the proposed method. Section~\ref{sec:conclusion} then concludes this paper.

\textbf{Notations.}  
Let $\mathbb{E}[\cdot]$ denote the mathematical expectation, and 
$\mathbb{E}[\cdot \mid \mathcal{F}]$ the conditional expectation with respect to $\sigma$-algebra $\mathcal{F}$. $\mathbb{R}_{+}$ stands for the set of positive real numbers, and $\mathbb{N}$ represents the set of non-negative integers $\{0, 1, 2, \dots\}$. The symbol $\times$ denotes the Cartesian product of sets, and $\otimes$
	denotes the Kronecker product of matrices.
$\mathbb{S}^n_+$ and $\mathbb{S}^n_{++}$ denote the sets of $n \times n$ symmetric positive semi-definite matrices and symmetric positive-definite matrices, respectively; accordingly, $A \succeq 0$ ($A \succ 0$) indicates that matrix $A$ is symmetric positive semi-definite (symmetric positive-definite). $\lambda_{\min} (A)$ and $\lambda_{\max} (A)$ denote the minimum and maximum eigenvalues of the symmetric matrix $A$, respectively. For a square matrix \(A\), \(\operatorname{spr}(A)\) denotes its spectral radius.
The transpose and inverse of a matrix are denoted by $(\cdot)^\top$ and $(\cdot)^{-1}$, respectively. $\operatorname{Tr}(A)$ denotes the trace of matrix $A$. $I$ denotes the identity matrix with compatible dimensions, and $\sigma(x)$ represents the $\sigma$-algebra generated by the random variable $x$.  
For a matrix $A$, $\|A\|$ and $\|A\|_F$ denote its spectral norm and Frobenius norm, respectively. For two positive functions \(f\) and \(g\), \(f(x)\asymp g(x)\) means that \(c g(x)\le f(x)\le C g(x)\) for all sufficiently small \(x\) and some constants \(0<c\leq C<\infty\). The infimum of a function $f$ is denoted as $f^*$. $x\mapsto f(x)$ indicates the mapping from input $x$ to output $f(x)$. 
For a mapping $f$, $Df(x)[H]$ denotes directional derivative of $f$ at $x$ in direction $H$ and
$Df(x)$ denotes the corresponding linear operator $H\mapsto Df(x)[H]$.
Finally, $L^2(\Omega;\mathbb{R}^m)$ denotes the Hilbert space of $\mathbb{R}^m$-valued square-integrable random variables, equipped with the inner product $\langle x, y \rangle := \mathbb{E}[x^\top y]$ and the induced norm $\big(\mathbb{E}[x^\top x]\big)^{1/2}$.

\section{Problem Formulation}\label{sec:problem}
Consider a discrete-time plant, which evolves according to some discrete-time stochastic difference equation
\begin{align}\label{system}
	x_{k+1}=(A+\bar{A}\omega_k)x_k+(B+\bar{B}\omega_k)u_k,\quad k\in \mathbb{N}.
\end{align}
Here, $x_k\in \mathbb{R}^n$ and $u_k \in \mathbb{R}^{m}$ are  system state and control input, respectively. The matrices
	$A,\bar A\in\mathbb R^{n\times n}$ and $B,\bar B\in\mathbb R^{n\times m}$
	are constant system matrices. The noise sequence $\{\omega_k\}$ consists of independent and identically distributed (i.i.d) Gaussian random variables with zero mean and variance $\sigma^2$. The initial state $x_0\in L^2(\Omega;\mathbb R^n)$ is independent of the noise sequence $\{\omega_k\}$. Here, \(\omega_k\) is interpreted as a common random
	parameter that simultaneously affects the state and input coefficient
	matrices. This formulation is suitable for systems driven by a shared
	stochastic operating condition. More general models with distinct and
	possibly correlated noise processes can be treated by replacing the
	covariance-weighted terms in the associated operator.

Let $\mathcal{F}_k=\sigma(x_0,\omega_0,\dots,\omega_{k-1})$ be the
natural filtration generated by the past information. For \(l\in\mathbb N\) and \(k\in\mathbb N\), define
	\[
	\begin{aligned}
		L^2_{\mathcal F_k}(\Omega;\mathbb R^l)
		:=
		\bigl\{
		\phi_k:\Omega\to\mathbb R^l
		\ \big|\ 
		\phi_k \text{ is } \mathcal F_k\text{-measurable},
		\mathbb E[\phi_k^\top\phi_k]<\infty
		\bigr\}.
	\end{aligned}
	\]%
\noindent The control must satisfy the causality and square-integrability
constraints. Accordingly, we introduce the following admissible control set
\[
\mathcal U
:=
\left\{
u=(u_0,u_1,\dots)
\;\middle|\;
u_k\in L^2_{\mathcal F_k}(\Omega;\mathbb R^m),\ k\in\mathbb N
\right\}.
\]
We next introduce some basic notions of mean-square stability.
\begin{definition}
	The control-free system~\eqref{system} with $u_k\equiv0$ is said to be $\ell^2$-exponentially mean-square stable, if for any initial value $x_0\in L^2(\Omega;\mathbb R^n)$, there exist  constants $\kappa\geq 0$, $\rho \in (0,1)$ such that the solution of (\ref{system}) satisfies
	\begin{align}\label{de-ms}
		\mathbb{E}[x^\top_k x_k]\leq \kappa \rho^{2k}\mathbb{E}[x_0^\top x_0],\quad  \forall k \in \mathbb{N}.
	\end{align}
\end{definition}

\begin{definition}
	System (\ref{system}) is said to be open-loop $\ell^2$-exponentially stabilizable in the mean-square sense, if for any initial value $x_0\in L^2(\Omega;\mathbb R^n)$, there exists $u\in \mathcal{U}$  such that the solution of (\ref{system}) satisfies (\ref{de-ms}).
\end{definition}

\begin{definition}
	System (\ref{system}) is said to be closed-loop $\ell^2$-exponentially stabilizable in the mean-square sense, if there exists $u_k=-Kx_k$  with a constant matrix $K$, such that for any $x_0\in L^2(\Omega;\mathbb R^n)$ the closed-loop system of (\ref{system}) is $\ell^2$-exponentially mean-square stable.
\end{definition}

For linear stochastic system \eqref{system}, 
Proposition~4.1 of \cite{NI201565} shows that open-loop $\ell^2$-stabilizable is equivalent to closed-loop $\ell^2$-stabilizable. Hence, the open-loop $\ell^2$-stabilizability and the closed-loop $\ell^2$-stabilizability will be both called $\ell^2$-stabilizability. 
Inspired by the work of~\cite{Hu17}, where the concept of the \emph{optimal stabilizing rate} is introduced for deterministic switched systems, we extend this concept to the stochastic systems with multiplicative noise. 
This motivates the following definitions.

\begin{definition}\label{def-rho-u}
	For each admissible control policy $u\in\mathcal U$, its 
	\emph{mean-square stabilizing rate} is defined as	
	\begin{align*}	
		&\rho(u)
		:= \inf\Big\{\rho > 0 \;\Big|\; \exists\,\kappa\ge 0 \ \text{such that} \quad \mathbb{E}[x_k^\top x_k]
		\le \kappa\,\rho^{2k}\,\mathbb{E}[x_0^\top x_0], \
		\forall x_0\in L^2(\Omega;\mathbb R^n),\ \forall k\in\mathbb N \Big\},
	\end{align*}
	where $\{x_k\}$ denotes the state trajectory of system~\eqref{system} driven by~$u$.
	We use the convention that the infimum of the empty set is $+\infty$;
	hence $\rho(u)$ takes values in $[0,+\infty]$.
\end{definition}

\begin{remark}
	The constant $\kappa$ in Definition \ref{def-rho-u}
	does not affect the value of $\rho(u)$.
	The reason is that $\kappa$ only compensates for a finite number of initial
	steps, while  $\rho(u)$ characterizes the asymptotic
	growth of the second moment as $k\to\infty$.
\end{remark}
It is worth emphasizing that we do \emph{not} assume a priori that system~\eqref{system} is
$\ell^2$-exponentially stabilizable in the mean-square sense.  The quantity $\rho(u)$ defined above simply quantifies the exponential rate associated with the chosen admissible control policy $u$; 
if $\rho(u)<1$, the corresponding state trajectories converge in the mean-square sense, whereas 
$\rho(u)\ge 1$ means that it does not achieve a
uniform exponential decay factor strictly below one.

\begin{definition}\label{def-rho-star}
	The optimal stabilizing rate of system~\eqref{system} is then defined as
	\[
	\rho^\ast := \inf_{u\in\mathcal U} \rho(u).
	\]
	A control policy $u^\ast\in\mathcal U$ (if it exists) is called \emph{optimal} if it achieves this
	infimum, i.e., $\rho(u^\ast)=\rho^\ast$.
\end{definition}

\begin{remark}\label{rem-rho}
	Both $\rho(u)$ and $\rho^\ast$ are always well defined.  
For any admissible policy \(u\in\mathcal U\), if the set in
		Definition~\ref{def-rho-u} is nonempty, its infimum exists because it is
		bounded below by \(0\); if it is empty, we use the convention
		\(\inf\emptyset=+\infty\). Hence, \(\rho(u)\) is well defined in
		\([0,+\infty]\). Moreover, since the zero control policy $u_k\equiv0$ belongs to
		\(\mathcal U\) and admits a finite exponential bound, the set
		\(\{\rho(u):u\in\mathcal U\}\) is nonempty, bounded below by \(0\), and
		contains a finite element. Therefore, \(\rho^\ast\) is well defined and
		finite.
	Our definition differs slightly from the deterministic setting of~\cite{Hu17},
	where the analogue of the growth inequality $\|x_k\|\leq \kappa \rho^k \|x_0\|$ used to define the stabilizing rate is written with an
	\emph{arbitrary} vector norm on $\mathbb R^{n}$. For finite-dimensional linear space, all norms are equivalent and the stabilizing rate is therefore
	independent of the chosen norm.
	For the stochastic system~\eqref{system}, the state evolves in
	$L^{2}(\Omega;\mathbb R^{n})$, where different norms are not equivalent in general.
	Consequently, the value of $\rho^\ast$ depends on the chosen norm.
	This paper works with the mean-square norm
	$\big(\mathbb{E}[x^\top x]\big)^{1/2}$.
\end{remark}

\begin{remark}
	The boundary case \(\rho^\ast=0\) is allowed by
	Definition~\ref{def-rho-star}, but it should be interpreted as a limiting
	zero-rate case of the optimal stabilizing rate problem rather than as an
	ordinary exponential stabilization regime with a positive decay factor.
	For a fixed admissible policy \(u\), the quantity \(\rho(u)\) is defined as
	the infimum of all achievable exponential factors in the mean-square estimate
	\[
	\mathbb E[x_k^\top x_k]
	\leq
	\kappa\rho^{2k}\mathbb E[x_0^\top x_0],
	\qquad
	\forall x_0\in L^2(\Omega;\mathbb R^n),\ \forall k\in\mathbb N .
	\]
	Therefore, \(\rho^\ast=0\) means that the achievable decay factors over the
	admissible policy class can be made arbitrarily small: for every
	\(\rho>0\), there exists an admissible policy \(u_\rho\) satisfying the above
	estimate with this decay factor. Hence, \(\rho^\ast=0\) represents a boundary
	phenomenon of the optimization problem, which is stronger than ordinary
	mean-square exponential stabilizability.
	
	The zero rate may or may not be attained by a fixed policy. If there exists
	\(u^\ast\in\mathcal U\) such that \(\rho(u^\ast)=0\), then the infimum is
	achieved. In the linear feedback subclass \(u_k=-Kx_k\), this case has a
	finite-dimensional interpretation. Let
	\[
	F_K=A-BK,\qquad
	\bar F_K=\bar A-\bar B K .
	\]
	The second-moment dynamics satisfy
	\[
	\operatorname{vec}(\Sigma_{k+1})
	=
	\left(
	F_K\otimes F_K+
	\sigma^2\bar F_K\otimes\bar F_K
	\right)
	\operatorname{vec}(\Sigma_k),
	\]
	and therefore
	\[
	\rho(K)
	=
	\sqrt{
		\operatorname{spr}
		\left(
		F_K\otimes F_K+
		\sigma^2\bar F_K\otimes\bar F_K
		\right)} .
	\]
	Consequently, if a finite gain \(K\) attains zero rate, then the associated
	second-moment operator is nilpotent, and the second moment vanishes after
	finitely many steps.
	
	For comparison, consider the deterministic discrete-time linear system
	\[
	x_{k+1}=Ax_k+Bu_k .
	\]
	Under a linear feedback law \(u_k=-Kx_k\), the closed-loop dynamics are
	characterized by \(A-BK\), and the state decay rate is determined by
	\(\operatorname{spr}(A-BK)\). If a finite gain \(K\) achieves zero rate, then
	\[
	\operatorname{spr}(A-BK)=0,
	\]
	which implies that \(A-BK\) is nilpotent and the state converges to zero in
	finite time. This interpretation is consistent with the classical
	controllability decomposition and pole-placement theory: the controllable
	modes can be arbitrarily assigned by state feedback, whereas the
	uncontrollable modes are unaffected by the input and determine the
	limitation of achievable closed-loop rates.
	
	The multiplicative-noise case is fundamentally different. Although the
	second-moment operator becomes linear once a feedback gain is fixed, the
	feedback gain simultaneously affects both the nominal dynamics and the noise
	channel. Therefore, the zero-rate boundary is not characterized solely by a
	controllability decomposition of the pair \((A,B)\), and additional analysis
	of the induced second-moment dynamics is required.
	
	The case \(\rho^\ast=0\) may also be nonattained. Then no admissible policy
	achieves zero rate, but there exists a sequence of admissible policies
	\(\{u_j\}\) such that
	\[
	\rho(u_j)\rightarrow0 .
	\]
	Within the linear feedback subclass, if a sequence of finite gains
	\(\{K_j\}\) satisfies \(\rho(K_j)\rightarrow0\) while zero rate is not
	attained, then the gain sequence cannot remain bounded; otherwise, a
	convergent subsequence together with continuity of the associated
	second-moment operator would produce a finite gain achieving zero rate.
\end{remark}
Having introduced the general optimal stabilizing rate $\rho^\ast$, we also consider its restriction to the state-feedback policy class.
We first introduce the admissible feedback policies
\begin{align*}
	\mathcal U_{\mathrm{fb}}
	:=
	\Bigl\{
	u=(u_0,u_1,\dots)\ \Big|\ 
	&u_k=\mu_k(x_k),\
	\mu_k:\mathbb R^n\to\mathbb R^m,
	\mathbb E[u_k^\top u_k]<\infty, k\in\mathbb N
	\Bigr\},
\end{align*}
where $\{x_k\}$ denotes the state trajectory of system~\eqref{system} driven
by $u$. Clearly, $\mathcal U_{\mathrm{fb}}\subseteq\mathcal U$.
We define the feedback-restricted optimal stabilizing rate by
\[
\rho_{\mathrm{fb}}^\ast:=\inf_{u\in\mathcal U_{\mathrm{fb}}}\rho(u).
\]
In particular, $\rho^\ast\leq\rho_{\mathrm{fb}}^\ast$. The following feedback-restricted problem is the object of the optimal control and RNVI analysis.

\vspace{4pt}
\noindent\textbf{Problem (OSR).} Solve the optimization problem: 
\[
\inf_{u\in\mathcal U_{\mathrm{fb}}} \rho(u),
\]
subject to (\ref{system}).

\medskip
\begin{remark}
	Problem (OSR) does not assume mean-square stabilizability. 
	By the definition of $\rho_{\mathrm{fb}}^\ast$, this quantity is the infimum over nonnegative $\rho$ for which the stability condition holds for some $\kappa$ and an admissible state-feedback policy. 
	Accordingly, $\rho_{\mathrm{fb}}^\ast$ may be less than, equal to, or greater than one; 
	$\rho_{\mathrm{fb}}^\ast<1$ implies that mean-square stabilization is achievable and $\rho_{\mathrm{fb}}^\ast$ characterizes the fastest rate of convergence, 
	whereas $\rho_{\mathrm{fb}}^\ast\ge 1$ indicates that only a minimal growth rate can be attained. 
	We distinguish the general optimal stabilizing rate $\rho^\ast$ from the feedback-restricted optimal stabilizing rate $\rho_{\mathrm{fb}}^\ast$ throughout.
\end{remark}

In the remainder of the paper, we first develop a norm-based
characterization of the general optimal stabilizing rate, which extends the
quantitative stability analysis from deterministic switched systems to the
stochastic setting with multiplicative noise and provides computable bounds
for the optimal stabilizing rate.
We then introduce an optimal control formulation within the state-feedback
policy class. This formulation leads to a Bellman-type equation, which can
be further reduced to a nonlinear matrix eigenvalue problem. Since valid
solutions of this eigenvalue problem require strictly positive-definite
matrices, we develop the Regularized Normalized Value Iteration scheme to
construct regularized fixed points and corresponding feedback controllers.
The resulting solutions yield certified bounds for the feedback-restricted optimal stabilizing
rate. We also study the gap between the certified bounds, including a
sufficient condition for gap closing and a necessary structural condition
for nonvanishing gaps.

\section{Norm-Based Characterization of the Optimal Stabilizing Rate}
\label{sec:norm}

The norm-based framework introduced in~\cite{Hu17} provides a quantitative
approach for analyzing stabilizing rates in deterministic switched systems.
In particular, it establishes computable upper and lower bounds for the
optimal stabilizing rate through the analysis of an induced one-step
operator. In this section, we extend this framework to stochastic systems
with multiplicative noise. The results in this section concern the general optimal stabilizing rate $\rho^\ast$ over $\mathcal U$.

Unlike the deterministic setting, the state of the stochastic system is a
random variable, and the natural state space becomes
\(L^2(\Omega;\mathbb{R}^n)\), which is infinite-dimensional. Unless stated otherwise, all subsequent results are derived under the following nondegeneracy assumption.
\begin{assumption}\label{Nondegenerate}
	The stacked matrix 
	\[
	U := 
	\begin{bmatrix}
		B \\[2pt]
		\sigma \bar{B}
	\end{bmatrix}
	\in \mathbb{R}^{2n\times m}
	\]
	has full column rank.
\end{assumption}

This condition rules out degenerate input configurations in which the deterministic channel $B$ and the multiplicative channel $\bar B$ fail to excite some control directions. In particular, for each $P \succ 0$, it implies
\[
B^\top P B + \sigma^2 \bar{B}^\top P \bar{B} \succ 0.
\]
Thus, the quadratic forms that appear in the mean-square estimates below are strictly positive.
We first introduce the  operator associated with a norm.

\begin{definition}\label{T-sem}
	Given a norm \(\xi\) on
	\(L^2(\Omega;\mathbb{R}^n)\), define the operator
	\(\mathcal T\) by
	\begin{equation}\label{map}
		(\mathcal T\xi)(x)
		:=
		\inf_{v\in L^2_{\mathcal F_k}(\Omega;\mathbb{R}^m)}
		\xi\left(
		(A+\bar A\omega)x+(B+\bar B\omega)v
		\right),
	\end{equation}
	where \(v\) is independent of the next noise realization
	\(\omega\sim\mathcal N(0,\sigma^2)\).
	For simplicity, we write
	$
	\xi_\sharp:=\mathcal T\xi .
	$
\end{definition}

The following proposition provides computable bounds for the optimal
stabilizing rate through the operator.

\begin{proposition}\label{norm-bounds}
	Let
	$
	\xi(x)=\sqrt{\mathbb E[x^\top x]}.
	$
	Then, the following assertions hold.
	
	\begin{itemize}
		\item[(i)]
		If there exists \(\alpha\geq0\) such that
		\[
		\xi_\sharp(x)\geq\alpha\xi(x),
		\qquad
		\forall x\in L^2_{\mathcal F_k}(\Omega;\mathbb R^n),
		\]
		then
		\[
		\rho^\ast\geq\alpha .
		\]
		
		\item[(ii)]
		If there exists \(\beta\geq0\) such that
		\[
		\xi_\sharp(x)\leq\beta\xi(x),
		\qquad
		\forall x\in L^2_{\mathcal F_k}(\Omega;\mathbb R^n),
		\]
		then
		\[
		\rho^\ast\leq\beta .
		\]
	\end{itemize}
\end{proposition}

\begin{proof}
	The proof of part (i) follows directly from the definition of the operator. Indeed, for any admissible policy \(u\),
	the one-step lower bound implies
	\[
	\xi(x_{k+1})\geq \alpha \xi(x_k),
	\]
	and hence
	\[
	\xi(x_k)\geq \alpha^k\xi(x_0).
	\]
	Therefore, no admissible policy can achieve a stabilizing rate smaller
	than \(\alpha\), which yields
	\[
	\rho^\ast\geq\alpha .
	\]
	
	For part (ii), we first characterize the  optimization problem.
	Fix \(x_k\in L^2_{\mathcal F_k}(\Omega;\mathbb R^n)\). By definition of
	\(\xi_\sharp\),
	\[
	\begin{aligned}
		&\xi_\sharp^2(x_k)
		=
		\inf_{v\in L^2_{\mathcal F_k}(\Omega;\mathbb R^m)}
		\mathbb E\Big[
		\big((A+\bar A\omega_k)x_k
		+(B+\bar B\omega_k)v\big)^\top
		\big((A+\bar A\omega_k)x_k
		+(B+\bar B\omega_k)v\big)
		\Big]\\
		&=\inf_{v\in L^2_{\mathcal F_k}(\Omega;\mathbb R^m)}
		\mathbb E\Big[
		x_k^\top (A^\top A+\sigma^2\bar A^\top\bar A)x_k
		+2x_k^\top (A^\top B+\sigma^2\bar A^\top\bar B)v
		+v^\top (B^\top B+\sigma^2\bar B^\top\bar B)v
		\Big].
	\end{aligned}
	\]
	Since \(B^\top B+\sigma^2\bar B^\top\bar B\succ 0\), the minimizing admissible random control exists and is
	given by
	\[
	v^*
	=
	-\big(B^\top B+\sigma^2\bar B^\top\bar B\big)^{-1}
	\big(B^\top A+\sigma^2\bar B^\top\bar A\big)x_k .
	\]
	Substituting this optimizer into the  optimization gives
	\[
	\xi_\sharp(x_k)
	=
	\xi\big(
	(A+\bar A\omega_k)x_k
	+
	(B+\bar B\omega_k)v^*
	\big).
	\]
	By the assumed upper bound,
	\[
	\xi(x_{k+1})
	\leq
	\beta\xi(x_k).
	\]
	Applying this estimate recursively yields
	\[
	\xi(x_k)\leq\beta^k\xi(x_0),
	\qquad
	\forall k\in\mathbb N,\; 	\forall x_0\in L^2(\Omega;\mathbb R^n) .
	\]
	Therefore, the achievable mean-square exponential decay factor satisfies
	\[
	\rho^\ast\leq\beta .
	\] This completes the proof.
\end{proof}

This section extends the norm-based quantitative analysis of
\cite{Hu17} to stochastic systems with multiplicative noise.  In the next section, we develop an alternative viewpoint based on
an optimal control formulation, which leads to a Bellman-type
characterization and a constructive RNVI scheme.
\section{Optimal Control Formulation for Problem (OSR)}\label{sec:ocp-bellman}

In this section, we introduce an optimal control formulation for the feedback-restricted Problem (OSR) defined in Section~\ref{sec:problem}. The purpose is to characterize the
exponential growth rate of the state energy through an optimization problem
defined over state dependent decision rules.

Recall the state-feedback policy class $\mathcal U_{\mathrm{fb}}$ introduced in Section~\ref{sec:problem}.
We emphasize that this restriction is imposed for technical reasons. Working
with feedback policies allows us to express the closed-loop trajectories in
terms of state dependent decision rules, which is needed for the optimal
control formulation below. We do not claim that feedback policies exhaust the full admissible class
\(\mathcal U\), nor that the optimal stabilizing rate over
\(\mathcal U_{\mathrm{fb}}\) coincides with that over \(\mathcal U\).
Nevertheless, Proposition~4.1 of~\cite{NI201565} shows that, 
for the
	controlled mean-field stochastic difference equations considered there,
	open-loop \(\ell^2\)-stabilizability is equivalent to closed-loop
	\(\ell^2\)-stabilizability. The present model (\ref{system}) is a special case of that
	framework with all mean-field coefficient matrices set to zero. Accordingly, the optimal control formulation and all subsequent RNVI results concern the feedback-restricted optimal stabilizing rate \(\rho_{\mathrm{fb}}^\ast\), without claiming its equivalence to the general value \(\rho^\ast\).

For an initial value $x_0\in L^2(\Omega;\mathbb R^n)$ with
$\mathbb E[x_0^\top x_0]>0$ and a feedback policy
$u\in\mathcal U_{\mathrm{fb}}$, define
\begin{align*}
	J(x_0,u)
	:=
	\limsup_{k\to\infty}
	\frac{1}{k}
	\log\left(
	\frac{\mathbb E[x_k^\top x_k]}
	{\mathbb E[x_0^\top x_0]}
	\right).
\end{align*}
This functional represents the asymptotic logarithmic growth rate of the
expected state energy along the trajectory generated from the prescribed
initial value $x_0$. When $J(x_0,u)<0$, the state energy decays
exponentially along this trajectory in the mean-square sense.

\medskip

\noindent\textbf{Problem (OC).} For a given initial value
$x_0\in L^2(\Omega;\mathbb R^n)$ with $\mathbb E[x_0^\top x_0]>0$, solve
\[
\begin{array}{cl}
	\inf\limits_{u \in \mathcal U_{\mathrm{fb}}}
	&
	J(x_0,u)
	=
	\limsup\limits_{k\to\infty}
	\frac{1}{k}
	\log\!\left(
	\frac{\mathbb E[x_k^\top x_k]}
	{\mathbb E[x_0^\top x_0]}
	\right),
	\\[1mm]
	\text{s.t.}
	&
	x_{k+1}
	=
	(A+\bar A\omega_k)x_k
	+
	(B+\bar B\omega_k)u_k .
\end{array}
\]

We next characterize Problem (OC) through a Bellman-type equation and
relate it to the stabilizing rate. Although \(J(x_0,u)\) is evaluated for a prescribed initial value \(x_0\),
	the Bellman-type equation below is imposed in the almost sure sense for each
	\(x_k\in L^2_{\mathcal F_k}(\Omega;\mathbb R^n)\). This formulation yields uniform
	exponential bounds over all admissible initial values and thus connects
	Problem (OC) with the global stabilizing rate under some conditions.

	\begin{theorem}\label{thm:bellman-main}
		Suppose there exist a scalar $\lambda\in\mathbb R$ and a quadratic function $h(x)=x^\top Px$ with
		$P\in\mathbb S^n_{++}$ such that, for all $x_k\in L^2_{\mathcal F_k}(\Omega;\mathbb R^n)$, the following equality holds almost surely,
		\begin{equation}\label{eq-B}
			e^\lambda h(x_k)
			=
			\inf_{\mu:\mathbb R^n\to\mathbb R^m}
			\mathbb E\left[
			h(x_{k+1})\;
			\middle|\;x_k
			\right],\;k\in\mathbb N,
		\end{equation}
		where
		\[
		x_{k+1}
		=
		(A+\bar A\omega_k)x_k
		+
		(B+\bar B\omega_k)\mu(x_k),
		\quad
		\omega_k\sim\mathcal N(0,\sigma^2).
		\]
		Then, the following statements hold.
		
		\begin{itemize}
			\item[i)] For each feedback policy $u\in\mathcal U_{\mathrm{fb}}$
			and each initial value $x_0\in L^2(\Omega;\mathbb R^n)$
			with $\mathbb E[x_0^\top x_0]>0$, it holds
			\[
			J(x_0,u)\ge \lambda,
			\qquad
			\rho(u)\ge e^{\lambda/2}.
			\]
			
			\item[ii)] If a feedback law $\mu^*$ attains the infimum
			in~\eqref{eq-B}, let $u^*_k=\mu^*(x_k)$ for all
			$k$, and define the corresponding feedback policy
			$u^*=(u_0^*,u_1^*,\dots)$. Then, for each initial value
			$x_0\in L^2(\Omega;\mathbb R^n)$ with
			$\mathbb E[x_0^\top x_0]>0$, it holds
			\[
			J(x_0,u^*)=J^*=\lambda,
			\qquad
			\rho(u^*)=\rho_{\mathrm{fb}}^\ast=e^{\lambda/2},
			\]
			where $J^*$ denotes the optimal value of Problem (OC). In particular, $J^*$ is independent of $x_0$.
			
		\end{itemize}
	\end{theorem}
	
	\begin{proof}
		Fix an arbitrary feedback policy \(u\in\mathcal U_{\rm fb}\) and an arbitrary
		initial value \(x_0\in L^2(\Omega;\mathbb R^n)\) with
		\(\mathbb E[x_0^\top x_0]>0\). Since
		\eqref{eq-B} holds  almost surely, the corresponding trajectory of (\ref{system})
		satisfies
		\[
		e^\lambda x_k^\top P x_k
		\le
		\mathbb E\!\left[
		x_{k+1}^\top P x_{k+1}
		\;\middle|\; x_k
		\right], \quad \forall k\in\mathbb N.
		\]
		Taking expectations and iterating this inequality yield
		\begin{equation}\label{eq-exp-growth}
			\mathbb E[x_k^\top P x_k]
			\ge
			e^{k\lambda}\mathbb E[x_0^\top P x_0],
			\qquad \forall k\in\mathbb N.
		\end{equation}
		Since $P\succ0$, we have
		\[
		\lambda_{\min}(P)x^\top x
		\le
		x^\top P x
		\le
		\lambda_{\max}(P)x^\top x .
		\]
		Combining these bounds with~\eqref{eq-exp-growth}, we obtain
		\begin{align}\label{big}
			\mathbb E[x_k^\top x_k]
			\ge
			\frac{\lambda_{\min}(P)}{\lambda_{\max}(P)}
			e^{k\lambda}
			\mathbb E[x_0^\top x_0],
			\qquad \forall k\in\mathbb N .
		\end{align}
		Hence, it holds
		\[
		\frac{1}{k}
		\log\!\left(
		\frac{\mathbb E[x_k^\top x_k]}
		{\mathbb E[x_0^\top x_0]}
		\right)
		\ge
		\lambda
		+
		\frac{1}{k}
		\log\!\left(
		\frac{\lambda_{\min}(P)}{\lambda_{\max}(P)}
		\right).
		\]
		Taking the limit superior yields
		\[
		J(x_0,u)\ge \lambda .
		\]
		Since the initial value \(x_0\) was chosen arbitrarily, this
		lower bound holds for each
		\(x_0\in L^2(\Omega;\mathbb R^n)\) with
		\(\mathbb E[x_0^\top x_0]>0\).
		
		It remains to prove the lower bound for $\rho(u)$. If the set in the
		definition of $\rho(u)$ is empty, then $\rho(u)=+\infty$, and
		inequality $\rho(u)\ge e^{\lambda/2}$ is trivial. Otherwise, take any $\bar\rho>0$ for which there
		exists $\kappa>0$ such that
		\[
		\mathbb E[x_k^\top x_k]
		\le
		\kappa \bar\rho^{2k}\mathbb E[x_0^\top x_0],
		\quad
		\forall x_0\in L^2(\Omega;\mathbb R^n),\quad
		\forall k\in\mathbb N .
		\]
		Combining this estimate with~\eqref{big} and using
		\(\mathbb E[x_0^\top x_0]>0\), we obtain
		\[
		\frac{\lambda_{\min}(P)}{\lambda_{\max}(P)}
		e^{k\lambda}
		\le
		\kappa \bar\rho^{2k},
		\quad
		\forall k\in\mathbb N.
		\]
		Taking the $k$th root and letting $k\to\infty$ yield
		\[
		e^\lambda\le \bar\rho^2.
		\]
		Since this holds for each such \(\bar\rho\), it follows that
		\[
		\rho(u)\ge e^{\lambda/2}.
		\]
		This proves statement (i).
		
		Now, suppose that $\mu^*$ attains the infimum in~\eqref{eq-B} for
		all $x_k\in L^2_{\mathcal F_k}(\Omega;\mathbb R^n)$, and let $u_k^*=\mu^*(x_k)$. Then, the preceding
		inequality (\ref{eq-exp-growth}) becomes an equality along the trajectory generated by
		$u^*$:
		\[
		\mathbb E\!\left[
		x_{k+1}^\top P x_{k+1}
		\;\middle|\; x_k
		\right]
		=
		e^\lambda x_k^\top P x_k, \quad \forall k\in\mathbb N.
		\]
		Hence, it holds
		\[
		\mathbb E[x_k^\top P x_k]
		=
		e^{k\lambda}\mathbb E[x_0^\top P x_0],
		\qquad \forall k\in\mathbb N.
		\]
		Using again the eigenvalue bounds of $P$, we obtain, for the arbitrary
		initial value \(x_0\) fixed above,
		\begin{align}\label{bbig}
			\frac{\lambda_{\min}(P)}{\lambda_{\max}(P)}
			e^{k\lambda}\mathbb E[x_0^\top x_0]
			\le
			\mathbb E[x_k^\top x_k]
			\le
			\frac{\lambda_{\max}(P)}{\lambda_{\min}(P)}
			e^{k\lambda}\mathbb E[x_0^\top x_0].
		\end{align}
		Consequently,
		\[
		J(x_0,u^*)=\lambda
		\]
		holds for each \(x_0\in L^2(\Omega;\mathbb R^n)\) with
		\(\mathbb E[x_0^\top x_0]>0\).
		Moreover, since \(x_0\) was arbitrary and the constant in the upper bound
		of~\eqref{bbig} is independent of \(x_0\), the number \(e^{\lambda/2}\)
		satisfies the defining exponential estimate for \(u^*\). Hence,
		\[
		\rho(u^*)\le e^{\lambda/2}
		\]
		holds. Combining this with statement (i), we obtain
		\[
		\rho(u^*)=e^{\lambda/2}.
		\]
		Therefore, $u^*$ attains the optimal value of Problem (OC) for each
		initial value $x_0$ with
		\(\mathbb E[x_0^\top x_0]>0\), and this value is $\lambda$, independent
		of $x_0$. It also attains the optimal stabilizing rate within the
		feedback policy class, and thus
		\[
		J^*=\lambda,
		\qquad
		\rho_{\mathrm{fb}}^\ast=e^{\lambda/2}.
		\]
		This completes the proof.
		\hfill $\square$
	\end{proof}
	Although \(J(x_0,u)\) may depend on the initial state for a general
	policy, Theorem~\ref{thm:bellman-main} shows that, when an optimal feedback
	law attains the infimum in equation (\ref{eq-B}), the optimal value is
	independent of \(x_0\) and satisfies
	\[
	J^*=2\log\rho_{\mathrm{fb}}^\ast.
	\]
	Hence, within the class of feedback policies, Problem (OC) is equivalent to Problem (OSR) under the condition of Theorem \ref{thm:bellman-main}.

\begin{remark}
	Equation~\eqref{eq-B} is not the exact Bellman equation associated with Problem (OC), since the logarithmic growth rate in $J$ does not admit a direct additive decomposition.  
	Nevertheless, it serves as a powerful surrogate: if a positive-definite quadratic solution exists, it provides a rigorous lower bound on $J(x_0,u)$ and coincides with the optimal value when the infimum is attained.  
	Hence, we shall refer to~\eqref{eq-B} as a Bellman-type equation in the remainder of the paper.
\end{remark}

For any $P\in\mathbb S^n_{++}$, define the standard Riccati-type blocks
\begin{align*}
	R(P)&:=B^\top P B+\sigma^2\,\bar B^\top P\bar B,\\
	S(P)&:=A^\top P B+\sigma^2\,\bar A^\top P\bar B,
\end{align*}
and the associated operator
\begin{align}\label{eq-Phi}
	\Phi(P):=A^\top P A+\sigma^2\,\bar A^\top P\bar A\;-\;S(P)\,R(P)^{-1}\,S(P)^\top,
\end{align}
together with the feedback gain
\begin{align}\label{eq-K}
	K(P):=R(P)^{-1}S(P)^\top.
\end{align}
Under Assumption~\ref{Nondegenerate}, $R(P)\succ0$, ensuring that the inverse in~\eqref{eq-Phi} is well-defined. 
If $P$ is only positive semi-definite, $R(P)$ may fail to be invertible;  
in this case, the operator $\Phi(P)$ is defined by replacing the inverse $R(P)^{-1}$ with the Moore-Penrose generalized inverse $R(P)^{\dagger}$, where $R(P)^\dagger$ denotes the Moore-Penrose generalized inverse (\cite{penrose55}),
\begin{align}\label{eq:Phi-dagger}
	\tilde\Phi(P)
	:= A^\top P A+\sigma^2\,\bar A^\top P\bar A
	- S(P)\,R(P)^{\dagger} S(P)^\top.
\end{align}

\begin{proposition}\label{prop-quadratic}
	The following statements are equivalent.
	\begin{itemize}
		\item [i)] There exists $(\lambda,h(\cdot))$ with $h(x)=x^\top P x$, $P\in\mathbb S^n_{++}$, satisfying the Bellman-type equation~\eqref{eq-B}.
		\item [ii)] There exists $(\gamma,P)\in \mathbb R_{+}\times \mathbb S_{++}^n$ such that 
		\begin{align}\label{eq-eig}
			\Phi(P)=\gamma P,\qquad \gamma=e^\lambda.
		\end{align}
	\end{itemize}
	Equation~\eqref{eq-eig} thus represents a \emph{nonlinear matrix eigenvalue problem}.
	In this case, the feedback policy achieving the infimum in~\eqref{eq-B} 
	is  $\mu^{*}(x)=-K(P)x$
	with $K(P)$ defined in~\eqref{eq-K}.  
\end{proposition}

\begin{proof}
	Suppose i) holds.
	For any  state $x\in L^2_{\mathcal F_k}(\Omega;\mathbb R^n)$, consider a control law $\mu(\cdot)$ applied at this state. 
	Substituting $h(x)$ into~\eqref{eq-B} then gives the quadratic optimization problem
	\begin{align}\label{eq-quadratic-u}
		&\inf_{\mu:\,\mathbb{R}^n\to\mathbb{R}^m}\ \mathbb E\Big[ h\big((A+\bar A \omega)x+(B+\bar B \omega)\mu(x)\big)\,\big|\,x\Big] \nonumber = \inf_{\mu:\,\mathbb{R}^n\to\mathbb{R}^m} \Big\{ x^\top\!\big(A^\top P A+\sigma^2\bar A^\top P \bar A\big)x   \nonumber\\
		&\quad\quad\quad + 2x^\top\!\big(A^\top P B+\sigma^2\bar A^\top P \bar B\big)\mu(x)
		+ (\mu(x))^\top\!\big(B^\top P B+\sigma^2\bar B^\top P \bar B\big)\mu(x) \Big\}, 
	\end{align}
	where $\omega \sim \mathcal N(0,\sigma^2)$ is independent of $x$.
	By Assumption~\ref{Nondegenerate}, 
	$B^\top P B+\sigma^2\bar B^\top P\bar B\succ0$ 
	and thus the problem  has a unique minimizer
	\[
	\mu^*(x) = -K(P)x.
	\]
	Therefore, the Bellman-type equation~\eqref{eq-B} reduces exactly to the nonlinear eigenvalue problem $\Phi(P)=\gamma P$.  
	
	Conversely, suppose there exist $P\succ0$ and $\gamma>0$ such that $\Phi(P)=\gamma P$. Then, for $h(x)=x^\top P x$ and control policy $\mu^{*}(x)=-K(P)x$, we obtain
	\begin{align*}
		&\inf_{\mu:\,\mathbb{R}^n\to\mathbb{R}^m}
		\mathbb{E}\!\left[
		h\!\left( (A+\bar A \omega)x + (B+\bar B \omega)\mu(x) \right)
		\middle| x
		\right]\\
		&\quad=x^\top\Phi(P)x = \gamma h(x);
	\end{align*}
	this shows that $\Phi(P)\succeq 0$ and $(\lambda,h)$ with $\lambda=\log\gamma$ solves the Bellman-type equation~\eqref{eq-B}.  
	This proves the equivalence. \hfill $\square$
\end{proof}

This equivalence reduces the search for solutions of (\ref{eq-B}) to a nonlinear matrix eigenvalue problem of the form $\Phi(P)=\gamma P$. 
In the next section, we establish the existence of strictly positive-definite fixed points for a perturbed version of equation (\ref{eq-eig}) by introducing a regularization scheme, which also leads to computable performance bounds for $\rho_{\mathrm{fb}}^\ast$.

\section{Regularized Approximations and Certified Performance Bounds}\label{sec:regularization}

In this section, we develop a constructive framework for analyzing and computing the solution of a perturbed version of nonlinear eigenvalue matrix equation~\eqref{eq-eig}. The framework addresses both the existence of strictly positive-definite solutions and the computation of certified performance bounds. 
It is organized around four mutually reinforcing components.

(i) \textbf{Regularized Existence Guarantee.} 
Direct attempts to solve~\eqref{eq-eig} may lead only to semi-definite solutions. 
To overcome the possible absence of such solutions, we introduce a regularization parameter $\tau\in(0,1)$ and construct a regularized normalized operator $\hat{\Phi}_\tau$ defined in (\ref{Phi-tau}).  
The corresponding fixed points $P^{(\tau)}\succ0$ are well-defined and are subsequently used to construct
certified bounds for the optimal stabilizing rate.

(ii) \textbf{Certified Performance Bounds.} 
The fixed point $P^{(\tau)}$ provides computable information on the optimal cost and optimal stabilizing rate. 
The quadratic functions yield explicit lower bounds, while the feedback policy induced by $P^{(\tau)}$ provides computable upper bounds. 
Together, these results provide computable bounds for   $\rho_{\mathrm{fb}}^\ast$.

	(iii) \textbf{Certified Gap Analysis.}
	We further study the gap between the certified lower and upper bounds.
	First, we give a sufficient condition under which this gap closes as
	\(\tau\to0\). Second, we identify a necessary structural condition
	satisfied by each regular nonvanishing-gap fixed-point sequence.
	Specifically, there exists a subsequence of the regularization parameters
	and their corresponding fixed points such that the fixed points converge
	to a singular limit and the feedback gains corresponding to these fixed
	points converge to a limiting feedback matrix. The resulting limiting
	closed-loop matrices admit simultaneous block upper triangular forms,
	and their critical and dominant diagonal block pairs satisfy a strict
	mean-square spectral separation.

(iv) \textbf{Regularized Value Iteration Algorithm.} 
To compute $P^{(\tau)}$ in practice, we propose a \emph{regularized normalized value iteration} (RNVI) algorithm. 
For sufficiently large~$\tau$, we establish global contraction via explicit Lipschitz constants, and then employ a continuation strategy to extend the convergence guarantees to smaller~$\tau$.

Taken together, these elements yield a unified algorithm that produces both the regularized approximations $P^{(\tau)}$ and the certified performance bounds for the optimal stabilizing rate. 
For clarity, we summarize the constants and notations used throughout this section.  
Define
\begin{align}\label{def-C}
	R_0 := B^\top B + \sigma^2\,\bar B^\top \bar B,\quad
	C_A := \lambda_{\max}\big(A A^\top + \sigma^2\,\bar A \bar A^\top\big).
\end{align}
For a regularization parameter $\tau\in(0,1)$, set
\begin{align}\label{def-delta}
	\delta_\tau := \frac{\tfrac{\tau}{n}}{(1-\tau)C_A + \tau}.
\end{align}
We further denote the following auxiliary constants that appear in the Lipschitz analysis:
\begin{align}
	\alpha_A := \|A\|^2 + \sigma^2\|\bar A\|^2, \quad
	\alpha_S := \,\|A\|\,\|B\| + \sigma^2\|\bar A\|\,\|\bar B\|, \quad
	\alpha_R := \|B\|^2 + \sigma^2\|\bar B\|^2.
\end{align}

\subsection{Existence of Regularized Positive-Definite Solutions}

The nonlinear matrix eigenvalue equation~\eqref{eq-eig} may in general admit only semi-definite solutions.  
To overcome the possible absence of such solutions, we introduce a \emph{regularized normalization scheme}, which forms the foundation of our subsequent fixed-point analysis.
\begin{definition}\label{dam-eq}
	For $\tau\in(0,1)$, define
	\begin{align}\label{Phi-tau}
		\widehat\Phi_\tau(P):=\frac{(1-\tau)\,\Phi(P)+\tfrac{\tau}{n}I}
		{\operatorname{Tr}\!\big((1-\tau)\,\Phi(P)+\tfrac{\tau}{n}I\big)},\qquad P\in\mathbb S_{++}^n.
	\end{align}
\end{definition}
This operator interpolates between the unregularized map $\frac{\Phi}{\operatorname{Tr}\Phi}$ (as $\tau\to0$) and the uniform normalization $I/n$ (as $\tau\to1$), ensuring that its image always remains in the interior of $\mathbb{S}_{+}^n$.
We recall Brouwer’s fixed-point theorem, which provides the existence of fixed points for continuous self-maps on compact convex sets (see, e.g.,Theorem 8.1.3 of~\cite{borwein2006}). 
A self map refers to a mapping whose domain and range coincide.
\begin{theorem}[Brouwer]\label{Brouwer}
	Any continuous self map of a nonempty compact convex subset of a finite-dimensional Euclidean space has a fixed point.
\end{theorem}

\begin{theorem}\label{thm-regularized}
	For each $\tau\in(0,1)$, there exists $P^{(\tau)}\succ 0$ with $\operatorname{Tr}P^{(\tau)}=1$ such that
	\begin{align}\label{tau-fix}
		\widehat\Phi_\tau(P^{(\tau)})=P^{(\tau)}.
	\end{align}
\end{theorem}

\begin{proof}
	For any $P\succeq0$ with $\operatorname{Tr}P=1$, note that
	\[
	\operatorname{Tr}\Phi(P)\ \le\ \operatorname{Tr}\!\big(P(AA^\top+\sigma^2\bar A \bar A^\top)\big)\ \le\ C_A
	\]
	with $C_A$ defined in (\ref{def-C}). Hence, for any $\tau\in(0,1)$, it holds
	\begin{align*}
		(1-\tau)\Phi(P)+\tfrac{\tau}{n}I \ \succeq\ \tfrac{\tau}{n}I,\quad \operatorname{Tr}\big((1-\tau)\Phi(P)+\tfrac{\tau}{n}I\big)\ \le\ (1-\tau)C_A+\tau.
	\end{align*}
	Then, the regularized operator $\widehat\Phi_\tau(P)$ satisfies
	\begin{align*}
		\lambda_{\min}\!\Big(\widehat\Phi_\tau(P)\Big)
		=\frac{\lambda_{\min}\!\big((1-\tau)\Phi(P)+\tfrac{\tau}{n}I\big)}
		{\operatorname{Tr}\!\big((1-\tau)\Phi(P)+\tfrac{\tau}{n}I\big)}
	\ge \frac{\tfrac{\tau}{n}}{(1-\tau)C_A+\tau}
		=\delta_\tau>0.
	\end{align*}
	Define the trace-normalized slice
	\[
	\mathcal S_{\delta_\tau}:=\Big\{X\in\mathbb S^n_{++}:\ X\succeq \delta_\tau I,\ \mathrm{Tr}(X)=1\Big\}.
	\]
	Since $\delta_\tau \le 1/n$, we have $I/n \in \mathcal S_{\delta_\tau}$, which shows that the set $\mathcal S_{\delta_\tau}$ is nonempty; moreover, this set is convex and compact. By its construction,  $\widehat\Phi_\tau:\mathcal S_{\delta_\tau}\to \mathcal S_{\delta_\tau}$ is a continuous mapping.  
	Thus, by Theorem~\ref{Brouwer}, there exists $P^{(\tau)}\in\mathcal S_{\delta_\tau}$ with $\widehat\Phi_\tau(P^{(\tau)})=P^{(\tau)}$. \hfill $\square$
\end{proof}

Having established the existence of strictly positive definite fixed points for each $\tau>0$, 
we now investigate the behavior as the regularization parameter tends to zero.

\begin{theorem}\label{thm:limit-Ptau}
	Let $\{\tau_j\}_{j\ge1}$ be a decreasing sequence in $(0,1)$ with
	$\tau_j\downarrow 0$. For each $\tau_j$, let $P^{(\tau_j)}$ be the
	regularized fixed point from Theorem~\ref{thm-regularized}, and define
	\[
	\gamma^{(\tau_j)}
	:=
	\operatorname{Tr}\!\Big((1-\tau_j)\Phi(P^{(\tau_j)})
	+\tfrac{\tau_j}{n}I\Big).
	\]
	Then, there exist convergent subsequences of
	\(\{P^{(\tau_j)}\}_{j\ge1}\) and
	\(\{\gamma^{(\tau_j)}\}_{j\ge1}\) (not relabelled), with
	limits $P^\ast\succeq0$ and $\gamma^\ast\in[0,C_A]$ such that
	\[
	P^{(\tau_j)} \to P^\ast,\qquad
	\gamma^{(\tau_j)} \to \gamma^\ast,\qquad
	\Phi(P^{(\tau_j)}) \to \gamma^\ast P^\ast
	\]
	as $j\to \infty$.
\end{theorem}

\begin{proof}
	By Theorem~\ref{thm-regularized}, for each \(\tau\in(0,1)\), the
	regularized fixed point \(P^{(\tau)}\) satisfies
	\[
	P^{(\tau)}\succeq0,\qquad \operatorname{Tr}P^{(\tau)}=1.
	\]
	Thus, for any sequence \(\tau_j\downarrow0\), all matrices
	\(P^{(\tau_j)}\) belong to the set
	\[
	\mathcal S
	:=
	\{X\in\mathbb S^n_+:\operatorname{Tr}X=1\}.
	\]
	This set is closed and bounded in the finite-dimensional space of
	symmetric matrices, and hence compact. By the Bolzano--Weierstrass
	theorem, the sequence \(\{P^{(\tau_j)}\}_{j\ge1}\) admits a convergent
	subsequence. Passing to this subsequence and without relabeling, we may
	assume that
	\[
	P^{(\tau_j)}\to P^\ast
	\qquad\text{for some }P^\ast\in\mathcal S.
	\]
	In particular, \(P^\ast\succeq0\) and \(\operatorname{Tr}P^\ast=1\).
	We next extract a convergent subsequence for the corresponding scalars.
	By the definition of \(\gamma^{(\tau_j)}\) and the trace estimate used in
	the proof of Theorem~\ref{thm-regularized}, we have
	\begin{align}\label{gamma-bound}
		0<\gamma^{(\tau_j)}
		=
		(1-\tau_j)\operatorname{Tr}\Phi(P^{(\tau_j)})+\tau_j\le
		(1-\tau_j)C_A+\tau_j
		\le C_A+1 .
	\end{align}
	Thus, \(\{\gamma^{(\tau_j)}\}_{j\ge1}\) is bounded in \(\mathbb R\).
	Applying the Bolzano--Weierstrass theorem again, we pass if necessary to a
	further subsequence of the already selected subsequence of \(\{\tau_j\}_{j\ge1}\) and, without
	relabeling, assume that
	\[
	\gamma^{(\tau_j)}\to\gamma^\ast
	\]
	for some \(\gamma^\ast\in\mathbb R\). Since
	\(\tau_j\to0\), the above bound (\ref{gamma-bound}) also gives \(\gamma^\ast\in[0,C_A]\).
	Since this second extraction is taken from the subsequence along which
	\(P^{(\tau_j)}\to P^\ast\), the convergence of \(P^{(\tau_j)}\) is
	preserved along the same indices. Therefore, along the final subsequence, it holds
	\[
	P^{(\tau_j)}\to P^\ast,\qquad
	\gamma^{(\tau_j)}\to\gamma^\ast,\qquad
	\tau_j\to0.
	\]
	For each \(j\), the regularized fixed point relation gives
	\begin{align*}
		P^{(\tau_j)}
		=
		\frac{(1-\tau_j)\Phi(P^{(\tau_j)})+\tfrac{\tau_j}{n}I}
		{\operatorname{Tr}\!\left((1-\tau_j)\Phi(P^{(\tau_j)})
			+\tfrac{\tau_j}{n}I\right)}=
		\frac{(1-\tau_j)\Phi(P^{(\tau_j)})+\tfrac{\tau_j}{n}I}
		{\gamma^{(\tau_j)}} .
	\end{align*}
	Rearranging this identity yields
	\[
	\Phi(P^{(\tau_j)})
	=
	\frac{\gamma^{(\tau_j)}P^{(\tau_j)}
		-\tfrac{\tau_j}{n}I}{1-\tau_j}.
	\]
	Using the convergences obtained above, we have
	\[
	\gamma^{(\tau_j)}P^{(\tau_j)}
	-\tfrac{\tau_j}{n}I
	\to
	\gamma^\ast P^\ast,
	\qquad
	1-\tau_j\to1.
	\]
	Consequently, it holds
	\[
	\Phi(P^{(\tau_j)})\to\gamma^\ast P^\ast.
	\]
	This proves the claim. 
\end{proof}

\begin{remark}
	The limiting matrix $P^\ast$ may in general lie on the boundary of $\mathbb S^n_{+}$, and thus fail to be strictly positive definite. 
	Nevertheless, the regularization procedure ensures that for each $\tau>0$ the approximation $P^{(\tau)}$ remains strictly positive definite, thereby providing a family of well-conditioned surrogates approaching~$P^\ast$. As shown in the next subsection, these approximations also enable the derivation of explicit lower and upper bounds for the optimal stabilizing rate $\rho_{\mathrm{fb}}^\ast$, thus yielding computable near-optimality guarantees.
\end{remark}

\subsection{Certified Performance Bounds}

In this subsection, we leverage the regularized fixed points established in Theorem~\ref{thm-regularized} to derive the certified performance guarantees.
The quadratic functions yield universal lower bounds for the optimal stabilizing rate, 
while the feedback policies induced by the regularized fixed points provide computable upper bounds.  
Together, these results lead to two-sided estimates for the optimal stabilizing rate~$\rho_{\mathrm{fb}}^\ast$.

\begin{lemma}\label{lem:CW-local}
	Let $P\in\mathbb S_{++}^n$ and define
	\begin{align*}
		L(P)&:=\lambda_{\min}\big(P^{-1/2}\Phi(P)P^{-1/2}\big),\\
		U(P)&:=\lambda_{\max}\big(P^{-1/2}\Phi(P)P^{-1/2}\big),
	\end{align*}
	where $\Phi(P)$ is given by~\eqref{eq-Phi}. Then, for all $x\in\mathbb{R}^n$, it holds
	\[
	L(P)\,h_P(x)\ \le\ x^\top\Phi(P)\,x\ \le\ U(P)\,h_P(x)
	\]
	with $h_P(x)=x^\top Px$.
\end{lemma}

\begin{proof}
	Set $y:=P^{1/2}x$. Then, we obtain
	\[
	x^\top\Phi(P)x \;=\; y^\top\!\big(P^{-1/2}\Phi(P)P^{-1/2}\big)y,
	\quad
	h_P(x)=y^\top y.
	\]
	The bounds follow immediately from the min-max characterization of eigenvalues.
\end{proof}

The next result shows that each quadratic test function yields a lower bound for cost $J$ and $\rho(u)$.

\begin{proposition}\label{prop:J-lower}
	For $P\in\mathbb S_{++}^n$, let $L(P)$ be given in
	Lemma~\ref{lem:CW-local}. For any admissible policy
	$u\in\mathcal U_{\mathrm{fb}}$ and any
	$x_0\in L^2(\Omega;\mathbb R^n)$ with
	$\mathbb E[x_0^\top x_0]>0$, it holds
	\[
	J(x_0,u)\ge \log L(P),
	\quad
	\rho(u)\ge\sqrt{L(P)},
	\]
	where $\log0:=-\infty$. In particular,
	\[
	\inf_{u\in\mathcal U_{\mathrm{fb}}}J(x_0,u)\ge\sup_{P\succ0}\log L(P),
	\quad
	\rho_{\mathrm{fb}}^\ast\ge\sup_{P\succ0}\sqrt{L(P)}.
	\]
\end{proposition}

\begin{proof}
	For each $x_k$, the one-step quadratic minimization defining
	$\Phi(P)$ gives
	\[
	\min_{u_k}\mathbb E\!\left[h_P(x_{k+1})\mid x_k\right]
	=x_k^\top\Phi(P)x_k
	\ge L(P)h_P(x_k).
	\]
	Hence, under any $u\in\mathcal U_{\mathrm{fb}}$, it holds
	\[
	\mathbb E[h_P(x_k)]
	\ge \big(L(P)\big)^k\mathbb E[h_P(x_0)].
	\]
	Since $P\succ0$, this implies
	\begin{align}\label{J-upp}
		\mathbb E[x_k^\top x_k]
		\ge
		\frac{\lambda_{\min}(P)}{\lambda_{\max}(P)}
		\big(L(P)\big)^k
		\mathbb E[x_0^\top x_0].
	\end{align}
	Taking the logarithmic growth rate yields
	$J(x_0,u)\ge\log L(P)$. To prove the bound for \(\rho(u)\), let \(\bar\rho\) be any rate
	admissible in Definition~\ref{def-rho-u}. Then, there exists
	\(C_{\bar\rho}>0\) such that, for each \(k\) and each 
	initial value \(x_0\in L^2(\Omega;\mathbb R^n)\),
	\[
	\mathbb E[x_k^\top x_k]
	\le
	C_{\bar\rho}\bar\rho^{2k}
	\mathbb E[x_0^\top x_0].
	\]
	Combining this estimate with the preceding lower bound (\ref{J-upp}) gives
	\[
	\frac{\lambda_{\min}(P)}{\lambda_{\max}(P)}
	\bigl(L(P)\bigr)^k
	\le
	C_{\bar\rho}\bar\rho^{2k}.
	\]
	Taking the \(k\)th root and letting \(k\to\infty\) yields
	\(L(P)\le\bar\rho^2\). Since \(\bar\rho\) is arbitrary,
	\[
	\rho(u)\ge\sqrt{L(P)}.
	\]
	Taking the infimum over $u$ and the supremum over $P\succ0$
	completes the proof. 
\end{proof}

\begin{remark}\label{rem:best-quad}
	The term $\sup_{P\succ0}L(P)$ depends only on the system matrices and represents the sharpest lower bound achievable by quadratic test functions.  
	In practice, however, directly maximizing $L(P)$ is computationally intractable.  
	Instead, one evaluates $L(P)$ at the regularized fixed points $P^{(\tau)}$ defined in~\eqref{tau-fix}, which are computable via the regularized operator $\widehat{\Phi}_\tau$.
\end{remark}

We now turn to upper bounds for the optimal stabilizing rate $\rho_{\mathrm{fb}}^\ast$.  
For each $\tau\in(0,1)$, recall from Definition~\ref{dam-eq} and Theorem~\ref{thm-regularized} that $\widehat{\Phi}_\tau(P^{(\tau)})=P^{(\tau)}$ admits a strictly positive-definite solution with trace one.  
Let 
\begin{align}\label{def-u-tau}
	\mu^{*}_{(\tau)}(x):=-K(P^{(\tau)})x
\end{align}
denote the corresponding one-step minimizing feedback law associated with
the minimization $\min_{u_k}\mathbb E[
h_{P^{(\tau)}}(x_{k+1})\mid x_k]$.
\begin{proposition}\label{prop:J-upper}
	For $\tau\in(0,1)$, let $P^{(\tau)}\succ0$ solve
	\eqref{tau-fix} and define
	\[
	\gamma^{(\tau)}
	:=
	\operatorname{Tr}\!\Big(
	(1-\tau)\Phi(P^{(\tau)})+\tfrac{\tau}{n}I
	\Big).
	\]
	Then, for each $x_0\in L^2(\Omega;\mathbb R^n)$ with
	$\mathbb E[x_0^\top x_0]>0$, the feedback policy
	$u_{(\tau)}^{*}=(\mu_{(\tau)}^{*}(x_0), \mu_{(\tau)}^{*}(x_1),\cdots)$ defined in~\eqref{def-u-tau} satisfies
	\[
	\inf_{u\in\mathcal U_{\mathrm{fb}}}J(x_0,u)
	\le J(x_0,u_{(\tau)}^{*})
	\le\log\!\Big(\tfrac{\gamma^{(\tau)}}{1-\tau}\Big),
	\]
	and
	\[
	\rho_{\mathrm{fb}}^\ast
	\le\rho(u_{(\tau)}^{*})
	\le\sqrt{\tfrac{\gamma^{(\tau)}}{1-\tau}}.
	\]
\end{proposition}
\begin{proof}
	From~\eqref{tau-fix}, for any $x$, it holds
	\[
	x^\top\Phi(P^{(\tau)})x
	=
	\frac{\gamma^{(\tau)}h_{P^{(\tau)}}(x)
		-\tfrac{\tau}{n}x^\top x}{1-\tau}
	\le
	\frac{\gamma^{(\tau)}}{1-\tau}h_{P^{(\tau)}}(x)
	\]
	with $h_{P^{(\tau)}}(x)=x^\top P^{(\tau)}x$.
	Under $u_{(\tau)}^{*}$, we have  
	\[
	\mathbb E[h_{P^{(\tau)}}(x_{k+1})\mid x_k]=x_k^\top \Phi(P^{(\tau)})x_k \le \tfrac{\gamma^{(\tau)}}{1-\tau}\ h_{P^{(\tau)}}(x_k).
	\]
	Taking expectations and iterating give
	\[
	\mathbb E[h_{P^{(\tau)}}(x_k)]
	\le
	\left(\frac{\gamma^{(\tau)}}{1-\tau}\right)^k
	\mathbb E[h_{P^{(\tau)}}(x_0)].
	\]
	Thus, it holds
	\[
	J(x_0,u_{(\tau)}^{*})
	\le
	\log\!\Big(\tfrac{\gamma^{(\tau)}}{1-\tau}\Big).
	\]
	Since \(P^{(\tau)}\succ0\), for each
	\(x_0\in L^2(\Omega;\mathbb R^n)\), the same estimate gives
	\[
	\mathbb E[x_k^\top x_k]
	\le
	\frac{\lambda_{\max}(P^{(\tau)})}
	{\lambda_{\min}(P^{(\tau)})}
	\left(\frac{\gamma^{(\tau)}}{1-\tau}\right)^k
	\mathbb E[x_0^\top x_0].
	\]
	Both
	\[
	\frac{\lambda_{\max}(P^{(\tau)})}
	{\lambda_{\min}(P^{(\tau)})}
	\quad\text{and}\quad
	\frac{\gamma^{(\tau)}}{1-\tau}
	\]
	are independent of \(k\) and \(x_0\).
	Therefore,
	\(\sqrt{\frac{\gamma^{(\tau)}}{1-\tau}}\) is admissible in
	Definition~\ref{def-rho-u} for $u_{(\tau)}^{*}$, and hence
	$
	\rho(u_{(\tau)}^{*})
	\le
	\sqrt{\tfrac{\gamma^{(\tau)}}{1-\tau}}.
	$
	The remaining inequalities follow directly from the definitions of
	\(\inf_{u\in\mathcal U_{\mathrm{fb}}}J(x_0,u)\) and \(\rho_{\mathrm{fb}}^\ast\). 
\end{proof}

The following theorem summarizes the main result of this paper, 
providing explicit two-sided bounds that certify the optimal  stabilizing rate.
\begin{theorem}\label{cor:certified}
	For any $\tau\in(0,1)$ and any solution
	$P^{(\tau)}\succ0$ of~\eqref{tau-fix}, 
	the optimal stabilizing rate satisfy
	\begin{align*}
		\sqrt{L\big(P^{(\tau)}\big)}
		&\le \rho_{\mathrm{fb}}^\ast
		\le \sqrt{\tfrac{\gamma^{(\tau)}}{1-\tau}}.
	\end{align*}
	Here, $L(P^{(\tau)})$ is defined in
	Lemma~\ref{lem:CW-local}, and $\gamma^{(\tau)}$ is defined in
	Proposition~\ref{prop:J-upper}.
\end{theorem}

\subsection{Certified Gap Analysis.}

Having obtained the computable upper and lower bounds for $\rho_{\mathrm{fb}}^\ast$ in Theorem~\ref{cor:certified}, 
we next analyze how tight these bounds are as $\tau$ varies.

\begin{theorem}\label{thm:gap}
	For any $\tau\in(0,1)$ and any solution $P^{(\tau)}$ of~\eqref{tau-fix}, define
	\[
	L_\tau := L\big(P^{(\tau)}\big),\quad
	U_\tau := \frac{\gamma^{(\tau)}}{1-\tau},\quad
	\Delta_\tau := U_\tau - L_\tau .
	\]
	Then, the following statements hold.
	\begin{itemize}
		\item [i)] The gap between the squared certified bounds for $\rho_{\mathrm{fb}}^\ast$ in Theorem~\ref{cor:certified} admits the explicit expression
		\begin{align}\label{gap-em}
			\Delta_\tau
			= \frac{\tau}{n(1-\tau)}\,\lambda_{\max}\big((P^{(\tau)})^{-1}\big).
		\end{align}
		\item [ii)] Let $\{\tau_k\}_{k\ge1}\subset(0,1)$ with $\tau_k\downarrow0$, and let
		$P^*\succeq0$ be a limit point of $\{P^{(\tau_k)}\}$ as characterized in
		Theorem~\ref{thm:limit-Ptau}.
		If the limit satisfies $P^*\succ0$, then the certified squared bound becomes
		asymptotically tight along the corresponding convergent subsequence:
		\[
		\Delta_{\tau_k}\ \longrightarrow\ 0
		\qquad\text{as }k\to\infty.
		\]
	\end{itemize}
\end{theorem}

\begin{proof}
	Fix $\tau\in(0,1)$ and a corresponding fixed point $P^{(\tau)}\succ0$ of~\eqref{tau-fix}.
	By the fixed-point relation~\eqref{tau-fix}, it holds
	\[
	(1-\tau)\,\Phi(P^{(\tau)})+\tfrac{\tau}{n}I = \gamma^{(\tau)} P^{(\tau)}.
	\]
	Multiplying by $({P^{(\tau)}})^{-1/2}$ on both sides of above equation and rearranging it yield
	\[
	(1-\tau)\,({P^{(\tau)}})^{-1/2}\Phi(P^{(\tau)})({P^{(\tau)}})^{-1/2}
	= \gamma^{(\tau)} I - \tfrac{\tau}{n}(P^{(\tau)})^{-1}.
	\]
	Hence, the eigenvalues of $({P^{(\tau)}})^{-1/2}\Phi(P^{(\tau)})({P^{(\tau)}})^{-1/2}$ are
	\[
	\lambda_i = \frac{\gamma^{(\tau)}}{1-\tau}
	- \frac{\tau}{n(1-\tau)}\,\mu_i,\quad i=1,\cdots,n,
	\]
	where $\mu_i$ is the corresponding eigenvalues of $(P^{(\tau)})^{-1}$. 
	Therefore, it holds
	\[
	L_\tau = \min_i \lambda_i
	= \frac{\gamma^{(\tau)}}{1-\tau} - \frac{\tau}{n(1-\tau)}\,\lambda_{\max}((P^{(\tau)})^{-1}),
	\]
	which proves part~(i).
	
	For part~(ii),  
	if $P^*\succ0$, then $\lambda_{\min}(P^{(\tau_k)})\to\lambda_{\min}(P^*)>0$ and hence
	$\lambda_{\max}\big((P^{(\tau_k)})^{-1}\big)$ stays bounded.  
	Since $\tau_k/(1-\tau_k)\to0$, it holds
	\[
	\Delta_{\tau_k}
	=\frac{\tau_k}{n(1-\tau_k)}\,\lambda_{\max}\big((P^{(\tau_k)})^{-1}\big)
	\longrightarrow 0 .
	\]
	We complete the proof.
	
\end{proof}

	The same result also implies gap closing on the original \(\rho\)-scale,
	because
	\[
	0\le \sqrt{U_\tau}-\sqrt{L_\tau}
	\le \sqrt{\Delta_\tau}.
	\]

\begin{remark}\label{rem:gap}
	Theorem~\ref{thm:gap} shows that the certified squared gap is governed by
	(\ref{gap-em}). Thus, the asymptotic tightness follows if
	\(\lambda_{\min}(P^{(\tau)})\) decays slower than \(\tau\). On the other
	hand, the regularization term in the fixed point equation (\ref{dam-eq}), together with
	the boundedness of \(\gamma^{(\tau)}\), prevents
	\(\lambda_{\min}(P^{(\tau)})\) from decaying faster than the order of
	\(\tau\). Therefore, a nonvanishing certified gap corresponds to the
	critical case where
	\[
	\lambda_{\min}(P^{(\tau)})\asymp \tau .
	\]
	This motivates the structural diagnosis below.
\end{remark}

Theorem~\ref{thm:gap} provides a sufficient condition for the certified
gap to close. For completeness, we next consider the complementary case
and identify necessary structural features of sequences of regularized
fixed points along which the certified gap does not vanish. 
Let $\{\tau_k\}_{k\ge1}\subset(0,1)$ with $\tau_k\downarrow0$.
For each \(k\), let \(P_k\) be a fixed point of
\(\widehat{\Phi}_{\tau_k}\), satisfying
$
\widehat{\Phi}_{\tau_k}(P_k)=P_k.
$
Denote the corresponding feedback gain and certified squared gap by
\[
K_k:=K(P_k),
\qquad
\Delta_k:=
\frac{\tau_k}
{n(1-\tau_k)\lambda_{\min}(P_k)},
\]
where \(K(P)\) is defined in~\eqref{eq-K}.
If
$
\liminf_{k\to\infty}\Delta_k>0,
$
then every limit point of \(\{P_k\}_{k\ge1}\) is singular. The additional
regularity conditions on the feedback gains, the eigenvalue scales of
\(P_k\), and the associated spectral subspaces are specified below.

\begin{definition}\label{regu}
	The pair sequence
	$
	\{(\tau_k,P_k)\}_{k\ge1}
	$
	introduced above is called a \emph{regular nonvanishing-gap fixed-point
		sequence}, if
	\[
	\liminf_{k\to\infty}\Delta_k>0,\;\{K_k\}_{k\ge1}\ \text{is bounded},
	\]
	and the eigenvalues of \(P_k\) have the following three-layer
	asymptotic structure. Specifically, let
	\[
	0<\lambda_{1,k}\le\cdots\le\lambda_{n,k}
	\]
	denote the eigenvalues of \(P_k\); there exist integers
	\(1\le r\le\ell\le n-1\) such that the following structural results of  eigenvalues hold.
	
	\emph{1) Critical layer of eigenvalues.}
	The first \(r\) eigenvalues are of the same order as \(\tau_k\):
	\[
	\lambda_{i,k}\asymp\tau_k,
	\quad
	i=1,\ldots,r.
	\]
	Equivalently, there exist constants \(0<c\leq C<\infty\) such that, for all sufficiently large \(k\),
	\[
	c\tau_k
	\le
	\lambda_{i,k}
	\le
	C\tau_k,
	\quad
	i=1,\ldots,r.
	\]
	
	\emph{2) Intermediate vanishing layer of eigenvalues.}
	If \(r<\ell\), the eigenvalues
	\(\lambda_{r+1,k},\ldots,\lambda_{\ell,k}\) converge to zero but decay
	more slowly than \(\tau_k\):
	\[
	\lim_{k\to\infty}\lambda_{i,k}=0,
	\quad
	\lim_{k\to\infty}
	\frac{\lambda_{i,k}}{\tau_k}
	=\infty,
	\quad
	i=r+1,\ldots,\ell.
	\]
	If \(r=\ell\), this layer is absent.
	
	\emph{3) Dominant layer of eigenvalues.}
	The remaining eigenvalues stay away from zero:
	\[
	\liminf_{k\to\infty}\lambda_{i,k}>0,
	\quad
	i=\ell+1,\ldots,n.
	\]
	
	\emph{4) Spectral subspace convergence}. Let \(\Pi_k^{\rm c}\) and \(\Pi_k^{\rm v}\) denote the spectral
	projection matrices of \(P_k\) associated with its first \(r\) and first
	\(\ell\) eigenvalues, respectively. It is further required that there
	exist orthogonal projection matrices \(\Pi^{\rm c}\) and \(\Pi^{\rm v}\) such
	that
	\[
	\lim_{k\to\infty}
	\|\Pi_k^{\rm c}-\Pi^{\rm c}\|=0,
	\qquad
	\lim_{k\to\infty}
	\|\Pi_k^{\rm v}-\Pi^{\rm v}\|=0.
	\]
\end{definition}
\begin{remark}
	The condition \(P^*\succ0\) in ii) of Theorem~\ref{thm:gap} means that
	all eigenvalues remain bounded away from zero. In the layer notation
	of Definition~\ref{regu}, this corresponds to the degenerate case
	\[
	r=\ell=0,\qquad q=n,
	\]
	and the certified gap closes. More generally, \eqref{gap-em} shows
	that
	\[
	\lim_{k\to\infty}\frac{\lambda_{\min}(P_k)}{\tau_k}=\infty
	\]
	also implies \(\lim_{k\to\infty}\Delta_k=0\). This corresponds to
	\(r=0\), while \(\ell\) may range from \(0\) to \(n-1\); hence, an intermediate vanishing layer may still be present. The spectral subspace convergence in
	Definition~\ref{regu} guarantees the existence of limiting critical and
	vanishing spectral subspaces, which are used to construct the orthogonal matrix \(T=[X\ Y\ Q]\) in
	Theorem~\ref{thm:regular-gap-structure} below.
\end{remark}
\medskip
For a matrix pair \((F,\bar F)\), define the squared mean-square spectral
radius by
\[
\varrho_{\rm ms}(F,\bar F)
:=
\operatorname{spr}\big(F\otimes F+\sigma^2\bar F\otimes\bar F\big).
\]

\begin{theorem}
	\label{thm:regular-gap-structure}
	Let \(\{(\tau_k,P_k)\}_{k\ge1}\) be a regular
	nonvanishing-gap fixed-point sequence. Then, there exist a subsequence
	$
	\{(\tau_{k_j},P_{k_j})\}_{j\ge1},
	$
	a feedback matrix \(K^*\) satisfying
	$
	\lim_{j\to\infty}K_{k_j}=K^*,
	$ an orthogonal matrix
	$
	T=[X\ Y\ Q],
	$
	and integers
	\[
	r\ge1,\; s\ge0,\; \ell\ge 1,\; q\ge1,\quad r+s+q=n,\; s=\ell-r,
	\]
	such that the closed-loop matrices
	$
	A_{K^*}=A-BK^*, \bar A_{K^*}=\bar A-\bar B K^*
	$
	admit the following simultaneous block upper triangular forms under
	the same orthogonal transformation \(T\):
	\[
	T^\top\! A_{K^*}\!T
	\!=\!
	\begin{pmatrix}
		A_c & A_{cb} & A_{cd}\\
		0 & A_b & A_{bd}\\
		0 & 0 & A_d
	\end{pmatrix},
	T^\top\! \bar A_{K^*}\!T
	\!=\!
	\begin{pmatrix}
		\bar A_c & \bar A_{cb} & \bar A_{cd}\\
		0 & \bar A_b & \bar A_{bd}\\
		0 & 0 & \bar A_d
	\end{pmatrix}.
	\]
	Here,
	$
	A_c,\bar A_c\in\mathbb R^{r\times r},
	A_b,\bar A_b\in\mathbb R^{s\times s},
	A_d,\bar A_d\in\mathbb R^{q\times q},
	$
	and
	$
	X\in\mathbb R^{n\times r},
	Y\in\mathbb R^{n\times s},
	Q\in\mathbb R^{n\times q}.
	$
	Accordingly, \((A_c,\bar A_c)\) and \((A_d,\bar A_d)\) are referred
	to as the critical diagonal block and the dominant diagonal block, respectively.
	Moreover, these two blocks exhibit a strict mean square spectral
	separation:
	\begin{equation}\label{eq:strict-ms-separation}
		\varrho_{\rm ms}(A_c,\bar A_c)
		<
		\varrho_{\rm ms}(A_d,\bar A_d).
	\end{equation}
\end{theorem}
The proof of this theorem is provided in Appendix~\ref{sec:large-gap-diagnosis}.

\begin{corollary}\label{up}
	For the subsequence
	\(\{(\tau_{k_j},P_{k_j})\}_{j\ge1}\) and the feedback matrix
	\(K^*\) given in
	Theorem~\ref{thm:regular-gap-structure}, it holds
	\[
	\lim_{j\to\infty}U_{\tau_{k_j}}
	=
	\varrho_{\rm ms}(A_d,\bar A_d),
	\]
	where \(U_\tau\) is defined in Theorem~\ref{thm:gap}.
\end{corollary}
\begin{proof}
	By the definition of \(U_\tau\),
	\[
	U_{\tau_{k_j}}
	=
	\frac{\gamma_{k_j}}{1-\tau_{k_j}}.
	\]
	Since
	\[
	\lim_{j\to\infty}\tau_{k_j}=0,
	\qquad
	\lim_{j\to\infty}\gamma_{k_j}=\gamma^*,
	\]
	we obtain
	\[
	\lim_{j\to\infty}U_{\tau_{k_j}}
	=
	\gamma^*.
	\]
	Moreover, relation~(R6) established in the proof of
	Theorem~\ref{thm:regular-gap-structure} gives
	\[
	\gamma^*
	=
	\varrho_{\rm ms}(A_d,\bar A_d).
	\]
	Therefore,
	\[
	\lim_{j\to\infty}U_{\tau_{k_j}}
	=
	\varrho_{\rm ms}(A_d,\bar A_d).
	\]
	This completes the proof.
\end{proof}
\medskip

For the subsequence \(\{(\tau_{k_j},P_{k_j})\}_{j\ge1}\) given in Theorem~\ref{thm:regular-gap-structure}, the proof yields \[ \lim_{j\to\infty}K_{k_j}=K^*, \qquad \lim_{j\to\infty}P_{k_j}=P^*. \]
The orthogonal matrix
\(T=[X\ Y\ Q]\) in Theorem~\ref{thm:regular-gap-structure} is constructed
in the proof from the limiting spectral subspaces of \(P_{k_j}\).
Specifically, the columns of \(X\) span the limiting spectral subspace
associated with the first \(r\) eigenvalues  of \(P_{k_j}\) as $\tau_{k_j}\downarrow0$, while the
columns of \([X\ Y]\) span the limiting spectral subspace
associated with its first \(\ell\) eigenvalues. Since the first \(\ell\) eigenvalues
converge to zero and the remaining eigenvalues are uniformly bounded away
from zero, we have
$
\ker P^*=\operatorname{Im}[X\ Y].
$
Moreover, the columns of \(Q\) span \((\ker P^*)^\perp\). Therefore, it holds
\[
\mathcal W\subseteq\mathcal V\subsetneq\mathbb R^n,~
\mathcal W:=\operatorname{Im}X,
~
\mathcal V:=\ker P^*=\operatorname{Im}[X\ Y].
\]
The construction in the proof also shows that both \(\mathcal W\) and
\(\mathcal V\) are invariant under \(A_{K^*}\) and
\(\bar A_{K^*}\). With respect to the ordered orthonormal matrix
\(T=[X\ Y\ Q]\), these nested invariance relations force all blocks
below the diagonal to vanish. This gives the simultaneous block upper
triangular forms of \(A_{K^*}\) and \(\bar A_{K^*}\) stated in
Theorem~\ref{thm:regular-gap-structure}.

Consequently, a necessary structural feature of any regular
nonvanishing-gap fixed-point sequence is the existence of a subsequence
\(\{(\tau_{k_j},P_{k_j})\}_{j\ge1}\) and a limiting feedback matrix
\(K^*\) such that
$
\lim_{j\to\infty}K_{k_j}=K^*,
$
and the corresponding closed-loop matrices \(A_{K^*}\) and
\(\bar A_{K^*}\) admit the simultaneous block upper triangular
decompositions stated in Theorem~\ref{thm:regular-gap-structure}. The
critical and dominant diagonal block pairs in these decompositions satisfy
the strict mean-square spectral separation in~\eqref{eq:strict-ms-separation}.
Furthermore, Corollary~\ref{up} shows that the dominant diagonal block
pair determines the limiting value of the squared upper bound $\lim_{j\to\infty}U_{\tau_{k_j}}$.

\subsection{Regularized Normalized Value Iteration }

Having established the existence of regularized fixed points in Theorem~\ref{thm-regularized}, we now turn to a practical procedure for computing them. 
We propose a \emph{regularized normalized value iteration} (RNVI), a fixed-point scheme for solving fixed-point problem $\widehat\Phi_\tau(P^{(\tau)})=P^{(\tau)}$. 

Given $\tau\in(0,1)$ and an initial $P_0\in\mathbb S^n_{++}$ with $\operatorname{Tr}P_0=1$, define the iteration
\[
P_{k+1}\ :=\ \widehat\Phi_\tau(P_k),\qquad k\in\mathbb{N}.
\]
We next derive explicit constants leading to a global contraction condition and, consequently, to linear convergence.
The next result bounds the Lipschitz constant of $\Phi$ on interior slices, which is the key step toward a contraction argument.

\begin{lemma}\label{lem:Lip-Phi}
	Let $a>0$ and define $c_R(a):=(a\,\lambda_{\min}(R_0))^{-1}$. 
	Then, for any $P,Q\in\mathcal S_a:=\{X\in\mathbb S^n_{++}: X\succeq aI,\ \operatorname{Tr}X=1\}$, it holds
	\begin{align*}
		\|\Phi(P)-\Phi(Q)\|\ \le\ L_\Phi(a)\,\|P-Q\|,
		\quad L_\Phi(a):=\alpha_A+2\alpha_S^2 c_R(a)+\alpha_S^2\alpha_R\,c_R(a)^2.
	\end{align*}
	Moreover,
	\[
	\sup_{P\in\mathcal S_a}\ \|\Phi(P)\|\ \le\ C_\Phi(a):=\alpha_A+\alpha_S^2\,c_R(a)
	\] holds.
\end{lemma}
\begin{proof}
	Let \(a>0\) and \(P,Q\in\mathcal S_a\). 
	Since $P\succeq aI$, we have
	\[
	R(P)\succeq aR_0 \quad\Rightarrow\quad \|R(P)^{-1}\|\le c_R(a),
	\]
	and the same holds for $R(Q)^{-1}$. 
	Moreover, it holds
	\begin{align*}
		\|S(P)-S(Q)\| \le\ \alpha_S\,\|P-Q\|, \quad \|R(P)-R(Q)\| \le\ \alpha_R\,\|P-Q\|,
	\end{align*}
	and $\|S(P)\|,\|S(Q)\|\le \alpha_S$ since $\operatorname{Tr}P=\operatorname{Tr}Q=1$.
	Decompose
	\begin{align*}
		&\Phi(P)-\Phi(Q)
		=A^\top(P-Q)A+\sigma^2\,\bar A^\top(P-Q)\bar A\quad-\Big(S(P)R(P)^{-1}S(P)^\top-S(Q)R(Q)^{-1}S(Q)^\top\Big).
	\end{align*}
	The affine part is bounded by $\alpha_A\|P-Q\|$. 
	For the nonlinear term, inserting and subtracting intermediate expressions yield
	\begin{align*}
		&\|S(P)R(P)^{-1}S(P)^\top-S(Q)R(Q)^{-1}S(Q)^\top\|\le\big(2\alpha_S^2 c_R(a)+\alpha_S^2\alpha_R c_R(a)^2\big)\|P-Q\|.
	\end{align*}
	Combining the two estimates gives the Lipschitz bound.
	Finally, for any $P\in\mathcal S_a$, it holds
	\begin{align*}
		&\|\Phi(P)\|
		\le \|A^\top P A+\sigma^2\bar A^\top P\bar A\|
		\quad+\|S(P)R(P)^{-1}S(P)^\top\|\le \alpha_A+\alpha_S^2 c_R(a),
	\end{align*}
	which completes the proof. 
\end{proof}

We can now control the Lipschitz constant of the regularized normalized operator $\widehat\Phi_\tau$.

\begin{theorem}\label{thm:RNVI-global}
	Fix $\tau\in(0,1)$ and let $\delta_\tau$ be as in (\ref{def-delta}).
	For all $P,Q\in
	\mathcal S_{\delta_\tau}=\{X\in\mathbb S^n_{++}:\ X\succeq \delta_\tau I,\ \mathrm{Tr}(X)=1\}
	$, it holds
	\[
	\|\widehat\Phi_\tau(P)-\widehat\Phi_\tau(Q)\|\ \le\ \Lambda(\tau)\,\|P-Q\|
	\]
	with the constant
	\begin{align*}
		\Lambda(\tau)\ :=&\ \frac{(1-\tau)\,L_\Phi(\delta_\tau)}{\tau}
	+
		\frac{(1-\tau)\,\big((1-\tau)\,C_\Phi(\delta_\tau)+\tfrac{\tau}{n}\big)\,n\,L_\Phi(\delta_\tau)}{\tau^2},
	\end{align*}
	where $L_\Phi$ and $C_\Phi$ are defined in Lemma $\ref{lem:Lip-Phi}$. If $\Lambda(\tau)<1$, the following statements hold.
	\begin{itemize}
		\item [i)] $\widehat\Phi_\tau$ is a contraction mapping on $(\mathcal S_{\delta_\tau},\|\cdot\|)$. 
		\item [ii)] There exists a \emph{unique} fixed point $P^{(\tau)}\in \mathcal S_{\delta_\tau}$ with $\widehat\Phi_\tau(P^{(\tau)})=P^{(\tau)}$.
		\item [iii)]  The sequence $(P_k)_{k\in\mathbb N}$ generated by RNVI 
		converges globally and linearly to $P^{(\tau)}$:
		\[
		\|P_{k}-P^{(\tau)}\|\ \le\ \Lambda(\tau)^k\,\|P_0-P^{(\tau)}\|,\qquad k\in \mathbb{N}.
		\]
	\end{itemize}
	
\end{theorem}
\begin{proof}
	Recall that
	\[
	\widehat\Phi_\tau(P)
	=\frac{Y_\tau(P)}{\operatorname{Tr}Y_\tau(P)},
	\qquad
	Y_\tau(P):=(1-\tau)\Phi(P)+\tfrac{\tau}{n}I .
	\]
	For any $P,Q\in\mathcal S_{\delta_\tau}$, a standard normalization estimate yields
	\begin{align*}
		&\|\widehat\Phi_\tau(P)-\widehat\Phi_\tau(Q)\|
		\le
		\frac{\|Y_\tau(P)-Y_\tau(Q)\|}{\nu}+\frac{\max\{\|Y_\tau(P)\|,\|Y_\tau(Q)\|\}}{\nu^2}
		\,|\operatorname{Tr}(Y_\tau(P)-Y_\tau(Q))|,
	\end{align*}
	where $\nu:=\inf_{X\in\mathcal S_{\delta_\tau}}\operatorname{Tr}Y_\tau(X)\ge\tau$.
	By Lemma~\ref{lem:Lip-Phi}, we know
	\begin{align*}
		\|Y_\tau(P)-Y_\tau(Q)\|
		&=(1-\tau)\|\Phi(P)-\Phi(Q)\|\le (1-\tau)L_\Phi(\delta_\tau)\|P-Q\|.
	\end{align*}
	Moreover, it holds
	\begin{align*}
		&|\operatorname{Tr}(Y_\tau(P)-Y_\tau(Q))|
		\le n\|Y_\tau(P)-Y_\tau(Q)\|,\quad\|Y_\tau(\cdot)\|
		\le (1-\tau)C_\Phi(\delta_\tau)+\tfrac{\tau}{n}.
	\end{align*}
	Combining these estimates gives
	\[
	\|\widehat\Phi_\tau(P)-\widehat\Phi_\tau(Q)\|
	\le \Lambda(\tau)\,\|P-Q\|
	\]
	with $\Lambda(\tau)$ as stated.
	If $\Lambda(\tau)<1$, Banach’s fixed-point theorem yields
	(i)--(iii). 
\end{proof}

\begin{lemma}\label{Lip-con}
	Fix a trace-normalized interior slice 
	\[
	S_a := \{\,X\in\mathbb S_{++}^n : X\succeq aI,\ \operatorname{Tr}X=1\,\},\quad a>0.
	\]
	For $\tau\in(0,1)$, let $\widehat\Phi_\tau$ be defined by~\eqref{Phi-tau}.
	Then, the mapping
	\[
	(P,\tau,H)\ \longmapsto\ D\widehat\Phi_\tau(P)[H]
	\]
	is jointly continuous on $S_a\times(0,1)\times\{H:\|H\|=1\}$.
	Consequently, the local Lipschitz modulus
	\[
	{\rm Lip}(P,\tau)
	:=\|D\widehat\Phi_\tau(P)\|
	=\sup_{\|H\|=1}\,\|D\widehat\Phi_\tau(P)[H]\|
	\]
	depends continuously on $(P,\tau)\in S_a\times(0,1)$.
\end{lemma}

\begin{proof}
	On $S_a$, $R(P)=B^\top P B+\sigma^2\bar B^\top P\bar B$ satisfies 
	$R(P)\succeq aR_0\succ0$; hence, $R(P)^{-1}$ depends continuously on $P$.
	Direct differentiation of~\eqref{eq-Phi} yields
	\begin{align*}
		\quad D\Phi(P)[H]
		&= A^\top H A+\sigma^2\bar A^\top H\bar A
		- (A^\top HB+\sigma^2\bar A^\top H\bar B)R(P)^{-1}S(P)^\top\\
		&- S(P) R(P)^{-1}(B^\top H A+\sigma^2\bar B^\top H\bar A) + S(P)R(P)^{-1}(B^\top H B+\sigma^2\bar B^\top H\bar B)R(P)^{-1}S(P)^\top,
	\end{align*}
	an affine function of $H$ whose coefficients depend continuously on $P$.
	Thus, $(P,H)\mapsto D\Phi(P)[H]$ is continuous.
	
	For $Y_\tau(P)=(1-\tau)\Phi(P)+\tfrac{\tau}{n}I$ and 
	$s_\tau(P)=\operatorname{Tr} Y_\tau(P)$, we have
	\begin{align*}
		D\widehat\Phi_\tau(P)[H]
		=& \frac{1-\tau}{s_\tau(P)}\Big(
		D\Phi(P)[H]- \widehat\Phi_\tau(P)\,\operatorname{Tr}(D\Phi(P)[H])
		\Big),
	\end{align*}
	where $(P,\tau)\mapsto s_\tau(P)$ and $(P,\tau)\mapsto \widehat\Phi_\tau(P)$ are continuous.
	Hence, the trivariate map 
	\[
	(P,\tau,H)\ \longmapsto\ D\widehat\Phi_\tau(P)[H]
	\]
	is continuous on $S_a\times(0,1)\times\{H:\|H\|=1\}$.
	
	For fixed $(P,\tau)$, $D\widehat\Phi_\tau(P)$ is linear in $H$.  
	Since the unit sphere $\{H:\|H\|=1\}$ is compact and the above trivariate map is continuous, 
	Berge's maximum theorem (see \cite{Berge63}, p.115, Theorem 1) implies that
	\[
	(P,\tau)\ \longmapsto\ 
	{\rm Lip}(P,\tau)
	= \sup_{\|H\|=1}\,\|D\widehat\Phi_\tau(P)[H]\|
	\]
	is continuous on $S_a\times(0,1)$. 
\end{proof}

The continuity of the local Lipschitz modulus established in Lemma~\ref{Lip-con} 
allows us to track the fixed points along a decreasing sequence of~$\tau$, 
which is the essence of the continuation method formalized below.
\begin{algorithm}[t]\label{alg:RNVI}
	\caption{Regularized Normalized Value Iteration}
	\begin{algorithmic}[1]
		\Statex \hspace{-\algorithmicindent} \textbf{Input:} System matrices $(A,\bar A,B,\bar B)$; noise $\omega_k\sim\mathcal{N}(0,\sigma^2)$; 
		initial $P_0\succ0$ with $\operatorname{Tr}(P_0)=1$;
		decreasing schedule $\tau_0>\tau_1>\cdots>\tau_M$;
		tolerance $\varepsilon$.
		\Statex \hspace{-\algorithmicindent}\textbf{Output:} The tightest certified bounds for the optimal stabilizing rate, together with 
		the parameters at which the best lower and upper bounds are attained and the gain associated with the best upper bound.
		
		\State $P \gets P_0$
		\State $\rho_{\mathrm{low}}^{\mathrm{best}}\gets 0$, 
		$\rho_{\mathrm{up}}^{\mathrm{best}}\gets +\infty$
		\State $\tau_{\mathrm{low}}\gets\texttt{None}$, 
		$\tau_{\mathrm{up}}\gets\texttt{None}$,
		$K_{\mathrm{up}}\gets\texttt{None}$
		
		\For{$j=0,1,\dots,M$}
		\State $\tau\gets\tau_j$, \quad $P^{(0)}\gets P$
		\For{$k=0,1,2,\dots$}
		\State $P^{(k+1)}\gets
		\dfrac{(1-\tau)\Phi(P^{(k)})+\tfrac{\tau}{n}I}
		{\operatorname{Tr}\!\big((1-\tau)\Phi(P^{(k)})+\tfrac{\tau}{n}I\big)}$
		\If{$\|P^{(k+1)}-P^{(k)}\|_F\le\varepsilon$}
		\State\textbf{break}
		\EndIf
		\EndFor
		\State $P^{(\tau)}\gets P^{(k+1)}$
		\State $K^{(\tau)}\gets R(P^{(\tau)})^{-1}S(P^{(\tau)})^\top$
		\State $\gamma^{(\tau)}\gets
		\operatorname{Tr}\!\big((1-\tau)\Phi(P^{(\tau)})+\tfrac{\tau}{n}I\big)$
		\State $L(P^{(\tau)})\gets
		\lambda_{\min}\big((P^{(\tau)})^{-1/2}
		\Phi(P^{(\tau)})(P^{(\tau)})^{-1/2}\big)$
		\State $\rho_{\mathrm{low}}\gets\sqrt{L(P^{(\tau)})}$, \quad 
		$\rho_{\mathrm{up}}\gets
		\sqrt{\tfrac{\gamma^{(\tau)}}{1-\tau}}$
		
		\If{$\rho_{\mathrm{low}}>\rho_{\mathrm{low}}^{\mathrm{best}}$}
		\State $\rho_{\mathrm{low}}^{\mathrm{best}}\gets\rho_{\mathrm{low}}$, \quad
		$\tau_{\mathrm{low}}\gets\tau$
		\EndIf
		\If{$\rho_{\mathrm{up}}<\rho_{\mathrm{up}}^{\mathrm{best}}$}
		\State $\rho_{\mathrm{up}}^{\mathrm{best}}\gets\rho_{\mathrm{up}}$, \quad
		$\tau_{\mathrm{up}}\gets\tau$, \quad
		$K_{\mathrm{up}}\gets K^{(\tau)}$
		\EndIf
		
		\State $P\gets P^{(\tau)}$ 
		\EndFor
		
		\State \textbf{Return:}
		$\rho_{\mathrm{fb}}^\ast\in
		[\rho_{\mathrm{low}}^{\mathrm{best}},
		\rho_{\mathrm{up}}^{\mathrm{best}}]$, \quad
		corresponding parameters $(\tau_{\mathrm{low}},\tau_{\mathrm{up}})$
		and gain $K_{\mathrm{up}}$.
	\end{algorithmic}
\end{algorithm}

	\begin{theorem}\label{thm:continuation}
		Suppose $\tau_{0}\in(0,1)$ satisfies $\Lambda(\tau_{0})<1$.  
		Then, there exists a strictly decreasing sequence
		\[
		\tau_{0}>\tau_{1}>\tau_{2}>\cdots
		\]
		with associated locally selected fixed points $\{P^{(\tau_j)}\}_{j\ge0}$
		of~\eqref{tau-fix}, such that RNVI with parameter $\tau_{j+1}$,
		initialized at $P^{(\tau_j)}$, converges locally to
		$P^{(\tau_{j+1})}$ for each $j$.
	\end{theorem}
	
	\begin{proof}
		Let $P_j:=P^{(\tau_j)}$. We argue inductively. For $j=0$, the assumption
		$\Lambda(\tau_0)<1$ gives ${\rm Lip}(P_0,\tau_0)<1$. Suppose that
		\[
		{\rm Lip}(P_j,\tau_j)
		=
		\|D\widehat\Phi_{\tau_j}(P_j)\|<1.
		\]
		Choose $q_j$ such that
		\[
		{\rm Lip}(P_j,\tau_j)<q_j<1.
		\]
		Since $P_j\succ0$, choose $a_j>0$ with
		$a_j<\lambda_{\min}(P_j)$, and then choose $r_j>0$ so small that
		\[
		r_j<\lambda_{\min}(P_j)-a_j.
		\]
		Define the relative closed ball
		\[
		\overline B(P_j,r_j)
		:=
		\{P\in\mathbb S^n:\operatorname{Tr}P=1,\ \|P-P_j\|\le r_j\}.
		\]
		Then, it holds
		\[
		\overline B(P_j,r_j)
		\subset
		S_{a_j}
		:=
		\{P\in\mathbb S^n_{++}:P\succeq a_jI,\ \operatorname{Tr}P=1\},
		\]
		because, for each $P\in\overline B(P_j,r_j)$,
		\[
		\lambda_{\min}(P)
		\ge
		\lambda_{\min}(P_j)-\|P-P_j\|
		>
		a_j.
		\]
		
		By Lemma~\ref{Lip-con}, after reducing $r_j$ if necessary, there exists
		$\varepsilon_{j,1}>0$ such that
		\[
		\|D\widehat\Phi_\tau(P)\|\le q_j,
		\qquad
		P\in\overline B(P_j,r_j),\quad
		\tau\in(\tau_j-\varepsilon_{j,1},\tau_j].
		\]
		Hence, for all such $\tau$ and all
		$P,Q\in\overline B(P_j,r_j)$,
		\[
		\|\widehat\Phi_\tau(P)-\widehat\Phi_\tau(Q)\|
		\le
		q_j\|P-Q\|.
		\]
		Thus $\widehat\Phi_\tau$ is a contraction on $\overline B(P_j,r_j)$.
		
		It remains to make this ball invariant. Since
		$\widehat\Phi_{\tau_j}(P_j)=P_j$ and $\widehat\Phi_\tau(P_j)$ is continuous
		in $\tau$, there exists $\varepsilon_{j,2}>0$ such that
		\[
		\|\widehat\Phi_\tau(P_j)-P_j\|
		\le
		(1-q_j)r_j,
		\qquad
		\tau\in(\tau_j-\varepsilon_{j,2},\tau_j].
		\]
		Set
		\[
		\varepsilon_j
		:=
		\min\{\varepsilon_{j,1},\varepsilon_{j,2},\tau_j\},
		\]
		and choose
		\[
		\tau_{j+1}\in(\tau_j-\varepsilon_j,\tau_j).
		\]
		Then, for each $P\in\overline B(P_j,r_j)$, it holds
		\[
		\begin{aligned}
			\|\widehat\Phi_{\tau_{j+1}}(P)-P_j\|
			&\le
			\|\widehat\Phi_{\tau_{j+1}}(P)
			-\widehat\Phi_{\tau_{j+1}}(P_j)\|+
			\|\widehat\Phi_{\tau_{j+1}}(P_j)-P_j\|\le
			q_j\|P-P_j\|+(1-q_j)r_j
			\le r_j.
		\end{aligned}
		\]
		Since $\widehat\Phi_{\tau_{j+1}}$ is trace normalized, it follows that
		\[
		\widehat\Phi_{\tau_{j+1}}\big(\overline B(P_j,r_j)\big)
		\subseteq
		\overline B(P_j,r_j).
		\]
		
		The ball \(\overline B(P_j,r_j)\) is complete under the spectral norm.
		Therefore, by the Banach fixed point theorem, the contraction self map
		\(\widehat\Phi_{\tau_{j+1}}\) has a unique fixed point
		\(P^{(\tau_{j+1})}\) in this ball, and RNVI initialized at \(P_j\) remains
		in the ball and converges to \(P^{(\tau_{j+1})}\). Since
		\(P^{(\tau_{j+1})}\in\overline B(P_j,r_j)\), the derivative bound above
		implies
		\[
		{\rm Lip}(P^{(\tau_{j+1})},\tau_{j+1})\le q_j<1.
		\]
		Thus the induction can be continued, and the theorem follows.	
	\end{proof}

Having established global convergence under a contraction condition and local convergence via continuation, we summarize the complete scheme in Algorithm~\textcolor{blue}{1}. 
The algorithm couples RNVI with the certified bounds for $\rho_{\mathrm{fb}}^\ast$ from Theorem~\ref{cor:certified}. 

	\begin{remark}
		For the bounds
		\(\underline{\rho}:=\rho_{\rm low}^{\rm best}\) and
		\(\overline{\rho}:=\rho_{\rm up}^{\rm best}\) returned by
		Algorithm~1 upon convergence of the fixed-point iterations, the corresponding gain \(K_{\rm up}\) satisfies
		\[
		\underline{\rho}
		\le \rho_{\mathrm{fb}}^\ast
		\le \rho(K_{\rm up})
		\le \overline{\rho},
		\]
		where \(\rho(K_{\rm up})\) denotes the stabilizing rate of the feedback
		control \(u_k=-K_{\rm up}x_k\). The inequalities follow, respectively,
		from the lower-bound certificate, the definition of \(\rho_{\mathrm{fb}}^\ast\), and
		Proposition~\ref{prop:J-upper}; hence $K_{\rm up}$  can  be regarded as a near-optimal gain.
	\end{remark}

\begin{remark}
	The fixed point of~\eqref{tau-fix} is known to exist by Brouwer’s
	fixed-point theorem, but that theorem is nonconstructive and does not
	provide a mechanism for obtaining the fixed point. The continuation result
	in Theorem~\ref{thm:continuation} addresses this issue by providing a
	problem-specific procedure that generates a decreasing sequence
	\(\{\tau_j\}_{j\geq 1}\) and associated fixed points
	\(\{P^{(\tau_j)}\}_{j\geq 1}\) through locally convergent RNVI iterations initialized
	from the preceding stage.
	The condition \(\Lambda(\tau)<1\) serves only as a sufficient global
		contraction certificate for obtaining a certified starting point of the
		continuation procedure, rather than a necessary condition for RNVI
		convergence. Since \(\Lambda(\tau)\to0\) as \(\tau\uparrow1\), there always
		exists \(\tau_0\) close to \(1\) such that
		\(\Lambda(\tau_0)<1\). The continuation steps do not require
		\(\Lambda(\tau_j)<1\) at every stage. Its limitation is that it does not
		guarantee that an arbitrarily prescribed target value of \(\tau\) can be
		reached along a continuation path. In numerical implementation, a
		prescribed decreasing grid is used as a practical continuation schedule.
		If RNVI fails to satisfy the stopping criterion at a trial value of
		\(\tau\), additional intermediate values can be inserted between this value
		and the last successfully computed one, and the continuation procedure is
		repeated. Only iterates satisfying the stopping criterion are used to
		evaluate the reported bounds.
\end{remark}

	\smallskip\noindent\textbf{Relation to Riccati theory and numerical methods.}
	Equation~\eqref{tau-fix} has a Riccati-type algebraic form, since the operator
	$\Phi$ resembles generalized Riccati operators arising in discrete-time
	stochastic LQ control with multiplicative noise.  However, its role here is
	fundamentally different from that of classical LQ Riccati equations.
	
	\textit{Classical Riccati methods.}
	In continuous- and discrete-time LQ control, algebraic Riccati equations
	characterize optimal regulators and stabilizing feedback gains under standard
	assumptions~(\cite{Anderson07}).  This structure enables Hamiltonian or
	symplectic methods~(\cite{Laub79}), Newton--Kleinman iterations~(\cite{feitzinger2009}),
	and projection/Krylov solvers~(\cite{bunse96}).
	
	\textit{Why do these methods not apply directly in our setting?}
	Despite its Riccati-type form, our problem is a Bellman-type \emph{nonlinear
		eigenvalue problem} $\Phi(P)=\gamma P$, and the regularization in~\eqref{tau-fix}
	is essential.  As a result, $\widehat\Phi_\tau(P)=P$ is not a standard ARE: there
	is no associated Hamiltonian matrix, and trace normalization destroys the
	Loewner-order monotonicity exploited by classical Riccati iterations.
	We therefore rely on the contraction framework in
	Theorem~\ref{thm:RNVI-global}, which provides a direct analytical foundation for
	RNVI.
	
	\textit{Future directions.}
	The order structure of $P\mapsto\Phi(P)$ suggests possible connections to
	nonlinear Perron--Frobenius theory and Hilbert’s projective metric,
	which may lead to stronger analytical guarantees beyond contractivity.

\section{Numerical Experiments}\label{sec:numerics}
We now demonstrate the proposed framework through two representative numerical examples. 
The first example reports a randomized evaluation over \(100\) generated
	systems and examines the tightness of the certified lower and upper bounds.
The second example considers a challenging system near the mean-square
stability boundary, for which it is not clear a priori whether the system
can be stabilized in the mean-square sense.

\subsection{Randomized evaluation of certified gaps}
\label{subsec:random-4x4}

\textbf{Setup.}
We evaluate the proposed certificates on \(100\) independently generated
four-dimensional stochastic systems
\[
x_{k+1}=(A+\omega_k\bar A)x_k+(B+\omega_k\bar B)u_k.
\]
In this experiment, we take
\[
n=4,\qquad m=2,\qquad \sigma=1 .
\]
The coefficient matrices \(A,\bar A,B,\bar B\) are randomly generated with
independent entries drawn from uniform distributions:
\[
A_{ij}\sim {\rm Unif}[-0.8,0.8],
\qquad
\bar A_{ij}\sim {\rm Unif}[-0.3,0.3],
\]
and
\[
B_{ij}\sim {\rm Unif}[-0.8,0.8],
\qquad
\bar B_{ij}\sim {\rm Unif}[-0.3,0.3].
\]
The random seed is fixed by \texttt{rng(0)}. We accept only samples for which
\(\operatorname{rank}[B^\top,\,\sigma\bar B^\top]^\top=m\) and
\(\sigma_{\min}([B^\top,\,\sigma\bar B^\top]^\top)\ge0.03\).

\textbf{Methodology.}
For each accepted instance, RNVI is run over a logarithmic grid
\[
\tau\in[5\times10^{-1},10^{-8}]
\]
with \(30\) grid points, using continuation in \(\tau\). We record the best
certified lower and upper bounds produced along this grid, and measure the
certified gap on the \(\rho\)-scale by
\[
{\rm gap}
=
\rho_{\rm up}^{\rm best}
-
\rho_{\rm low}^{\rm best}.
\]

\textbf{Overall results.}
All \(100\) instances successfully converge on the tested regularization
grid. Table~\ref{tab:all-100-gaps} reports the certified gaps of all
instances, with boxed entries marking gaps exceeding
\(10^{-2}\). The summary statistics are reported in
Table~\ref{tab:random-gap-statistics}.
\begin{table}[htbp]
	\centering
	\caption{Certified gaps of all \(100\) random instances}
	\label{tab:all-100-gaps}
	\footnotesize
	\setlength{\tabcolsep}{3pt}
	\resizebox{0.75\linewidth}{!}{
		\begin{tabular}{c c|c c|c c|c c}
			\toprule
			Inst. & Gap & Inst. & Gap & Inst. & Gap & Inst. & Gap \\
			\midrule
			\(1\) & \(\boxed{0.2018}\) & \(26\) & \(3.990{\times}10^{-7}\) & \(51\) & \(1.675{\times}10^{-7}\) & \(76\) & \(5.866{\times}10^{-7}\) \\
			\(2\) & \(3.027{\times}10^{-7}\) & \(27\) & \(7.039{\times}10^{-6}\) & \(52\) & \(1.164{\times}10^{-7}\) & \(77\) & \(4.742{\times}10^{-5}\) \\
			\(3\) & \(7.030{\times}10^{-6}\) & \(28\) & \(3.166{\times}10^{-7}\) & \(53\) & \(\boxed{0.5018}\) & \(78\) & \(2.581{\times}10^{-7}\) \\
			\(4\) & \(1.031{\times}10^{-7}\) & \(29\) & \(4.273{\times}10^{-5}\) & \(54\) & \(2.519{\times}10^{-7}\) & \(79\) & \(1.252{\times}10^{-7}\) \\
			\(5\) & \(4.012{\times}10^{-7}\) & \(30\) & \(1.651{\times}10^{-7}\) & \(55\) & \(4.233{\times}10^{-7}\) & \(80\) & \(2.600{\times}10^{-3}\) \\
			\(6\) & \(9.584{\times}10^{-8}\) & \(31\) & \(1.008{\times}10^{-6}\) & \(56\) & \(3.832{\times}10^{-6}\) & \(81\) & \(8.989{\times}10^{-6}\) \\
			\(7\) & \(6.756{\times}10^{-6}\) & \(32\) & \(1.567{\times}10^{-7}\) & \(57\) & \(1.356{\times}10^{-7}\) & \(82\) & \(1.305{\times}10^{-4}\) \\
			\(8\) & \(8.782{\times}10^{-8}\) & \(33\) & \(8.326{\times}10^{-7}\) & \(58\) & \(7.905{\times}10^{-8}\) & \(83\) & \(7.441{\times}10^{-8}\) \\
			\(9\) & \(2.763{\times}10^{-7}\) & \(34\) & \(3.687{\times}10^{-8}\) & \(59\) & \(7.432{\times}10^{-7}\) & \(84\) & \(\boxed{0.2884}\) \\
			\(10\) & \(8.275{\times}10^{-7}\) & \(35\) & \(2.309{\times}10^{-8}\) & \(60\) & \(4.793{\times}10^{-8}\) & \(85\) & \(1.155{\times}10^{-6}\) \\
			\(11\) & \(6.958{\times}10^{-6}\) & \(36\) & \(1.218{\times}10^{-7}\) & \(61\) & \(2.398{\times}10^{-7}\) & \(86\) & \(8.100{\times}10^{-8}\) \\
			\(12\) & \(2.161{\times}10^{-6}\) & \(37\) & \(6.019{\times}10^{-8}\) & \(62\) & \(3.513{\times}10^{-8}\) & \(87\) & \(3.349{\times}10^{-7}\) \\
			\(13\) & \(4.734{\times}10^{-7}\) & \(38\) & \(3.773{\times}10^{-7}\) & \(63\) & \(8.692{\times}10^{-8}\) & \(88\) & \(3.747{\times}10^{-8}\) \\
			\(14\) & \(1.255{\times}10^{-6}\) & \(39\) & \(2.015{\times}10^{-7}\) & \(64\) & \(2.932{\times}10^{-8}\) & \(89\) & \(2.736{\times}10^{-7}\) \\
			\(15\) & \(9.383{\times}10^{-8}\) & \(40\) & \(1.057{\times}10^{-7}\) & \(65\) & \(1.275{\times}10^{-5}\) & \(90\) & \(5.776{\times}10^{-7}\) \\
			\(16\) & \(8.260{\times}10^{-8}\) & \(41\) & \(4.765{\times}10^{-5}\) & \(66\) & \(3.619{\times}10^{-8}\) & \(91\) & \(3.700{\times}10^{-5}\) \\
			\(17\) & \(2.169{\times}10^{-8}\) & \(42\) & \(7.379{\times}10^{-7}\) & \(67\) & \(2.649{\times}10^{-8}\) & \(92\) & \(1.162{\times}10^{-6}\) \\
			\(18\) & \(5.669{\times}10^{-7}\) & \(43\) & \(1.924{\times}10^{-7}\) & \(68\) & \(1.908{\times}10^{-7}\) & \(93\) & \(2.041{\times}10^{-7}\) \\
			\(19\) & \(1.086{\times}10^{-5}\) & \(44\) & \(2.053{\times}10^{-6}\) & \(69\) & \(4.781{\times}10^{-8}\) & \(94\) & \(1.250{\times}10^{-7}\) \\
			\(20\) & \(1.929{\times}10^{-7}\) & \(45\) & \(6.385{\times}10^{-4}\) & \(70\) & \(3.317{\times}10^{-7}\) & \(95\) & \(1.850{\times}10^{-7}\) \\
			\(21\) & \(1.792{\times}10^{-7}\) & \(46\) & \(7.379{\times}10^{-8}\) & \(71\) & \(7.108{\times}10^{-8}\) & \(96\) & \(1.741{\times}10^{-7}\) \\
			\(22\) & \(9.576{\times}10^{-7}\) & \(47\) & \(1.048{\times}10^{-5}\) & \(72\) & \(2.479{\times}10^{-6}\) & \(97\) & \(2.783{\times}10^{-7}\) \\
			\(23\) & \(3.462{\times}10^{-7}\) & \(48\) & \(9.908{\times}10^{-7}\) & \(73\) & \(9.887{\times}10^{-8}\) & \(98\) & \(1.311{\times}10^{-5}\) \\
			\(24\) & \(4.282{\times}10^{-6}\) & \(49\) & \(3.280{\times}10^{-5}\) & \(74\) & \(1.450{\times}10^{-6}\) & \(99\) & \(3.953{\times}10^{-7}\) \\
			\(25\) & \(1.102{\times}10^{-7}\) & \(50\) & \(4.721{\times}10^{-8}\) & \(75\) & \(9.465{\times}10^{-8}\) & \(100\) & \(6.709{\times}10^{-5}\) \\
			\bottomrule
	\end{tabular}}
\end{table}
\begin{table}[htbp]
	\centering
	\caption{Certified gap statistics over \(100\) random instances}
	\label{tab:random-gap-statistics}
	\begin{tabular}{l c}
		\toprule
		Statistic & Value\\
		\midrule
		Successful instances & \(100/100\)\\
		Median gap & \(3.0968\times10^{-7}\)\\
		Mean gap & \(9.9576\times10^{-3}\)\\
		\(90\%\) quantile & \(3.7573\times10^{-5}\)\\
		\(95\%\) quantile & \(1.5591\times10^{-4}\)\\
		Maximum gap & \(5.0181\times10^{-1}\)\\
		\bottomrule
	\end{tabular}
\end{table}
The results show that the two certificates are tight for most of the
generated systems: \(97\) out of \(100\) instances have certified gaps below
\(10^{-2}\). The median gap is \(3.0968\times10^{-7}\), and the \(95\%\)
quantile is \(1.5591\times10^{-4}\), indicating that the certified intervals
are nearly closed in most cases. The larger mean reflects the influence of
the three outliers, namely instances \(1\), \(53\), and \(84\), whose gaps
exceed \(10^{-2}\). Their coefficient matrices and corresponding certified
bounds are reported together in
Table~\ref{tab:large-gap-instances}.
\begin{table*}[t]
	\centering
	\caption{Coefficient matrices and certified bounds for the three
		instances with large gaps}
	\label{tab:large-gap-instances}
	\scriptsize
	\setlength{\tabcolsep}{3pt}
	\renewcommand{\arraystretch}{1.05}
	\resizebox{\textwidth}{!}{
		\begin{tabular}{c|cccc|cccc|cc|cc|c|c|c}
			\toprule
			Inst.
			& \multicolumn{4}{c}{\(A\)}
			& \multicolumn{4}{c}{\(\bar A\)}
			& \multicolumn{2}{c}{\(B\)}
			& \multicolumn{2}{c}{\(\bar B\)}
			& \(\rho_{\rm low}^{\rm best}\)
			& \(\rho_{\rm up}^{\rm best}\)
			& Gap\\
			\midrule
			
			\multirow{4}{*}{\(1\)}
			& 0.5036 & 0.2118 & 0.7320 & 0.7315
			& -0.0469 & 0.0934 & 0.1072 & 0.0933
			& -0.3569 & 0.3117
			& -0.0368 & -0.1879
			& \multirow{4}{*}{\(0.1219\)}
			& \multirow{4}{*}{\(0.3237\)}
			& \multirow{4}{*}{\(0.2018\)}\\
			& 0.6493 & -0.6439 & 0.7438 & -0.0234
			& 0.2494 & -0.2786 & 0.1546 & -0.1973
			& -0.7261 & -0.2926
			& -0.0711 & -0.0061
			& & &\\
			& -0.5968 & -0.3544 & -0.5478 & 0.4804
			& 0.1753 & 0.2095 & 0.1459 & 0.1236
			& -0.6446 & 0.7204
			& 0.1593 & -0.0326
			& & &\\
			& 0.6614 & 0.0750 & 0.7529 & -0.5730
			& 0.2757 & 0.2604 & -0.0647 & -0.2809
			& 0.5175 & -0.7449
			& 0.1771 & 0.0878
			& & &\\
			\midrule
			
			\multirow{4}{*}{\(53\)}
			& 0.4432 & 0.5692 & 0.3264 & 0.5487
			& -0.2992 & -0.2863 & 0.2557 & -0.0957
			& -0.1851 & -0.0422
			& -0.0318 & -0.0375
			& \multirow{4}{*}{\(0.2582\)}
			& \multirow{4}{*}{\(0.7600\)}
			& \multirow{4}{*}{\(0.5018\)}\\
			& 0.0030 & 0.2733 & -0.1894 & 0.6381
			& -0.2981 & -0.0455 & -0.1209 & -0.2171
			& 0.3131 & 0.7195
			& 0.0525 & 0.1477
			& & &\\
			& -0.1192 & 0.0377 & 0.1083 & 0.7024
			& -0.2475 & -0.0954 & -0.0971 & 0.0047
			& 0.2046 & -0.6664
			& 0.2266 & -0.0193
			& & &\\
			& 0.1780 & -0.3219 & 0.6206 & 0.5047
			& -0.1436 & 0.0248 & 0.2157 & 0.2140
			& -0.0794 & -0.3523
			& -0.0185 & 0.2165
			& & &\\
			\midrule
			
			\multirow{4}{*}{\(84\)}
			& -0.7796 & -0.6373 & -0.2225 & -0.7366
			& -0.0438 & 0.0242 & 0.2561 & 0.2983
			& -0.6315 & -0.6332
			& -0.1078 & -0.1870
			& \multirow{4}{*}{\(0.3578\)}
			& \multirow{4}{*}{\(0.6462\)}
			& \multirow{4}{*}{\(0.2884\)}\\
			& -0.1965 & -0.7372 & 0.2307 & -0.0490
			& 0.2732 & 0.1233 & -0.2950 & -0.1634
			& -0.3709 & -0.0484
			& 0.2145 & 0.1555
			& & &\\
			& -0.5315 & 0.6932 & -0.6913 & -0.5598
			& 0.1345 & -0.2970 & 0.1948 & 0.2517
			& 0.4222 & -0.4495
			& -0.1441 & -0.2810
			& & &\\
			& 0.0644 & 0.7545 & -0.4673 & 0.7861
			& 0.0485 & 0.1695 & 0.1604 & 0.0852
			& 0.4888 & 0.6763
			& 0.2268 & 0.0854
			& & &\\
			\bottomrule
	\end{tabular}}
\end{table*}
For all three large-gap instances, both the best lower bound and the best
upper bound are attained at the smallest tested value \(\tau=10^{-8}\).

\textbf{Diagnosis of the large gap instances.}
We next examine whether the three outliers are numerically consistent with
the conditions in Definition~\ref{regu} and the
necessary structural conclusions in
Theorem~\ref{thm:regular-gap-structure}. The diagnosis uses the last six
tested values of \(\tau\), ending at \(\tau=10^{-8}\).

\textbf{1) Numerical assessment of the regularity conditions.}
We next examine whether the numerical results for instances
\(1\), \(53\), and \(84\) are consistent with the conditions in
Definition~\ref{regu}. The diagnosis uses the last six tested values
\[
2.1252\times10^{-7},\quad
1.1533\times10^{-7},\quad
6.2581\times10^{-8},\quad
3.3960\times10^{-8},\quad
1.8428\times10^{-8},\quad
10^{-8}.
\]

\paragraph{(a) Nonvanishing certified gaps.}
For all three selected instances, both
\(\rho_{\rm low}^{\rm best}\) and \(\rho_{\rm up}^{\rm best}\) are attained
at the smallest tested value \(\tau=10^{-8}\), while the resulting gaps
remain \(0.2018\), \(0.5018\), and \(0.2884\), respectively. Thus, the
computed tail does not exhibit a closing of the certified gap. This
provides finite-grid evidence for the nonvanishing gap condition in
Definition~\ref{regu}.

\paragraph{(b) Multiscale eigenvalue structure.}
Let
\[
0<\lambda_1(P^{(\tau)})\le\cdots\le\lambda_4(P^{(\tau)})
\]
be the eigenvalues of \(P^{(\tau)}\). For each eigenvalue, we fit the tail
data according to
\[
\log\lambda_i(P^{(\tau)})
\approx
a_i+\alpha_i\log\tau,
\qquad i=1,\ldots,4.
\]
A fitted exponent close to \(1\) indicates a critical eigenvalue at the
\(\tau\) scale, whereas an exponent close to \(0\) is consistent with a
nonvanishing \(O(1)\) scale.

Table~\ref{tab:regular_branch_eigenvalues} reports the fitted exponents
together with the ratios of the first two eigenvalues to \(\tau\) at
\(\tau=10^{-8}\).

\begin{table}[htbp]
	\centering
	\caption{Eigenvalue scale diagnosis for the three large-gap instances}
	\label{tab:regular_branch_eigenvalues}
	\small
	\setlength{\tabcolsep}{4pt}
	\begin{tabular}{c|cccc|cc}
		\hline
		Instance
		& \(\alpha_1\)
		& \(\alpha_2\)
		& \(\alpha_3\)
		& \(\alpha_4\)
		& \(\lambda_1/\tau\)
		& \(\lambda_2/\tau\)\\
		\hline
		\(1\)
		& \(0.99952\)
		& \(0.99909\)
		& \(5.6056\times10^{-4}\)
		& \(-1.4496\times10^{-5}\)
		& \(2.7803\)
		& \(10.208\)\\
		\(53\)
		& \(1.00000\)
		& \(1.00000\)
		& \(1.1193\times10^{-6}\)
		& \(-2.7490\times10^{-7}\)
		& \(0.48926\)
		& \(0.68931\)\\
		\(84\)
		& \(0.99999\)
		& \(0.99996\)
		& \(-9.0151\times10^{-8}\)
		& \(-2.7034\times10^{-7}\)
		& \(0.86332\)
		& \(3.1837\)\\
		\hline
	\end{tabular}
\end{table}

In every instance, the first two exponents are close to \(1\), while the
last two are close to \(0\). Moreover, the ratios
\(\lambda_i(P^{(\tau)})/\tau\), \(i=1,2\), remain finite and positive over
the tested tail. The numerical eigenvalue structure is therefore
consistent with
\[
\lambda_1(P^{(\tau)}),\lambda_2(P^{(\tau)})\asymp\tau,
\]
while the remaining two eigenvalues stay on a nonvanishing scale. Hence,
all three instances have the detected dimensions
\[
r=\ell=2,\qquad s=0,\qquad q=2.
\]
In particular, there is no intermediate vanishing layer.

Figure~\ref{fig:large_gap_eig_over_tau} displays the eigenvalues after
division by \(\tau\). The first two normalized eigenvalues approach stable
positive levels over the tested tail, whereas the last two grow after this
normalization, further supporting the above classification into two
distinct scales.

\begin{figure}[htbp]
	\centering
	\begin{subfigure}{0.32\textwidth}
		\centering
		\includegraphics[width=\textwidth]
		{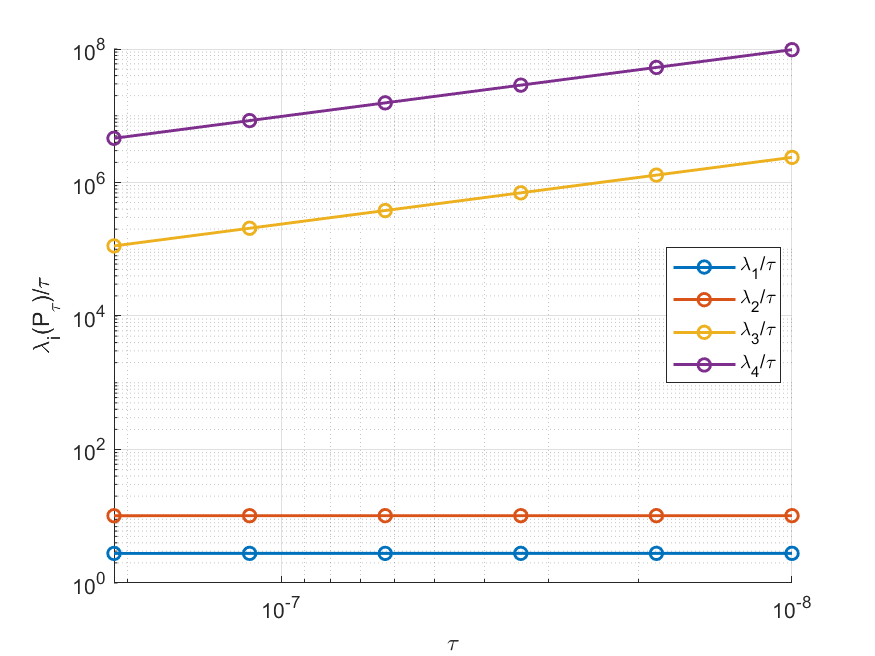}
		\caption{Instance \(1\)}
	\end{subfigure}
	\begin{subfigure}{0.32\textwidth}
		\centering
		\includegraphics[width=\textwidth]
		{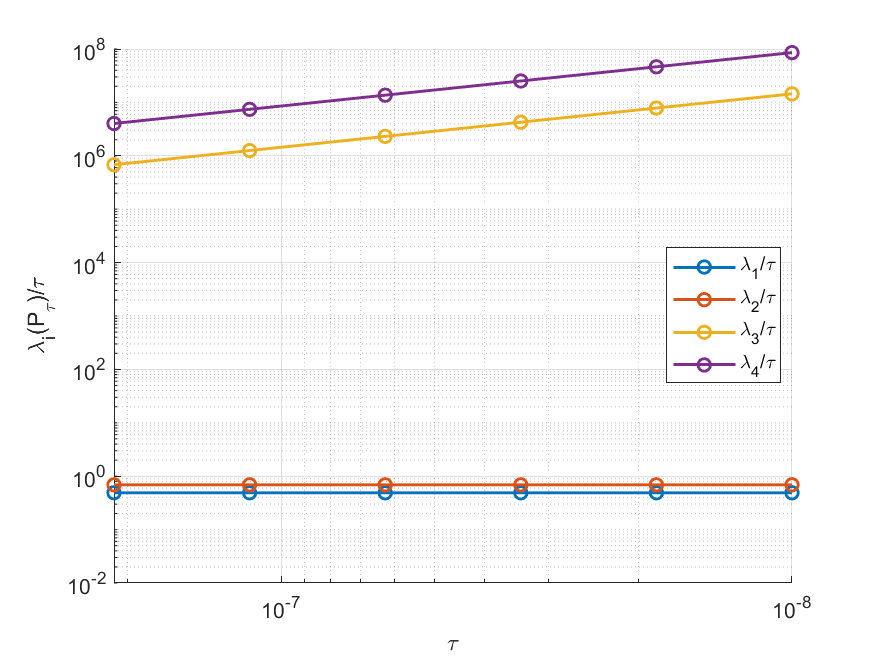}
		\caption{Instance \(53\)}
	\end{subfigure}
	\begin{subfigure}{0.32\textwidth}
		\centering
		\includegraphics[width=\textwidth]
		{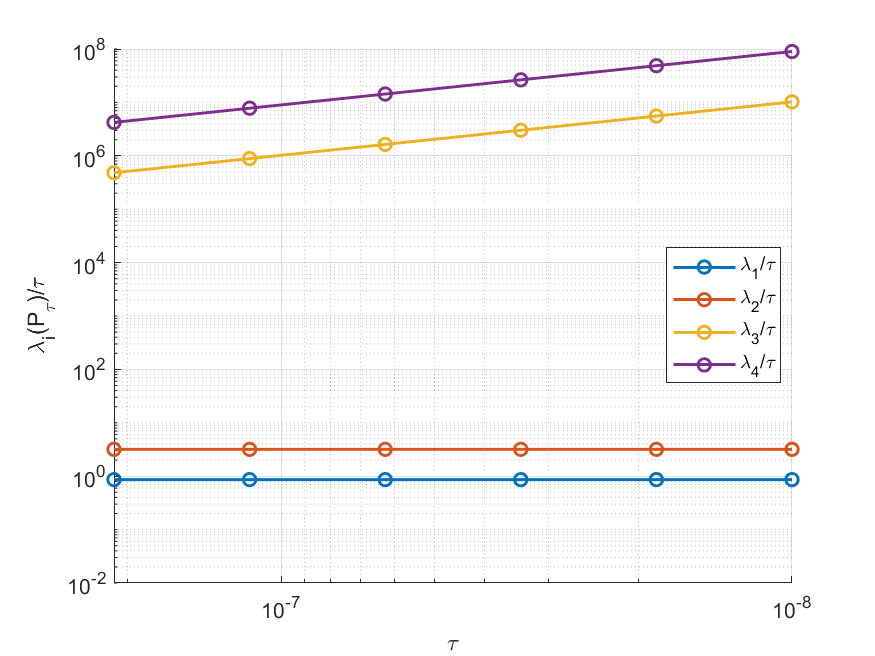}
		\caption{Instance \(84\)}
	\end{subfigure}
	\caption{Eigenvalues of \(P^{(\tau)}\) normalized by \(\tau\) over the
		tested tail.}
	\label{fig:large_gap_eig_over_tau}
\end{figure}

\paragraph{(c) Stability of the spectral projections and feedback matrices.}
Let the columns of \(X_\tau\) form an orthonormal basis of the critical
eigenspace corresponding to the first \(r=2\) eigenvalues, and let
\[
\Pi_{c,\tau}:=X_\tau X_\tau^\top.
\]
We measure the changes between adjacent tail points by
\[
D_{\Pi,k}
:=
\|\Pi_{c,\tau_{k+1}}-\Pi_{c,\tau_k}\|_F,
\qquad
D_{K,k}
:=
\|K_{\tau_{k+1}}-K_{\tau_k}\|_F,
\]
where
\[
K_\tau:=K(P^{(\tau)}).
\]
Since \(r=\ell=2\), the critical and overall vanishing layers coincide.
Consequently,
\[
\Pi_{c,\tau}=\Pi_{v,\tau},
\]
and \(D_{\Pi,k}\) simultaneously assesses the stability of both spectral
projection sequences appearing in Definition~\ref{regu}.
The resulting projection and feedback differences are summarized in
Table~\ref{tab:regular_branch_convergence}. The maximum projection
difference is reported to show the scale of the variation over the whole
tested tail, while the last differences describe the behavior nearest
the smallest tested value of \(\tau\).

\begin{table}[htbp]
	\centering
	\caption{Stability of the critical projection and feedback matrix}
	\label{tab:regular_branch_convergence}
	\small
	\setlength{\tabcolsep}{5pt}
	\begin{tabular}{c|c|c|c}
		\hline
		Instance
		& \(\max_k D_{\Pi,k}\)
		& \(D_{\Pi,\mathrm{last}}\)
		& \(D_{K,\mathrm{last}}\)\\
		\hline
		\(1\)
		& \(1.218\times10^{-4}\)
		& \(1.063\times10^{-5}\)
		& \(5.107\times10^{-5}\)\\
		\(53\)
		& \(3.062\times10^{-7}\)
		& \(2.655\times10^{-8}\)
		& \(7.975\times10^{-7}\)\\
		\(84\)
		& \(4.691\times10^{-6}\)
		& \(4.068\times10^{-7}\)
		& \(7.835\times10^{-7}\)\\
		\hline
	\end{tabular}
\end{table}

For all three instances, both \(D_{\Pi,k}\) and \(D_{K,k}\) decrease along
the tested tail. Thus, the critical eigenspaces do not exhibit persistent
rotation, and the feedback matrices become increasingly stable as
\(\tau\) decreases. These observations provide numerical support for the
spectral projection condition and the boundedness of the feedback gains
in Definition~\ref{regu}. Figure~\ref{fig:projection_diagnostics} shows
these trends in detail. For completeness, it also reports
\[
D_{P,k}
:=
\|P^{(\tau_{k+1})}-P^{(\tau_k)}\|_F,
\]
which provides supplementary evidence that the regularized fixed points
themselves become increasingly stable along the tested tail.

\begin{figure}[htbp]
	\centering
	\begin{subfigure}{0.32\textwidth}
		\centering
		\includegraphics[width=\textwidth]
		{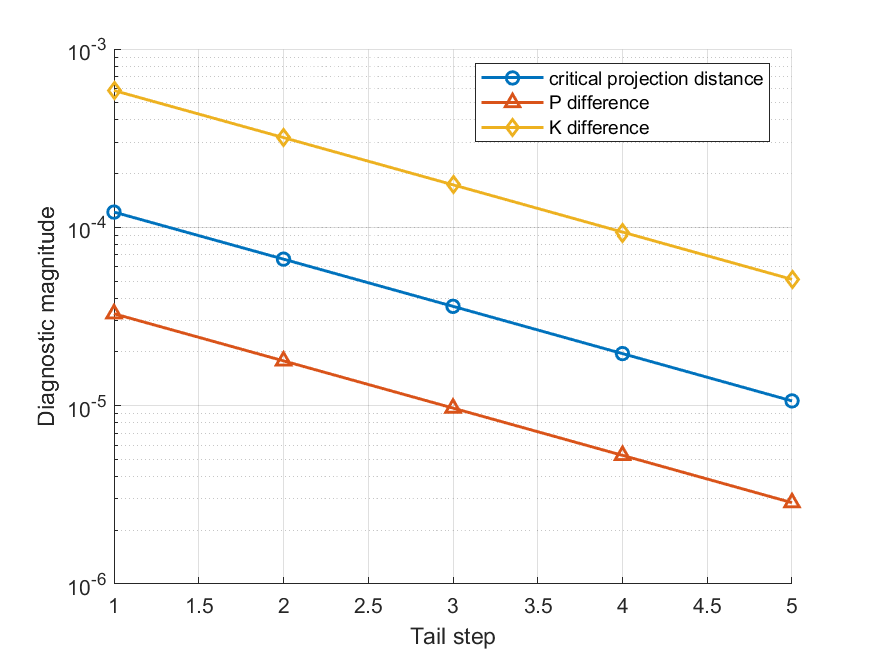}
		\caption{Instance \(1\)}
	\end{subfigure}
	\begin{subfigure}{0.32\textwidth}
		\centering
		\includegraphics[width=\textwidth]
		{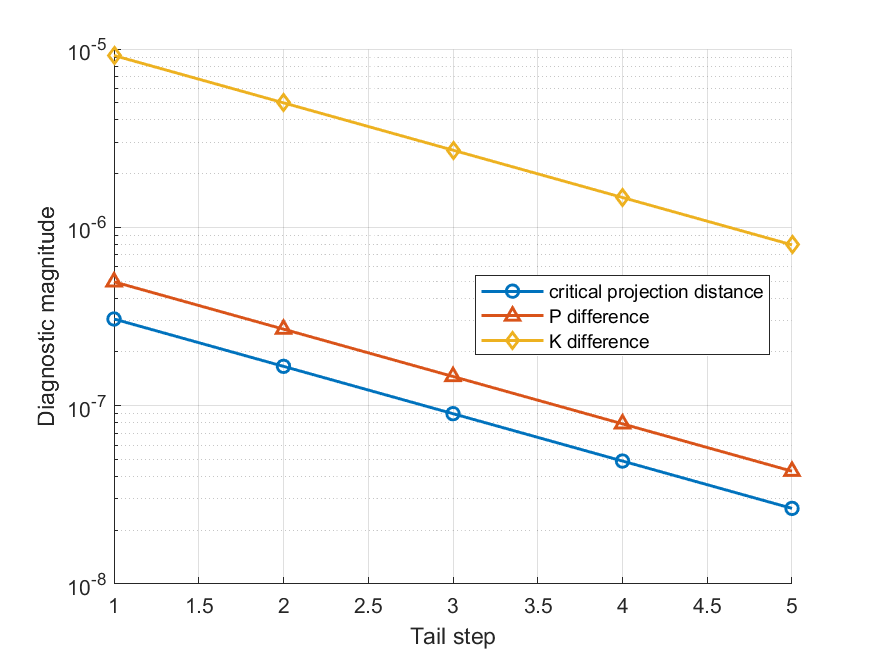}
		\caption{Instance \(53\)}
	\end{subfigure}
	\begin{subfigure}{0.32\textwidth}
		\centering
		\includegraphics[width=\textwidth]
		{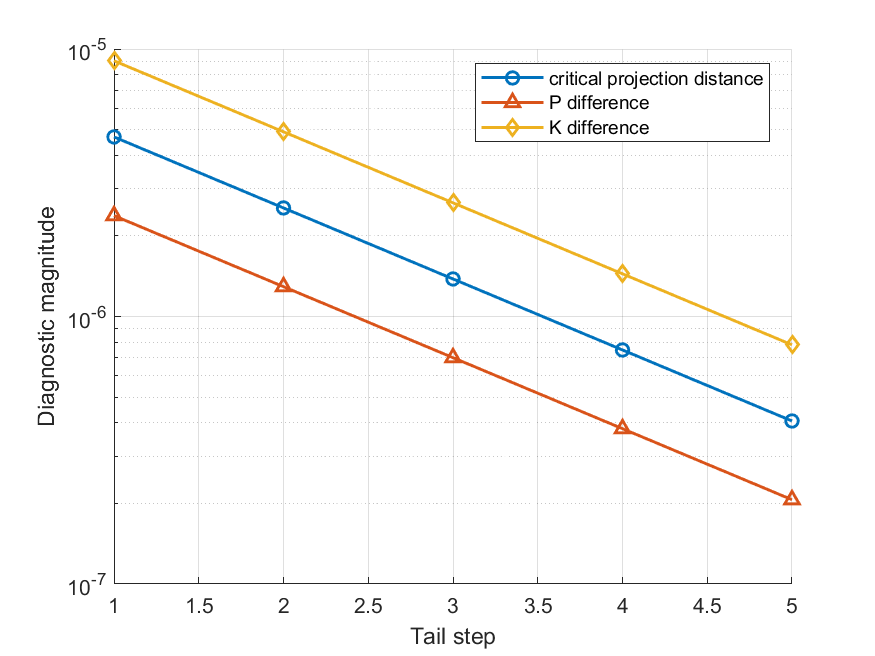}
		\caption{Instance \(84\)}
	\end{subfigure}
	\caption{Successive differences of the critical spectral projection,
		regularized fixed point, and feedback matrix over the tested tail.}
	\label{fig:projection_diagnostics}
\end{figure}

Combining the persistent terminal gaps, the eigenvalue structure with two
distinct scales, and the decreasing projection and feedback differences,
the numerical results for the three instances are consistent with the
conditions in Definition~\ref{regu}. This assessment concerns only the
conditions in that definition. The necessary invariant subspace, block
decomposition, and spectral separation conclusions are examined
separately in the next part.

\textbf{2) Necessary structural evidence.}
Having obtained numerical evidence for the regular  assumptions, we
next examine the necessary structural conclusions of
Theorem~\ref{thm:regular-gap-structure}. The calculations in this
subsection are performed at the smallest tested value
\[
\tau=10^{-8}.
\]
Let the columns of \(X\) form an orthonormal basis of the critical
eigenspace of \(P^{(\tau)}\). Since \(r=\ell=2\), the critical and
vanishing eigenspaces coincide, and there is no intermediate vanishing layer.
Thus, the columns of \(Q\) form an orthonormal basis of the orthogonal
complement of this common eigenspace, and the orthogonal matrix in the
theorem takes the form
\[
T=[X\ Q].
\]
Define
\[
K_\tau:=K(P^{(\tau)}),\qquad
A_\tau:=A-BK_\tau,\qquad
\bar A_\tau:=\bar A-\bar B K_\tau.
\]

\paragraph{(a) Simultaneous block upper triangular structure.}
To examine the simultaneous block upper triangular structure, we compute
the lower left blocks of the two closed-loop matrices under the orthogonal change of coordinates defined by \[T=[X\ Q].\]
Specifically, define
\[
R_A
:=
\frac{\|(I-XX^\top)A_\tau X\|_F}
{\|A_\tau X\|_F},
\qquad
R_{\bar A}
:=
\frac{\|(I-XX^\top)\bar A_\tau X\|_F}
{\|\bar A_\tau X\|_F},
\]
and set
\[
R_{\rm inv}:=\max\{R_A,R_{\bar A}\}.
\]
Since
\[
I-XX^\top=QQ^\top,
\]
the two numerators satisfy
\[
\|(I-XX^\top)A_\tau X\|_F
=
\|Q^\top A_\tau X\|_F
\]
and
\[
\|(I-XX^\top)\bar A_\tau X\|_F
=
\|Q^\top\bar A_\tau X\|_F,
\]
respectively. Thus, they are the Frobenius norms of the lower left
blocks of the two transformed closed loop matrices. Under the orthogonal matrix \(T=[X\ Q]\), we have
\[
T^\top A_\tau T
=
\begin{pmatrix}
	X^\top A_\tau X & X^\top A_\tau Q\\
	Q^\top A_\tau X & Q^\top A_\tau Q
\end{pmatrix}
=
\begin{pmatrix}
	A_c & X^\top A_\tau Q\\
	Q^\top A_\tau X & A_d
\end{pmatrix},
\]
and
\[
T^\top \bar A_\tau T
=
\begin{pmatrix}
	X^\top\bar A_\tau X & X^\top\bar A_\tau Q\\
	Q^\top\bar A_\tau X & Q^\top\bar A_\tau Q
\end{pmatrix}
=
\begin{pmatrix}
	\bar A_c & X^\top\bar A_\tau Q\\
	Q^\top\bar A_\tau X & \bar A_d
\end{pmatrix},
\]
where
\[
A_c:=X^\top A_\tau X,\qquad
\bar A_c:=X^\top\bar A_\tau X,
\]
and
\[
A_d:=Q^\top A_\tau Q,\qquad
\bar A_d:=Q^\top\bar A_\tau Q.
\]In particular,
\[
R_A=R_{\bar A}=0
\]
would imply
\[
Q^\top A_\tau X=0,
\qquad
Q^\top \bar A_\tau X=0.
\]

\begin{table}[htbp]
	\centering
	\caption{Residuals of the lower left blocks at
		\(\tau=10^{-8}\)}
	\label{tab:critical_invariance}
	\begin{tabular}{c|c|c|c}
		\hline
		Instance
		& \(R_A\)
		& \(R_{\bar A}\)
		& \(R_{\rm inv}\)\\
		\hline
		\(1\)
		& \(1.421\times10^{-6}\)
		& \(1.073\times10^{-5}\)
		& \(1.073\times10^{-5}\)\\
		\(53\)
		& \(5.980\times10^{-9}\)
		& \(2.153\times10^{-7}\)
		& \(2.153\times10^{-7}\)\\
		\(84\)
		& \(1.181\times10^{-7}\)
		& \(6.402\times10^{-7}\)
		& \(6.402\times10^{-7}\)\\
		\hline
	\end{tabular}
\end{table}
\noindent The small values of \(R_A\) and \(R_{\bar A}\) in Table~\ref{tab:critical_invariance} imply that
\[
Q^\top A_\tau X\approx0,
\qquad
Q^\top\bar A_\tau X\approx0.
\]
Therefore,
\[
T^\top A_\tau T
\approx
\begin{pmatrix}
	A_c & X^\top A_\tau Q\\
	0 & A_d
\end{pmatrix},
\qquad
T^\top\bar A_\tau T
\approx
\begin{pmatrix}
	\bar A_c & X^\top\bar A_\tau Q\\
	0 & \bar A_d
\end{pmatrix}.
\]
This supports the approximate
simultaneous block upper triangular forms.  Note that extremely small as it is, the value of $\tau$ is non-zero; this should be the reason of the ``approximate"
simultaneous block upper triangular forms.

\paragraph{(b) Strict mean square spectral separation.}
We next compute the squared mean square spectral radii
\[
\varrho_c
:=
\varrho_{\rm ms}(A_c,\bar A_c),
\qquad
\varrho_d
:=
\varrho_{\rm ms}(A_d,\bar A_d).
\]
To assess the relation between the dominant block and the upper
certificate, we also report
\[
E_\gamma
:=
\left|\frac{\gamma^{(\tau)}}{1-\tau}-\varrho_d\right|
\]
and
\[
E_\rho
:=
\left|
\rho_{\rm up}(\tau)-\sqrt{\varrho_d}
\right|.
\]

\begin{table}[htbp]
	\centering
	\caption{Mean square spectral separation and upper certificate matching
		at \(\tau=10^{-8}\)}
	\label{tab:necessary_structure_diagnosis}
	\small
	\setlength{\tabcolsep}{4pt}
	\begin{tabular}{c|c|c|c|c|c}
		\hline
		Instance
		& \(\varrho_c\)
		& \(\varrho_d\)
		& \(\varrho_c/\varrho_d\)
		& \(E_\gamma\)
		& \(E_\rho\)\\
		\hline
		\(1\)
		& \(1.9061\times10^{-2}\)
		& \(1.0478\times10^{-1}\)
		& \(1.8192\times10^{-1}\)
		& \(4.133\times10^{-6}\)
		& \(6.385\times10^{-6}\)\\
		\(53\)
		& \(1.2019\times10^{-1}\)
		& \(5.7766\times10^{-1}\)
		& \(2.0807\times10^{-1}\)
		& \(8.860\times10^{-8}\)
		& \(6.209\times10^{-8}\)\\
		\(84\)
		& \(3.3363\times10^{-1}\)
		& \(4.1759\times10^{-1}\)
		& \(7.9894\times10^{-1}\)
		& \(2.182\times10^{-7}\)
		& \(1.721\times10^{-7}\)\\
		\hline
	\end{tabular}
\end{table}

Table~\ref{tab:necessary_structure_diagnosis} shows that all three
instances satisfy
\[
\varrho_c<\varrho_d.
\]
The corresponding ratios \(\varrho_c/\varrho_d\) are approximately
\(0.182\), \(0.208\), and \(0.799\), respectively, and are all strictly
smaller than one. Thus, all three instances exhibit the
critical and dominant block spectral separation required by
Theorem~\ref{thm:regular-gap-structure}.
Moreover, the small values of \(E_\gamma\) show that
\[
\gamma^{(\tau)}\approx\varrho_d.
\]
Since the best upper certificate for each instance is attained at
\(\tau=10^{-8}\), the small values of \(E_\rho\) also give
\[
\rho_{\rm up}^{\rm best}
=
\rho_{\rm up}(10^{-8})
\approx
\sqrt{\varrho_d}.
\]
This numerical matching is consistent with Corollary~\ref{up}, which
states that along a regular nonvanishing certified gap,
\[
\lim_{j\to\infty}\rho_{\rm up}(\tau_j)
=
\sqrt{\varrho_{\rm ms}(A_d,\bar A_d)}.
\]

Combining the small invariance residuals, the approximate simultaneous
block upper triangular forms, and the strict inequality
\[
\varrho_c<\varrho_d,
\]
the three large gap instances are numerically consistent with the
necessary structural conclusions of
Theorem~\ref{thm:regular-gap-structure}. The close agreement between the
dominant block spectral level and the upper certificate further supports
the limiting relation in Corollary~\ref{up}.

\subsection{Four-Dimensional System Near the Stability Boundary}
\label{subsec:4d-near-marginal}

This example consists of two complementary parts.  
First, we examine a single near-marginal instance for which it is not clear \emph{a priori} whether mean-square stabilization is possible; the RNVI framework confirms that this instance is indeed stabilizable and provides tight, verifiable bounds for its optimal stabilizing rate.   
Second, we perform a \emph{scaling experiment across the stability boundary} to study how the certified bounds evolve as the system gradually transitions from the strictly stable region to the unstable regime.

\subsubsection{Single Near-Marginal Instance }
\label{subsubsec:4d-single}

\textbf{Setup.}
We consider a four-dimensional stochastic system
with $\sigma=2$,
\begin{align*}
	&A=\begin{bmatrix}
		0.9999 & 0.34   & 0      & 0\\
		0      & 0.9996 & 0.25   & 0\\
		0      & 0      & 0.9992 & 0.22\\
		0      & 0      & 0      & 0.9988
	\end{bmatrix},\\[0.7em]
	&\qquad\;\bar A=\begin{bmatrix}
		0.16 & 0.06 & 0    & 0\\
		0.05 & 0.13 & 0.05 & 0\\
		0    & 0.04 & 0.11 & 0.05\\
		0    & 0    & 0.03 & 0.10
	\end{bmatrix},\\[0.4em]
	&B=\begin{bmatrix}
		0.0024 & 0\\
		0      & 0.05\\
		0.22   & 0\\
		0      & 0.14
	\end{bmatrix},\;
	\bar B=\begin{bmatrix}
		0.375 & 0\\
		0     & 0.25\\
		0.15  & 0\\
		0     & 0.15
	\end{bmatrix}.
\end{align*}
\begin{figure}[!htb]
	\centering
	\begin{subfigure}[t]{0.48\linewidth}
		\centering
		\includegraphics[width=1.0\linewidth]{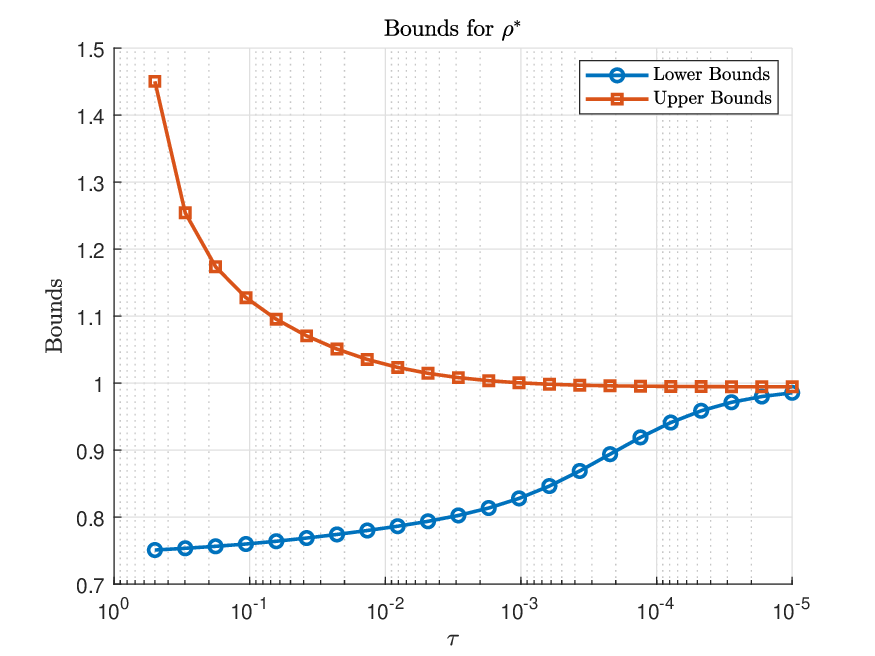}
		\caption{Certified bounds for $\rho_{\mathrm{fb}}^\ast$ over the $\tau$-grid.}
		\label{fig:4d-bounds}
	\end{subfigure}
	\hfill
	\begin{subfigure}[t]{0.48\linewidth}
		\centering
		\includegraphics[width=1.0\linewidth]{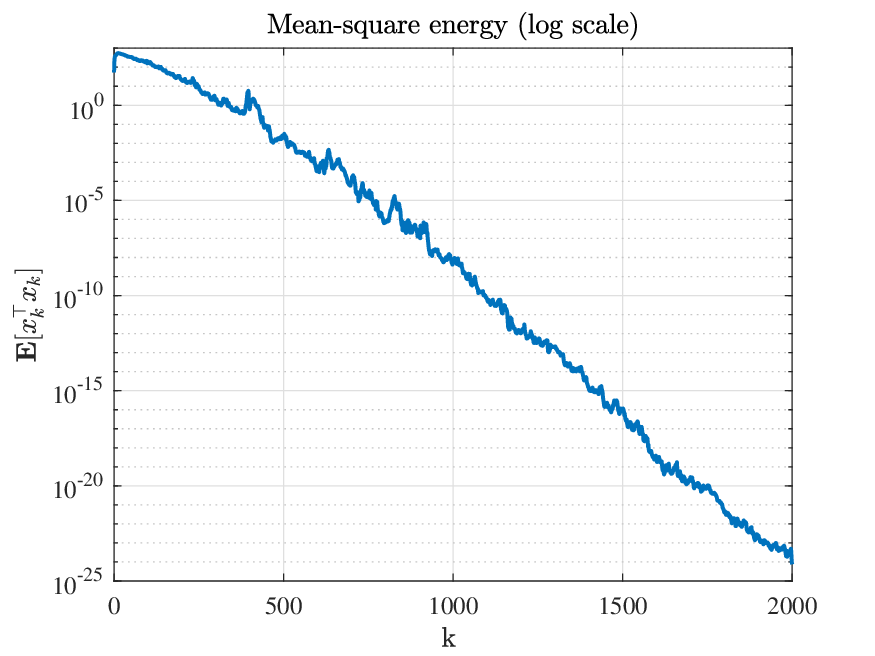}
		\caption{Mean-square energy trajectory.}
		\label{fig:4d-ms-energy}
	\end{subfigure}
	\caption{Behavior of RNVI 
		for the near-marginally stable system.}
	\label{fig:4d-certified-and-ms}
\end{figure}

Stabilizing this system is challenging, due to the following features.
\begin{enumerate}
	\item  The noise-free, uncontrolled part of the system already lies at the boundary of stability, since $A$ has eigenvalues very close to one, and even weak multiplicative noise can readily push the overall dynamics into an unstable regime.
	\item  The input matrices $B$ and $\bar B$ contain many zero entries. Several state coordinates, especially those associated with the weakest modes of $A$, receive very limited actuation.
	As a result, the controller has little direct authority over exactly the states that require the most stabilization effort.
	\item  Multiplicative noise is injected precisely in these weakly controlled directions, with additional coupling through $\bar A$. 
\end{enumerate}
Together, these effects make it entirely possible that this system is not mean-square stabilizable, thus the optimal stabilizing rate $\rho_{\mathrm{fb}}^\ast$ is not known a priori to be smaller than one and might in fact be arbitrarily close to or even exceed one.

\textbf{Results.}
For a decreasing logarithmic grid of $\tau\in[0.5,\,10^{-5}]$, 
Figure~\ref{fig:4d-certified-and-ms}(a) shows the bounds for $\rho_{\mathrm{fb}}^\ast$ across~$\tau$ and the best bounds are
\[
\rho_{\mathrm{fb}}^\ast\in
\big[0.9856,\, 0.9946\big],
\]
achieved at 
$
\tau_{\mathrm{low}} =
\tau_{\mathrm{up}} = 1\times10^{-5}
$
which yields the near-optimal feedback stabilizing gain
\[
K_{\mathrm{up}}
=
\begin{bmatrix}
	0.8553 & 2.2097&1.5573&-0.9092\\
	-0.7399 & -1.5910&-0.3794&1.8438
\end{bmatrix}.
\]

To further illustrate the closed-loop behavior under this near-optimal gain, 
we simulate the mean-square evolution of the system with deterministic initial condition $x_0=(5,4,3,2)^\top$ and $u_k=-K_{\mathrm{up}}x_k$.  
The quantity $\mathbb{E}[x_k^\top x_k]$ is approximated by averaging over $10^4$ i.i.d.\ trajectories.  
The resulting curve on a logarithmic scale (Figure \ref{fig:4d-certified-and-ms}(b)) exhibits a clear linear trend with negative slope, demonstrating exponential mean-square decay under the RNVI derived gain.

\textbf{Discussion.}
This example highlights the effectiveness of our method in a particularly
challenging setting: although the system is a priori close to being
unstabilizable, our certified bounds establish that $\rho_{\mathrm{fb}}^\ast<1$, thereby
confirming mean-square stabilizability, and the RNVI framework further
produces a near-optimal gain $K_{\mathrm{up}}$ that achieves exponential
decay in practice.

\subsubsection{Scaling Study Across the Stability Boundary}
\label{subsubsec:4d-scaling}

\textbf{Objective and Methodology.}
To stress-test RNVI beyond a single operating point, we smoothly vary the system across the marginal regime.
We introduce a global scaling factor $\theta\in[0.95,1.10]$ and consider $A_\theta=\theta A$, while keeping $(\bar A,B,\bar B,\sigma)$ fixed.
This radial scaling moves the eigenvalues of $A$ with respect to the unit circle: $\theta<1$ pushes them inward, $\theta=1$ recovers the near-marginal instance, and $\theta>1$ moves the  dynamics beyond the boundary.
For each $\theta$, we run RNVI over a decreasing grid $\tau \in[0.5,\,10^{-5}]$ and report the tightest \emph{certified bounds} for $\rho_{\mathrm{fb}}^\ast$.

\textbf{Results.}
Table~\ref{tab:gamma-bounds} summarizes the bounds for $\rho_{\mathrm{fb}}^\ast$ at $20$ values of $\theta$.
For a stabilizable case such as $\theta_1=1.005$, the tightest certified upper
bound on $\rho_{\mathrm{fb}}^\ast$ is attained at $\tau=10^{-5}$, yielding the near-optimal feedback stabilizing gain
\[
K_{1}=
\begin{bmatrix}
	0.8558 & 2.2192 & 1.5695 & -0.9108\\
	-0.7380 & -1.5940 & -0.3847 & 1.8506
\end{bmatrix}.
\]
In contrast, for a nearby parameter value $\theta_2=1.021$, where the certified
interval satisfies $\rho_{\mathrm{fb}}^\ast>1$ and hence stabilization is impossible, the
RNVI procedure returns the gain
\[
K_{2}=
\begin{bmatrix}
	0.8570 & 2.2475 & 1.6058 & -0.9154\\
	-0.7322 & -1.6023 & -0.4000 & 1.8707
\end{bmatrix}.
\]
\begin{table}[H]
	\centering
	\caption{Certified bounds and gaps for \(\rho_{\mathrm{fb}}^\ast\) under global
		scaling \(A_\theta\) in the four-dimensional system}
	\label{tab:gamma-bounds}
	\resizebox{0.75\textwidth}{!}{
		\begin{tabular}{c c c@{\hskip 0.35cm}|c c c}
			\hline
			$\theta$ & Bounds for $\rho_{\mathrm{fb}}^\ast$ & Gap
			& $\theta$ & Bounds for $\rho_{\mathrm{fb}}^\ast$ & Gap\\
			\hline
			0.950 & [0.9360, 0.9492] & 0.0132
			& 1.029 & [1.0132, 1.0208] & 0.0076\\
			0.958 & [0.9441, 0.9564] & 0.0123
			& 1.036 & [1.0206, 1.0279] & 0.0073\\
			0.966 & [0.9521, 0.9635] & 0.0114
			& 1.045 & [1.0281, 1.0351] & 0.0070\\
			0.974 & [0.9599, 0.9707] & 0.0108
			& 1.053 & [1.0355, 1.0422] & 0.0067\\
			0.981 & [0.9677, 0.9778] & 0.0101
			& 1.061 & [1.0429, 1.0494] & 0.0065\\
			0.989 & [0.9754, 0.9850] & 0.0096
			& 1.068 & [1.0503, 1.0566] & 0.0063\\
			0.997 & [0.9831, 0.9922] & 0.0091
			& 1.076 & [1.0577, 1.0637] & 0.0060\\
			1.005 & [0.9907, 0.9993] & 0.0086
			& 1.084 & [1.0650, 1.0709] & 0.0059\\
			1.013 & [0.9982, 1.0065] & 0.0083
			& 1.092 & [1.0723, 1.0780] & 0.0057\\
			1.021 & [1.0057, 1.0136] & 0.0079
			& 1.100 & [1.0797, 1.0852] & 0.0055\\
			\hline
		\end{tabular}
	}
\end{table}

\textbf{Discussion.}
Three points emerge clearly from the scaling study.
\begin{enumerate}
	\item \textbf{Localization of the stability threshold.}
	For \(\theta\leq1.005\), the certified intervals lie below one and
	certify mean-square stabilizability. For \(\theta\geq1.021\), both
	bounds exceed one, ruling out mean-square stabilizability and
	quantifying the minimal achievable exponential growth rate. Hence, the
	stability threshold is localized near \(\theta\approx1.013\); see
	Tab.~\ref{tab:gamma-bounds}.
	
	\item \textbf{Robustness near and beyond the threshold.}
	The certified intervals vary smoothly and become slightly narrower as
	\(\theta\) crosses the stability threshold, indicating that the RNVI
	certificates remain numerically well behaved near and beyond the loss
	of stabilizability.
	
	\item \textbf{Certification and controller synthesis.}
	For each \(\theta\), RNVI provides two-sided bounds on \(\rho_{\mathrm{fb}}^\ast\)
	and an associated feedback gain, such as \(K_1\) at \(\theta=1.005\)
	and \(K_2\) at \(\theta=1.021\).
\end{enumerate}

\section{Conclusion}\label{sec:conclusion}
We develop a quantitative framework for analyzing mean-square exponential
stabilization of stochastic systems with multiplicative noise. The study
distinguishes the general optimal stabilizing rate \(\rho^\ast\) from the feedback-restricted optimal stabilizing rate \(\rho_{\mathrm{fb}}^\ast\). By extending the
norm-based analysis to the stochastic setting, we establish quantitative
bounds for the general optimal stabilizing rate. Furthermore, within the
state-feedback policy class, we introduce an optimal control formulation for
Problem (OSR), derive a Bellman-type characterization, and reduce it to a
nonlinear matrix eigenvalue problem. To address the possible lack of
strictly positive-definite solutions, we introduce a regularized
approximation scheme and develop the RNVI algorithm. The resulting
regularized fixed points generate computable feedback gains and certified
two-sided bounds for the feedback-restricted optimal stabilizing rate. We also analyze the gap
between the certified bounds, including a sufficient condition for gap
closing and a necessary structural condition for nonvanishing gaps.

These results provide both theoretical insights and practical tools for
quantitative stabilization analysis under multiplicative noise. Future work
may extend the framework toward data-driven formulations, robustness against
modeling errors, and scalable algorithms for high-dimensional systems.

\bibliographystyle{plainnat}
\bibliography{JiaHui}

\appendix

\section{Proof of the Structural Condition for Nonvanishing Certified Gaps}
\label{sec:large-gap-diagnosis}

\setcounter{lemma}{0}
\setcounter{proposition}{0}
\setcounter{theorem}{0}

\renewcommand{\thelemma}{A.\arabic{lemma}}
\renewcommand{\theproposition}{A.\arabic{proposition}}
\renewcommand{\thetheorem}{A.\arabic{theorem}}

This appendix provides the proof of Theorem~\ref{thm:regular-gap-structure}.
The proof relies on several auxiliary results concerning the limiting
behavior of the regularized fixed-point sequence and the induced spectral
structure. We first establish these intermediate results and then combine
them to prove the theorem.

\subsection*{A.1 A general invariant-subspace property}

We first characterize a weaker but more general necessary structural
consequence that follows directly from a nonvanishing certified gap.
Unlike the stronger block separation result developed later, this result
does not require the additional regularity conditions introduced below
and only identifies the existence of a common invariant subspace.

For each \(\tau\in(0,1)\), define
\[
K_\tau:=K(P^{(\tau)}),\qquad
A_\tau:=A-BK_\tau,\qquad
\bar A_\tau:=\bar A-\bar B K_\tau,\qquad
\gamma_\tau:=\gamma^{(\tau)},
\]
where \(K(P^{(\tau)})\) denotes the feedback matrix induced by
\(P^{(\tau)}\) defined in (\ref{eq-K}).

\begin{proposition}[Common Invariant Subspace Structure]\label{pre-pro}
	Assume that there exists a sequence
	\(\{\tau_k\}_{k\ge1}\subset(0,1)\) such that
	\[
	\lim_{k\to\infty}\tau_k=0,
	\qquad
	\liminf_{k\to\infty}\Delta_{\tau_k}>0.
	\]
	Let
	\[
	P_k:=P^{(\tau_k)},
	\qquad
	K_k:=K_{\tau_k}.
	\]
	If the feedback sequence \(\{K_k\}_{k\ge1}\) is bounded, then there
	exist a subsequence
	\[
	\{(\tau_{k_j},P_{k_j})\}_{j\ge1},
	\]
	matrices \(P^*\succeq0\) and \(K^*\), and a nonzero proper subspace
	\[
	\mathcal V:=\ker P^*\subsetneq\mathbb R^n
	\]
	such that
	\[
	\lim_{j\to\infty}P_{k_j}=P^*,
	\qquad
	\lim_{j\to\infty}K_{k_j}=K^*,
	\]
	and
	\[
	(A-BK^*)\mathcal V\subseteq\mathcal V,
	\qquad
	(\bar A-\bar B K^*)\mathcal V\subseteq\mathcal V.
	\]
	In other words, under boundedness of the feedback gains, a
	nonvanishing certified gap implies that a limiting closed-loop matrix
	pair along a subsequence possesses a nontrivial common invariant
	subspace.
\end{proposition}

\begin{proof}
	Since
	\[
	\liminf_{k\to\infty}\Delta_{\tau_k}>0,
	\]
	there exist \(c_0>0\) and \(k_0\ge1\) such that
	\[
	\Delta_{\tau_k}\ge c_0,
	\qquad k\ge k_0.
	\]
	Hence, the certified gap identity gives
	\[
	\lambda_{\min}(P_k)
	\le
	\frac{\tau_k}{n(1-\tau_k)c_0},
	\qquad k\ge k_0.
	\]
	It follows that
	\[
	\lim_{k\to\infty}\lambda_{\min}(P_k)=0.
	\]
	
	Moreover,
	\[
	P_k\succeq0,
	\qquad
	\operatorname{Tr}(P_k)=1,
	\]
	and hence \(\{P_k\}_{k\ge1}\) is contained in the compact set
	\[
	\{P\succeq0:\operatorname{Tr}(P)=1\}.
	\]
	The sequence \(\{K_k\}_{k\ge1}\) is bounded by assumption, and
	\(\{\gamma_{\tau_k}\}_{k\ge1}\) is bounded by the trace estimate in
	Theorem~\ref{thm-regularized}. Therefore, there exist a subsequence
	\[
	\{(\tau_{k_j},P_{k_j})\}_{j\ge1},
	\]
	matrices \(P^*\succeq0\) and \(K^*\), and a scalar \(\gamma^*\ge0\)
	such that
	\[
	\lim_{j\to\infty}P_{k_j}=P^*,
	\qquad
	\lim_{j\to\infty}K_{k_j}=K^*,
	\qquad
	\lim_{j\to\infty}\gamma_{\tau_{k_j}}=\gamma^*.
	\]
	In addition,
	\[
	\operatorname{Tr}(P^*)=1.
	\]
	By continuity of the smallest eigenvalue,
	\[
	\lambda_{\min}(P^*)=0.
	\]
	Consequently,
	\[
	P^*\neq0,
	\qquad
	\ker P^*\neq\{0\}.
	\]
	
	Define
	\[
	A^*:=A-BK^*,
	\qquad
	\bar A^*:=\bar A-\bar B K^*.
	\]
	The regularized fixed-point equation~(\ref{tau-fix}) can be written as
	\[
	\gamma_{\tau_{k_j}}P_{k_j}
	=
	(1-\tau_{k_j})
	\left(
	A_{k_j}^\top P_{k_j}A_{k_j}
	+\sigma^2\bar A_{k_j}^\top P_{k_j}\bar A_{k_j}
	\right)
	+\frac{\tau_{k_j}}{n}I,
	\]
	where
	\[
	A_{k_j}:=A-BK_{k_j},
	\qquad
	\bar A_{k_j}:=\bar A-\bar B K_{k_j}.
	\]
	Taking limits in this equation gives
	\[
	\gamma^*P^*
	=
	(A^*)^\top P^*A^*
	+\sigma^2(\bar A^*)^\top P^*\bar A^*.
	\]
	
	Let
	\[
	\mathcal V:=\ker P^*.
	\]
	Since \(P^*\neq0\) and \(\ker P^*\neq\{0\}\), \(\mathcal V\) is a
	nonzero proper subspace. For any \(x\in\mathcal V\), we have
	\(P^*x=0\). Multiplying the limiting equation by \(x^\top\) and \(x\)
	gives
	\[
	0
	=
	(A^*x)^\top P^*(A^*x)
	+\sigma^2(\bar A^*x)^\top P^*(\bar A^*x).
	\]
	Since \(P^*\succeq0\), both terms are nonnegative. Therefore,
	\[
	(A^*x)^\top P^*(A^*x)=0,
	\qquad
	(\bar A^*x)^\top P^*(\bar A^*x)=0.
	\]
	For a positive semidefinite matrix, \(y^\top P^*y=0\) implies
	\(P^*y=0\). Hence,
	\[
	P^*A^*x=0,
	\qquad
	P^*\bar A^*x=0,
	\]
	which gives
	\[
	A^*x\in\ker P^*,
	\qquad
	\bar A^*x\in\ker P^*.
	\]
	Since \(x\in\mathcal V\) is arbitrary,
	\[
	A^*\mathcal V\subseteq\mathcal V,
	\qquad
	\bar A^*\mathcal V\subseteq\mathcal V.
	\]
	Equivalently,
	\[
	(A-BK^*)\mathcal V\subseteq\mathcal V,
	\qquad
	(\bar A-\bar B K^*)\mathcal V\subseteq\mathcal V.
	\]
	This proves the proposition.
\end{proof}

\begin{remark}
	The above proposition provides a weaker but more general necessary
	structural condition for nonvanishing certified gaps under boundedness
	of the corresponding feedback gains. Here, the degenerate subspace is
	the kernel of the singular limiting matrix \(P^*\), namely,
	\(\ker P^*\), on which the limiting quadratic form
	\(x^\top P^*x\) vanishes. The proposition shows that this subspace is
	invariant under both limiting closed-loop matrices.
\end{remark}

However, the existence of such a common invariant subspace only identifies
the location of the degeneracy. For a regular nonvanishing-gap fixed-point
sequence, the block separation result below provides a sharper
characterization by showing that the critical and dominant blocks must
satisfy a strict mean-square spectral separation.

\subsection*{A.2 Proof of Theorem \ref{thm:regular-gap-structure}}

We first
introduce the following lemma.
\begin{lemma}\label{lem:positive-eigenmatrix-spr}
	Let
	$
	\mathcal L:\mathbb S^n\rightarrow\mathbb S^n
	$
	be the linear operator defined by
	$
	\mathcal L(P)
	=
	A^\top P A
	+
	\sigma^2\bar A^\top P\bar A .
	$
	Suppose that there exist a scalar \(\gamma\ge0\) and a matrix
	\(P\succ0\) such that
	$
	\mathcal L(P)=\gamma P .
	$
	Then, it holds
	$
	\operatorname{spr}(\mathcal L)=\gamma .
	$
\end{lemma}

\begin{proof}
	For the given matrix \(P\succ0\), and for any \(X\in\mathbb S^n\),
	there exists \(c>0\) such that
	\[
	-cP\preceq X\preceq cP.
	\]
	Indeed, since \(P\succ0\), the symmetric matrix
	\(P^{-1/2}XP^{-1/2}\) has bounded eigenvalues.
	Define
	\[
	\|X\|_{P}
	:=
	\inf\{c>0:\ -cP\preceq X\preceq cP\},
	\qquad X\in\mathbb S^n .
	\]
	Moreover, from
	\[
	-cP\preceq X\preceq cP
	\quad\Longleftrightarrow\quad
	-cI\preceq P^{-1/2}XP^{-1/2}\preceq cI,
	\]
	it is easy to verify that
	\[
	\|X\|_{P}
	=
	\big\|P^{-1/2}XP^{-1/2}\big\|_2 .
	\]
	Thus, \(\|\cdot\|_P\) is a norm, with its norm properties inherited from
	the spectral norm.
	The operator \(\mathcal L\) is monotone with respect to the Loewner order.
	Therefore, whenever \(-cP\preceq X\preceq cP\), we have
	\[
	-c\mathcal L(P)
	\preceq
	\mathcal L(X)
	\preceq
	c\mathcal L(P).
	\]
	Using \(\mathcal L(P)=\gamma P\), we obtain
	\[
	-c\gamma P
	\preceq
	\mathcal L(X)
	\preceq
	c\gamma P .
	\]
	Thus, for each \(c>0\) satisfying
	\(-cP\preceq X\preceq cP\), it holds that
	\[
	\|\mathcal L(X)\|_{P}\le \gamma c .
	\]
	Taking the infimum over all such \(c\) yields
	\[
	\|\mathcal L(X)\|_{P}\le \gamma\|X\|_{P},
	\qquad X\in\mathbb S^n .
	\]
	Let \(\|\mathcal L\|_{P}\) be the operator norm induced by
	\(\|\cdot\|_P\), namely
	\[
	\|\mathcal L\|_{P}
	:=
	\sup_{X\ne0}
	\frac{\|\mathcal L(X)\|_{P}}{\|X\|_{P}} .
	\]
	The preceding inequality implies
	$
	\|\mathcal L\|_{P}\le \gamma .
	$
	By the standard bound of the spectral radius by any induced operator norm
	on a finite-dimensional linear space, we have
	\[
	\operatorname{spr}(\mathcal L)
	\le
	\|\mathcal L\|_{P}
	\le
	\gamma .
	\]
	On the other hand, \(\mathcal L(P)=\gamma P\) and \(P\ne0\), so
	\(\gamma\) is an eigenvalue of \(\mathcal L\). Hence, it holds
	\[
	\operatorname{spr}(\mathcal L)\ge |\gamma|.
	\]
	Combining the two inequalities yields
	$
	\operatorname{spr}(\mathcal L)=\gamma .
	$
\end{proof}
For any matrix pair \((F,\bar F)\), define its squared mean square
spectral radius by
\[
\varrho_{\rm ms}(F,\bar F)
:=
\operatorname{spr}
\left(
F\otimes F+\sigma^2\bar F\otimes\bar F
\right),
\]
We now prove Theorem~\ref{thm:regular-gap-structure}.
\begin{proof}
	
	\noindent\textbf{Step 1: Invariance of the critical layer.}
	
	For each \(k\), define
	\[
	\gamma_k
	:=
	\operatorname{Tr}\!\left(
	(1-\tau_k)\Phi(P_k)+\frac{\tau_k}{n}I
	\right).
	\]
	Since
	\[
	P_k\succeq0,
	\qquad
	\operatorname{Tr}(P_k)=1,
	\]
	the sequence \(\{P_k\}_{k\ge1}\) is contained in the compact set
	\[
	\left\{
	P\in\mathbb S_+^n:
	\operatorname{Tr}(P)=1
	\right\}.
	\]
	Moreover, \(\{K_k\}_{k\ge1}\) is bounded by
	Definition \ref{regu} of a regular nonvanishing-gap fixed-point sequence, and
	\(\{\gamma_k\}_{k\ge1}\) is bounded by the trace estimate in
	Theorem~\ref{thm-regularized}. Therefore, there exist a subsequence
	\[
	\{(\tau_{k_j},P_{k_j})\}_{j\ge1},
	\]
	matrices \(P^*\succeq0\) and \(K^*\), and a scalar
	\(\gamma^*\ge0\) such that
	\[
	\lim_{j\to\infty}P_{k_j}=P^*,
	\qquad
	\lim_{j\to\infty}K_{k_j}=K^*,
	\qquad
	\lim_{j\to\infty}\gamma_{k_j}=\gamma^*.
	\]
	Consequently,
	\[
	\lim_{j\to\infty}(A-BK_{k_j})=A_{K^*},
	\qquad
	\lim_{j\to\infty}(\bar A-\bar B K_{k_j})
	=
	\bar A_{K^*}.
	\]
	The regularized fixed-point equation is
	\[
	\gamma_{k_j}P_{k_j}
	=
	(1-\tau_{k_j})
	\left(
	A_{k_j}^{\top}P_{k_j}A_{k_j}
	+
	\sigma^2\bar A_{k_j}^{\top}P_{k_j}\bar A_{k_j}
	\right)
	+
	\frac{\tau_{k_j}}{n}I,
	\]
	where
	\[
	A_{k_j}:=A-BK_{k_j},
	\qquad
	\bar A_{k_j}:=\bar A-\bar B K_{k_j}.
	\]
	Taking limits gives
	\[
	\gamma^*P^*
	=
	A_{K^*}^{\top}P^*A_{K^*}
	+
	\sigma^2\bar A_{K^*}^{\top}P^*\bar A_{K^*}.
	\]
	By the preceding Proposition \ref{pre-pro}, the degenerate subspace
	\[
	\mathcal V:=\ker P^*
	\]
	is a nonzero proper subspace invariant under both limiting closed-loop
	matrices:
	\[
	A_{K^*}\mathcal V\subseteq\mathcal V,
	\qquad
	\bar A_{K^*}\mathcal V\subseteq\mathcal V.
	\]
	We next show that the smaller critical subspace contained in
	\(\mathcal V\) is also invariant.	
	For each \(k\), let
	\[
	\mathcal W_k:=\operatorname{Im}\Pi_k^{\rm c}
	\]
	denote the critical eigenspace associated with the first \(r\)
	eigenvalues of \(P_k\). Since
	\[
	\lim_{k\to\infty}
	\|\Pi_k^{\rm c}-\Pi^{\rm c}\|=0,
	\]
	the matrices \(X_k\) can be selected such that, along the subsequence
	\(\{k_j\}_{j\ge1}\),
	\[
	\lim_{j\to\infty}X_{k_j}=X,
	\]
	where the columns of \(X\) form an orthonormal basis of
	$
	\mathcal W:=\operatorname{Im}\Pi^{\rm c}.
	$
	By Definition \ref{regu}, there exist constants
	\(0<c\leq C<\infty\) such that, for all sufficiently large \(j\),
	\[
	c\tau_{k_j}I_r
	\preceq
	X_{k_j}^{\top}P_{k_j}X_{k_j}
	\preceq
	C\tau_{k_j}I_r.
	\]
	Hence,
	\[
	cI_r
	\preceq
	\frac{1}{\tau_{k_j}}
	X_{k_j}^{\top}P_{k_j}X_{k_j}
	\preceq
	CI_r.
	\]
	The subsequence \(\{k_j\}_{j\ge1}\) can therefore be chosen so that
	\[
	\lim_{j\to\infty}
	\frac{1}{\tau_{k_j}}
	X_{k_j}^{\top}P_{k_j}X_{k_j}
	=
	Z_c
	\]
	for some \(Z_c\succ0\).
	Multiplying the regularized fixed-point equation from the left and
	right by \(X_{k_j}^{\top}\) and \(X_{k_j}\), respectively, and
	dividing by \(\tau_{k_j}\), gives
	\[
	\begin{aligned}
		\gamma_{k_j}
		\frac{X_{k_j}^{\top}P_{k_j}X_{k_j}}{\tau_{k_j}}
		={}&
		(1-\tau_{k_j})
		\left(
		\frac{
			X_{k_j}^{\top}A_{k_j}^{\top}
			P_{k_j}A_{k_j}X_{k_j}}
		{\tau_{k_j}}
		\right.\left.
		+
		\sigma^2
		\frac{
			X_{k_j}^{\top}\bar A_{k_j}^{\top}
			P_{k_j}\bar A_{k_j}X_{k_j}}
		{\tau_{k_j}}
		\right)
		+
		\frac{1}{n}I_r.
	\end{aligned}
	\]
	The left-hand side is bounded. Since the two matrix terms inside the
	parentheses are positive semidefinite, both sequences
	\[
	\left\{
	\frac{
		X_{k_j}^{\top}A_{k_j}^{\top}
		P_{k_j}A_{k_j}X_{k_j}}
	{\tau_{k_j}}
	\right\}_{j\ge1}
	\]
	and
	\[
	\left\{
	\frac{
		X_{k_j}^{\top}\bar A_{k_j}^{\top}
		P_{k_j}\bar A_{k_j}X_{k_j}}
	{\tau_{k_j}}
	\right\}_{j\ge1}
	\]
	are bounded.
	
	Let \(X_{k_j}^{\perp}\) be a matrix whose columns form an orthonormal
	basis of \(\mathcal W_{k_j}^{\perp}\), chosen from the remaining
	eigenvectors of \(P_{k_j}\). Since \(X_{k_j}\) and \(X_{k_j}^{\perp}\) correspond
	to complementary spectral subspaces of \(P_{k_j}\), we have
	\[
	(X_{k_j}^{\perp})^{\top}
	P_{k_j}X_{k_j}^{\perp}
	\succeq
	\lambda_{r+1,k_j}I.
	\]
	Moreover,
	\[
	\lim_{j\to\infty}
	\frac{\lambda_{r+1,k_j}}{\tau_{k_j}}
	=
	\infty.
	\]
	Indeed, when \(r<\ell\), this follows from the intermediate-layer
	condition in Definition~\ref{regu}. When \(r=\ell\), it follows from
	\[
	\liminf_{k\to\infty}\lambda_{\ell+1,k}>0
	\]
	together with
	\[
	\lim_{j\to\infty}\tau_{k_j}=0.
	\]
	Decompose \(A_{k_j}X_{k_j}\) as
	\[
	A_{k_j}X_{k_j}
	=
	X_{k_j}F_j+X_{k_j}^{\perp}G_j,
	\]
	where
	\[
	F_j:=X_{k_j}^{\top}A_{k_j}X_{k_j},
	\qquad
	G_j:=(X_{k_j}^{\perp})^{\top}A_{k_j}X_{k_j}.
	\]
	Using the orthogonality of the spectral subspaces of \(P_{k_j}\),
	we obtain
	\[
	\begin{aligned}
		X_{k_j}^{\top}A_{k_j}^{\top}
		P_{k_j}A_{k_j}X_{k_j}
		={}&
		F_j^{\top}
		X_{k_j}^{\top}P_{k_j}X_{k_j}
		F_j\\
		&+
		G_j^{\top}
		(X_{k_j}^{\perp})^{\top}
		P_{k_j}X_{k_j}^{\perp}
		G_j\\
		\succeq{}&
		\lambda_{r+1,k_j}G_j^{\top}G_j.
	\end{aligned}
	\]
	Since
	\[
	\left\{
	\frac{
		X_{k_j}^{\top}A_{k_j}^{\top}
		P_{k_j}A_{k_j}X_{k_j}}
	{\tau_{k_j}}
	\right\}_{j\ge1}
	\]
	is bounded and
	\[
	\lim_{j\to\infty}
	\frac{\lambda_{r+1,k_j}}{\tau_{k_j}}
	=
	\infty,
	\]
	it follows that
	\[
	\lim_{j\to\infty}G_j=0.
	\]
	Moreover,
	\[
	\lim_{j\to\infty}F_j
	=
	X^{\top}A_{K^*}X.
	\]
	Taking limits in
	\[
	A_{k_j}X_{k_j}
	=
	X_{k_j}F_j+X_{k_j}^{\perp}G_j
	\]
	gives
	\[
	A_{K^*}X
	=
	X\bigl(X^{\top}A_{K^*}X\bigr).
	\]
	The right-hand side belongs to \(\operatorname{Im}X\). Therefore,
	for every \(z\in\mathbb R^r\),
	\[
	A_{K^*}Xz
	=
	X\bigl(X^{\top}A_{K^*}Xz\bigr)
	\in
	\operatorname{Im}X.
	\]
	Since every vector in \(\operatorname{Im}X\) can be written as
	\(Xz\) for some \(z\in\mathbb R^r\), we conclude that
	\[
	A_{K^*}\mathcal W\subseteq\mathcal W.
	\]
	Applying the same argument to
	\(\bar A_{k_j}X_{k_j}\) gives
	\[
	\bar A_{K^*}\mathcal W\subseteq\mathcal W.
	\]
	
	\noindent\textbf{Step 2: Construction of the closed-loop block
		decomposition.}
	
	The preceding proposition shows that the whole degenerate subspace
	\[
	\mathcal V:=\ker P^*
	\]
	is invariant under both limiting closed-loop matrices:
	\[
	A_{K^*}\mathcal V\subseteq\mathcal V,
	\qquad
	\bar A_{K^*}\mathcal V\subseteq\mathcal V.
	\]
	Step~1 further shows that the critical subspace
	\[
	\mathcal W:=\operatorname{Im}\Pi^{\rm c}
	\]
	is invariant under both matrices:
	\[
	A_{K^*}\mathcal W\subseteq\mathcal W,
	\qquad
	\bar A_{K^*}\mathcal W\subseteq\mathcal W.
	\]
	We next establish the relation between these two invariant subspaces.
	By  Definition~\ref{regu}, the first
	\(\ell\) eigenvalues of \(P_{k_j}\) converge to zero, whereas the
	remaining eigenvalues are uniformly bounded away from zero. Together
	with
	\[
	\lim_{j\to\infty}P_{k_j}=P^*
	\]
	and
	\[
	\lim_{k\to\infty}
	\|\Pi_k^{\rm v}-\Pi^{\rm v}\|=0,
	\]
	this gives
	\[
	\operatorname{Im}\Pi^{\rm v}
	=
	\ker P^*
	=
	\mathcal V.
	\]
	
	For every \(k\), the critical spectral subspace is contained in the
	vanishing spectral subspace:
	\[
	\operatorname{Im}\Pi_k^{\rm c}
	\subseteq
	\operatorname{Im}\Pi_k^{\rm v}.
	\]
	Equivalently,
	\[
	\Pi_k^{\rm v}\Pi_k^{\rm c}=\Pi_k^{\rm c}.
	\]
	Taking limits in this identity gives
	\[
	\Pi^{\rm v}\Pi^{\rm c}=\Pi^{\rm c},
	\]
	and therefore
	\[
	\mathcal W
	=
	\operatorname{Im}\Pi^{\rm c}
	\subseteq
	\operatorname{Im}\Pi^{\rm v}
	=
	\mathcal V.
	\]
	Thus, the state space contains the nested invariant subspaces
	\[
	\mathcal W\subseteq\mathcal V\subsetneq\mathbb R^n.
	\]
	Choose an orthogonal matrix
	\[
	T=[X\ Y\ Q],
	\]
	where the columns of \(X\) form an orthonormal basis of
	\(\mathcal W\), the columns of \([X\ Y]\) form an orthonormal basis
	of \(\mathcal V\), and the columns of \(Q\) form an orthonormal basis
	of \(\mathcal V^\perp\). Hence,
	\[
	\operatorname{Im}X=\mathcal W,
	\qquad
	\operatorname{Im}[X\ Y]=\mathcal V.
	\]
	
	We now show how the nested invariance relations yield the simultaneous
	block upper triangular forms. Since \(\mathcal W\) is invariant under
	\(A_{K^*}\), the columns of \(A_{K^*}X\) belong to
	\(\operatorname{Im}X\). Therefore,
	\[
	Y^\top A_{K^*}X=0,
	\qquad
	Q^\top A_{K^*}X=0.
	\]
	Since \(\mathcal V\) is invariant under \(A_{K^*}\), the columns of
	\(A_{K^*}Y\) belong to \(\mathcal V\), which gives
	\[
	Q^\top A_{K^*}Y=0.
	\]
	Consequently,
	\[
	\begin{aligned}
		T^\top A_{K^*}T
		&=
		\begin{pmatrix}
			X^\top A_{K^*}X
			&
			X^\top A_{K^*}Y
			&
			X^\top A_{K^*}Q
			\\
			Y^\top A_{K^*}X
			&
			Y^\top A_{K^*}Y
			&
			Y^\top A_{K^*}Q
			\\
			Q^\top A_{K^*}X
			&
			Q^\top A_{K^*}Y
			&
			Q^\top A_{K^*}Q
		\end{pmatrix}
		\\
		&=
		\begin{pmatrix}
			A_c & A_{cb} & A_{cd}\\
			0 & A_b & A_{bd}\\
			0 & 0 & A_d
		\end{pmatrix},
	\end{aligned}
	\]
	where
	\[
	A_c:=X^\top A_{K^*}X,
	\qquad
	A_b:=Y^\top A_{K^*}Y,
	\qquad
	A_d:=Q^\top A_{K^*}Q.
	\]
	
	Applying the same invariance argument to \(\bar A_{K^*}\) gives
	\[
	Y^\top\bar A_{K^*}X=0,
	\qquad
	Q^\top\bar A_{K^*}X=0,
	\qquad
	Q^\top\bar A_{K^*}Y=0,
	\]
	and hence
	\[
	T^\top\bar A_{K^*}T
	=
	\begin{pmatrix}
		\bar A_c & \bar A_{cb} & \bar A_{cd}\\
		0 & \bar A_b & \bar A_{bd}\\
		0 & 0 & \bar A_d
	\end{pmatrix},
	\]
	where
	\[
	\bar A_c:=X^\top\bar A_{K^*}X,
	\qquad
	\bar A_b:=Y^\top\bar A_{K^*}Y,
	\qquad
	\bar A_d:=Q^\top\bar A_{K^*}Q.
	\]
	This proves the simultaneous block upper triangular decomposition
	stated in the theorem.
	
	\noindent\textbf{Step 3: Strict mean-square spectral separation.}
	
	We first establish an upper estimate for the mean-square spectral
	radius of the critical block. Define
	\[
	Z_j
	:=
	\frac{1}{\tau_{k_j}}
	X_{k_j}^{\top}P_{k_j}X_{k_j}.
	\]
	As established in Step~1,
	\[
	\lim_{j\to\infty}Z_j=Z_c,
	\qquad
	Z_c\succ0.
	\]
	Recall also that
	\[
	F_j:=X_{k_j}^{\top}A_{k_j}X_{k_j},
	\qquad
	\lim_{j\to\infty}F_j=A_c.
	\]
	From the spectral decomposition used in Step~1, we have
	\[
	X_{k_j}^{\top}A_{k_j}^{\top}
	P_{k_j}A_{k_j}X_{k_j}
	\succeq
	F_j^{\top}
	X_{k_j}^{\top}P_{k_j}X_{k_j}
	F_j.
	\]
	Dividing this inequality by \(\tau_{k_j}\) gives
	\[
	\frac{
		X_{k_j}^{\top}A_{k_j}^{\top}
		P_{k_j}A_{k_j}X_{k_j}}
	{\tau_{k_j}}
	\succeq
	F_j^{\top}Z_jF_j.
	\]
	
	Similarly, define
	\[
	\bar F_j
	:=
	X_{k_j}^{\top}\bar A_{k_j}X_{k_j}.
	\]
	Then,
	\[
	\lim_{j\to\infty}\bar F_j=\bar A_c,
	\]
	and
	\[
	\frac{
		X_{k_j}^{\top}\bar A_{k_j}^{\top}
		P_{k_j}\bar A_{k_j}X_{k_j}}
	{\tau_{k_j}}
	\succeq
	\bar F_j^{\top}Z_j\bar F_j.
	\]
	The two matrix sequences on the left-hand sides of the preceding
	inequalities are bounded, as shown in Step~1. Passing to a further
	subsequence without changing the notation, there exist symmetric
	matrices \(M_c\) and \(\bar M_c\) such that
	\[
	\lim_{j\to\infty}
	\frac{
		X_{k_j}^{\top}A_{k_j}^{\top}
		P_{k_j}A_{k_j}X_{k_j}}
	{\tau_{k_j}}
	=
	M_c
	\]
	and
	\[
	\lim_{j\to\infty}
	\frac{
		X_{k_j}^{\top}\bar A_{k_j}^{\top}
		P_{k_j}\bar A_{k_j}X_{k_j}}
	{\tau_{k_j}}
	=
	\bar M_c.
	\]
	Taking limits in the preceding inequalities yields
	\[
	M_c\succeq A_c^{\top}Z_cA_c,
	\qquad
	\bar M_c\succeq\bar A_c^{\top}Z_c\bar A_c.
	\]
	On the other hand, projecting the regularized fixed-point equation
	onto the critical subspace and dividing by \(\tau_{k_j}\) gives
	\[
	\begin{aligned}
		\gamma_{k_j}Z_j
		={}&
		(1-\tau_{k_j})
		\left(
		\frac{
			X_{k_j}^{\top}A_{k_j}^{\top}
			P_{k_j}A_{k_j}X_{k_j}}
		{\tau_{k_j}}
		\right.\left.
		+
		\sigma^2
		\frac{
			X_{k_j}^{\top}\bar A_{k_j}^{\top}
			P_{k_j}\bar A_{k_j}X_{k_j}}
		{\tau_{k_j}}
		\right)
		+
		\frac{1}{n}I_r.
	\end{aligned}
	\]
	Taking limits in this equation gives
	\[
	\gamma^*Z_c
	=
	M_c+\sigma^2\bar M_c+\frac{1}{n}I_r.
	\]
	Therefore,
	\[
	\gamma^*Z_c
	\succeq
	A_c^{\top}Z_cA_c
	+
	\sigma^2\bar A_c^{\top}Z_c\bar A_c
	+
	\frac{1}{n}I_r.
	\]
	In particular,
	\[
	A_c^{\top}Z_cA_c
	+
	\sigma^2\bar A_c^{\top}Z_c\bar A_c
	\prec
	\gamma^*Z_c.
	\]
	Since \(Z_c\succ0\), the preceding inequality also implies that
	\(\gamma^*>0\). By the mean-square Lyapunov criterion given in
	Lemma~2 of~\cite{feng2008optimal}, we obtain
	\begin{equation}\label{R5}
		\varrho_{\rm ms}(A_c,\bar A_c)<\gamma^*.
	\end{equation}
	It remains to identify \(\gamma^*\) with the mean-square spectral
	radius of the dominant block. Recall that
	\[
	\mathcal V=\ker P^*
	\]
	and that the columns of \(Q\) form an orthonormal basis of
	\(\mathcal V^\perp\). Define
	\[
	P_d:=Q^{\top}P^*Q.
	\]
	We first show that
	\[
	P_d\succ0.
	\]
	Suppose, to the contrary, that there exists a nonzero vector
	\(\xi\) such that
	\[
	\xi^{\top}P_d\xi=0.
	\]
	It follows that
	\[
	(Q\xi)^{\top}P^*(Q\xi)=0.
	\]
	Since \(P^*\succeq0\), we have
	\[
	P^*Q\xi=0,
	\]
	and therefore
	\[
	Q\xi\in\ker P^*=\mathcal V.
	\]
	By the construction of \(Q\), we also have
	\[
	Q\xi\in\mathcal V^\perp.
	\]
	Consequently,
	\[
	Q\xi\in\mathcal V\cap\mathcal V^\perp=\{0\}.
	\]
	This contradicts \(\xi\neq0\) and the orthonormality of the columns
	of \(Q\). Hence,
	\[
	P_d\succ0.
	\]
	Multiplying the limiting fixed-point equation
	\[
	\gamma^*P^*
	=
	A_{K^*}^{\top}P^*A_{K^*}
	+
	\sigma^2\bar A_{K^*}^{\top}P^*\bar A_{K^*}
	\]
	from the left and right by \(Q^{\top}\) and \(Q\), respectively,
	gives
	\[
	\begin{aligned}
		\gamma^*P_d
		={}&
		Q^{\top}A_{K^*}^{\top}P^*A_{K^*}Q\\
		&+
		\sigma^2
		Q^{\top}\bar A_{K^*}^{\top}
		P^*\bar A_{K^*}Q.
	\end{aligned}
	\]
	By the block upper triangular decomposition established in Step~2,
	\[
	A_{K^*}Q
	=
	XA_{cd}+YA_{bd}+QA_d.
	\]
	Moreover,
	\[
	\operatorname{Im}[X\ Y]
	=
	\mathcal V
	=
	\ker P^*,
	\]
	and hence
	\[
	P^*X=0,
	\qquad
	P^*Y=0.
	\]
	It follows that
	\[
	Q^{\top}A_{K^*}^{\top}P^*A_{K^*}Q
	=
	A_d^{\top}P_dA_d.
	\]
	Similarly,
	\[
	\bar A_{K^*}Q
	=
	X\bar A_{cd}+Y\bar A_{bd}+Q\bar A_d,
	\]
	and therefore
	\[
	Q^{\top}\bar A_{K^*}^{\top}
	P^*\bar A_{K^*}Q
	=
	\bar A_d^{\top}P_d\bar A_d.
	\]
	Consequently,
	\[
	\gamma^*P_d
	=
	A_d^{\top}P_dA_d
	+
	\sigma^2\bar A_d^{\top}P_d\bar A_d.
	\]
	Define the positive linear operator
	\[
	\mathcal L_d(P)
	:=
	A_d^{\top}PA_d
	+
	\sigma^2\bar A_d^{\top}P\bar A_d,
	\qquad
	P\in\mathbb S^q.
	\]
	Then \(P_d\succ0\) is a positive definite eigenmatrix of
	\(\mathcal L_d\) associated with the eigenvalue \(\gamma^*\).
	By Lemma~\ref{lem:positive-eigenmatrix-spr},
	\[
	\operatorname{spr}(\mathcal L_d)=\gamma^*.
	\]
	Hence,
	\begin{equation}\label{R6}
		\begin{aligned}
			\operatorname{spr}(\mathcal L_d)
			&=
			\operatorname{spr}\!\left(
			A_d\otimes A_d
			+
			\sigma^2\bar A_d\otimes\bar A_d
			\right)\\
			&=
			\varrho_{\rm ms}(A_d,\bar A_d)
			=
			\gamma^*.
		\end{aligned}
	\end{equation}
	Combining (\ref{R5}) and (\ref{R6}) gives
	\[
	\varrho_{\rm ms}(A_c,\bar A_c)
	<
	\gamma^*
	=
	\varrho_{\rm ms}(A_d,\bar A_d),
	\]
	which proves the strict mean-square spectral separation
	in~\eqref{eq:strict-ms-separation}. This completes the proof.
\end{proof}

\end{document}